\documentclass[reqno,11pt]{amsart}
\usepackage{amsmath, latexsym, amsfonts, amssymb, amsthm, amscd}
\usepackage{graphics,epsf,psfrag}
\usepackage{graphics,epsf,psfrag}
 \usepackage[utf8]{inputenc}
\usepackage{stmaryrd}
\usepackage{mathabx}
\usepackage{hyperref}
\usepackage{xcolor}

\renewcommand{\bar}{\overline}

\newcommand{\lint}{\llbracket}
\newcommand{\rint}{\rrbracket}

\numberwithin{equation}{section}

\newtheorem{theorema}{Theorem}

\newtheorem{theorem}{Theorem}[section]
\newtheorem{lemma}[theorem]{Lemma}
\newtheorem{proposition}[theorem]{Proposition}

\newtheorem{rem}[theorem]{Remark}

\newcommand{\dd}{\mathrm{d}}
\newcommand{\ind}{\mathbf{1}}

\newcommand{\Supp}{\mathrm{Supp}}

\renewcommand{\tilde}{\widetilde}
\renewcommand{\hat}{\widehat}
\newcommand{\cc}{\complement}

\newcommand{\cG}{{\ensuremath{\mathcal G}} }

\newcommand{\cX}{{\ensuremath{\mathcal X}} }
\newcommand{\cA}{{\ensuremath{\mathcal A}} }
\newcommand{\cB}{{\ensuremath{\mathcal B}} }
\newcommand{\cF}{{\ensuremath{\mathcal F}} }

\newcommand{\cP}{{\ensuremath{\mathcal P}} }

\newcommand{\cC}{{\ensuremath{\mathcal C}} }
\newcommand{\cN}{{\ensuremath{\mathcal N}} }

\newcommand{\cT}{{\ensuremath{\mathcal T}} }
\newcommand{\cD}{{\ensuremath{\mathcal D}} }

\newcommand{\cZ}{{\ensuremath{\mathcal Z}} }
\newcommand{\cI}{{\ensuremath{\mathcal I}} }

\newcommand{\bP}{{\ensuremath{\mathbf P}} }

\newcommand{\bE}{{\ensuremath{\mathbf E}} }

\newcommand{\bx}{{\ensuremath{\mathbf x}} }
\newcommand{\by}{{\ensuremath{\mathbf y}} }

\DeclareMathSymbol{\leqslant}{\mathalpha}{AMSa}{"36} 
\DeclareMathSymbol{\geqslant}{\mathalpha}{AMSa}{"3E} 
\DeclareMathSymbol{\eset}{\mathalpha}{AMSb}{"3F}     

\DeclareMathOperator*{\inter}{\bigcap}       
\newcommand{\maxtwo}[2]{\max_{\substack{#1 \\ #2}}} 
\newcommand{\suptwo}[2]{\sup_{\substack{#1 \\ #2}}} 
\newcommand{\intertwo}[2]{\inter_{\substack{#1 \\ #2}}} 
\newcommand{\limtwo}[2]{\lim_{\substack{#1 \\ #2}}}     

\newcommand{\bbC}{{\ensuremath{\mathbb C}} }

\newcommand{\bbE}{{\ensuremath{\mathbb E}} }

\newcommand{\bbP}{{\ensuremath{\mathbb P}} }

\newcommand{\bbR}{{\ensuremath{\mathbb R}} }

\newcommand{\gep}{\varepsilon}       

\newcommand{\go}{\omega}

\newcommand{\gl}{\lambda}

\makeatletter
\def\captionfont@{\footnotesize}
\def\captionheadfont@{\scshape}

\long\def\@makecaption#1#2{%
  \vspace{2mm}
  \setbox\@tempboxa\vbox{\color@setgroup
    \advance\hsize-6pc\noindent
    \captionfont@\captionheadfont@#1\@xp\@ifnotempty\@xp
        {\@cdr#2\@nil}{.\captionfont@\upshape\enspace#2}%
    \unskip\kern-6pc\par
    \global\setbox\@ne\lastbox\color@endgroup}%
  \ifhbox\@ne 
    \setbox\@ne\hbox{\unhbox\@ne\unskip\unskip\unpenalty\unkern}%
  \fi
  \ifdim\wd\@tempboxa=\z@ 
    \setbox\@ne\hbox to\columnwidth{\hss\kern-6pc\box\@ne\hss}%
  \else 
    \setbox\@ne\vbox{\unvbox\@tempboxa\parskip\z@skip
        \noindent\unhbox\@ne\advance\hsize-6pc\par}%
\fi
  \ifnum\@tempcnta<64 
    \addvspace\abovecaptionskip
    \moveright 3pc\box\@ne
  \else 
    \moveright 3pc\box\@ne
    \nobreak
    \vskip\belowcaptionskip
  \fi
\relax
}
\makeatother
\def\writefig#1 #2 #3 {\rlap{\kern #1 truecm
\raise #2 truecm \hbox{#3}}}

\title{Convergence for Complex Gaussian Multiplicative Chaos on phase boundaries}

\author{Hubert Lacoin}
\address{
  IMPA, Institudo de Matem\'atica Pura e Aplicada, Estrada Dona Castorina 110
Rio de Janeiro, CEP-22460-320, Brasil. 
}

\begin{document}

 \begin{abstract}
 The complex Gaussian Multiplicative Chaos (or complex GMC) is informally defined as  a random measure $e^{\gamma X} \dd x$ where $X$ is a log correlated Gaussian field on $\bbR^d$  and $\gamma=\alpha+i\beta$ is a complex parameter. The correlation function of $X$ is of the form 
 $$ K(x,y)= \log \frac{1}{|x-y|}+ L(x,y),$$
where  $L$ is a continuous function. In the present paper, we consider the cases $\gamma\in \cP_{\mathrm{I/II}}$ and $\gamma\in  \cP'_{\mathrm{II/III}}$
 where 
 $$  \cP_{\mathrm{I/II}}:= \{ \alpha+i \beta \ : \alpha,\beta \in \bbR \ ; |\alpha|>|\beta| \ ; \ |\alpha|+|\beta|=\sqrt{2d}  \},
 $$ 
 and 
  $$  \cP'_{\mathrm{II/III}}:= \{ \alpha+i \beta \ : \alpha,\beta \in \bbR \ ; \ |\alpha|= \sqrt{d/2} \ ; \ |\beta|>\sqrt{2d}  \},
 $$
 We prove that if $X$ is replaced by an approximation $X_\gep$ obtained via mollification, then  $e^{\gamma X_\gep} \dd x$, when properly rescaled, converges when $\gep\to 0$. The limit does not depend on the mollification  kernel. When $\gamma\in \cP_{\mathrm{I/II}}$, the convergence holds in probability and in $L^p$ for some value of $p\in [1,\sqrt{2d}/\alpha)$. When $\gamma\in \cP'_{\mathrm{II/III}}$
 the convergence holds only in law. In this latter case, the  limit can be described
 a complex Gaussian white noise with a random intensity given by a critical real GMC. The regions $\cP_{\mathrm{I/II}}$ and $\cP'_{\mathrm{II/III}}$
 correspond to phase boundary between the three different regions of the complex GMC phase diagram. These results complete previous results obtained for the GMC in phase I \cite{lacoin2020} and III \cite{MR4413209} and only leave as an open problem the question of convergence in phase II.
   \\[10pt]
 \end{abstract}

\maketitle

  \tableofcontents
 
 \section{Introduction}

\noindent Let $K: \bbR^d\times \bbR^d \to (-\infty,\infty]$  be a \textit{positive definite} kernel on $\bbR^d$ ($d\ge 1$ is fixed) which admits a decomposition of the form
 \begin{equation}\label{fourme}
  K(x,y)= \log \frac{1}{|x-y|}+L(x,y),
 \end{equation}
(with the convention $\log(1/0)=\infty$) where $L$ is a continuous function on $\bbR^{2d}$ . 
A kernel $K$ is  positive definite if for $\rho\in C_c(\bbR^d)$ ($\rho$ continuous with compact support)
\begin{equation}\label{ladefpos}
 \int_{\bbR^{2d} }  K(x,y) \rho(x)\rho(y)\dd x \dd y\ge 0.
\end{equation}
Given a centered Gaussian field $X$ with covariance $K$ and $\gamma=\alpha+i\beta$  a complex number ($\alpha,\beta\in \bbR$)  the \textit{complex Gaussian Multiplicative Chaos} (or complex GMC) with parameter $\gamma$ is the random distribution formally defined by the expression
 \begin{equation}\label{defGMC}
 \ M^{\gamma}(\dd x)= e^{\gamma X(x)} \dd x. 
 \end{equation}
A  difficulty comes up when  trying to give an interpretation to the r.h.s.\ of  \eqref{defGMC}.
A field $X$ with a covariance given by \eqref{fourme} can be defined only as a random distribution. For a fixed $x\in \bbR^{d}$ it is not possible to make sense of $X(x)$.

\medskip

The problem of providing a mathematical construction of  $M^{\gamma}$ that gives a meaning to \eqref{defGMC} was first considered by Kahane in \cite{zbMATH03960673} in the case where $\gamma\in \bbR$, we refer to 
\cite{powell2020critical,RVreview} for reviews on the subject.
The case of $\gamma \in \bbC$ was considered only more recently, see for instance  \cite{junnila2019regularity,  junnila2019, junnila2018,  MR4413209,  lacoin2020, LRV19, LRV15} and references therein.
The standard procedure to define  the GMC is to  use a sequence of  approximation of the field $X$, consider the exponential of the approximation   
and then pass to the limit. Mostly  two kinds of approximation of $X$ have been considered in the literature: 
\begin{itemize}
 \item [(A)] A mollification of the field, $X_{\gep}$,  via convolution with a smooth kernel on scale $\gep$,
 \item [(B)] A martingale approximation, $X_t$, via an integral decomposition of the kernel $K$.
\end{itemize}
In the present paper we present convergence results for the random distribution  $e^{\gamma X_{\gep}(x)} \dd x$ and $e^{\gamma X_{t}(x)} \dd x$ and in a certain range of parameter $\gamma$. Before describing our results in more details and provide some motivation, we first rigorously introduce the setup.

\subsection{The mollification of a $\log$-correlated field}\label{molly}

\subsubsection*{Log-correlated fields  defined as distributions}

Since $K$ is infinite on the diagonal, it is not possible to define  a Gaussian field indexed by $\bbR^d$ with covariance function $K$.
We consider instead a process indexed by test functions. 
We define $\hat K$, a bilinear form on  $C_c(\bbR^d)$ (the set of compactly supported continuous functions) by   
\begin{equation}\label{hatK}
 \hat K(\rho,\rho')=
\int_{\bbR^{2d}}  K (x,y)\rho(x)\rho'(y)\dd x \dd y.
\end{equation}
Since $\hat K$ is positive definite (in the usual sense: for any $(\rho_i)^k_{i=1}$, the matrix $\hat K(\rho_i,\rho_j)_{i,j=1}^k$ is positive definite),  it is possible to define  $X= \langle X, \rho \rangle_{\rho\in C_c(\bbR^d)}$ a centered Gaussian process indexed by $C_c(\bbR^d)$ with covariance kernel given by $\hat K$.

\begin{rem}
There exists a modification of the process $X$ which take value in a distribution space (more specifically, such that $X$ takes values in the Sobolev space $H^{s}_{\mathrm{loc}}(\bbR^d)$ for every $s<0$ (see the definition \eqref{locsob} in the appendix). For this reason (and although we will not use this fact) we  refer to $X$ as a random distribution.
\end{rem}

\subsubsection*{Approximation of $X$ via mollification}

The random distribution $X$ can be approximated by a sequence of  functional fields - processes indexed by $\bbR^d$ - 
by the mean of mollification  by a smooth kernel.
 Consider $\theta$ a nonnegative function in $C_c^{\infty}(\bbR^d)$ (the set of infinitely differentiable functions in $C_c(\bbR^d)$)  whose compact support is included in  $B(0,1)$ (for the remainder 
 of the paper $B(x,r)$ denotes the closed 
 Euclidean ball of center $x$ and radius $r$) 
 and which satisfies  $\int_{B(0,1)} \theta(x) \dd x=1.$
We define for $\gep>0$,
$\theta_{\gep}:=\gep^{-d} \theta( \gep^{-1}\cdot)$
and consider $(X_{\gep}(x))_{x\in \bbR^d}$, the mollified version of $X$, that is
\begin{equation}\label{convolulu}
 X_{\gep}(x):= \langle X, \theta_{\gep}(x-\cdot) \rangle
\end{equation}
From \eqref{hatK}, the field $X_{\gep}(\cdot)$ has covariance 
\begin{equation}\label{labig}
K_{\gep}(x,y):=\bbE[X_{\gep}(x)X_{\gep}(y)]=
\int_{\bbR^{2d}} \theta_{\gep}(x-z_1) \theta_{\gep}(y-z_2)
K(z_1,z_2)\dd z_1 \dd z_2.
\end{equation}
We set $K_{\gep}(x):=K_{\gep}(x,x)$ and extend this convention to other functions of two variables in $\bbR^{d}$.
Since $K_{\gep}$ is infinitely differentiable - thus in particular is H\"older continuous - by Kolmogorov's Continuity Theorem (e.g.\  \cite[Theorem 2.9]{legallSto})  there exists a continuous modification of $X_{\gep}(\cdot)$. In the remainder of the paper, we always consider the continuous modification of a process when it exists. This ensures that integrals such as the one appearing in \eqref{forboundedfunctions} are well defined.
We  define the distribution     $M^{\gamma}_{\gep}$ by setting for $f\in C_c(\bbR^d)$
 \begin{equation}\label{forboundedfunctions}
    M^{\gamma}_{\gep}(f):=\int_{\bbR^d} f(x) e^{\gamma X_{\gep}(x)- \frac{\gamma^2}{2} K_\gep(x)}\dd x.
 \end{equation}
The question of interest in the present paper is the convergence of   $M^{\gamma}_{\gep}$ when $\gep\to0$.

 \begin{rem}
  Note that, even if we have chosen to omit this dependence in the notation, $X_{\gep}$ and $M^{\gamma}_{\gep}$ both depend on the particular convolution kernel $\theta$.
An important feature of  our results is that  the  limits obtained for $M^{\gamma}_{\gep}(f)$ do not depend on $\theta$.
 \end{rem}

 \subsection{Star-scale invariance and our assumption on $K$}

On top of assuming that $K$ admits a decomposition like \eqref{fourme}, we also assume that it has an \textit{almost star-scale invariant} part (see the definition \eqref{iladeuxstar}-\eqref{3star}). This assumption might seem at first quite restrictive, but it has been shown in 
\cite{junnila2019} that it is \textit{locally} satisfied as soon as the function  $L$ in \eqref{fourme} is \textit{sufficiently regular}. 
In  Appendix \ref{beyondstar} we provides details concerning the regularity assumption for $L$ and explain how to extend the validity of our results to all sufficiently regular $\log$-correlated kernels using the ideas in \cite{junnila2019}.

\medskip

Following a terminology introduced in \cite{junnila2019}, we say that a the kernel $K$ defined on $\bbR^d$ is \textit{almost star-scale invariant}
if it can be written in the form  
\begin{equation}\label{iladeuxstar}
 \forall x,y \in \bbR^d,\  K(x,y)=\int^{\infty}_{0} (1-\eta_1 e^{-\eta_2 t}) \kappa(e^{t}(x-y))\dd t,
\end{equation}
 where  $\eta_1\in[0,1]$ and $\eta_2>0$ are constants and the function $\kappa\in C^{\infty}_c(\bbR^d)$ is radial, nonnegative and definite positive. More precisely we assume the following:
\begin{itemize}
 \item [(i)] $\kappa\in C_c^{\infty}(\bbR^d)$ and there exists $\tilde \kappa : \ \bbR^+\to [0,\infty)$ such that $\kappa(x):=\tilde \kappa(|x|)$,
 \item [(ii)] $\tilde \kappa(0)=1$ and $\tilde\kappa(r)=0$ for $r\ge 1$,
 \item [(iii)]The mapping $(x,y)\mapsto \kappa(x-y)$ defines a positive definite kernel on $\bbR^d\times \bbR^d$.
\end{itemize}
We say furthermore that a kernel $K$   \textit{has an almost star-scale invariant part}, if 
\begin{equation}\label{3star}
 \forall x,y\in \bbR^{d}, \   K(x,y)=K_0(x,y)+ \bar K(x,y)
\end{equation}
where $\bar K(x,y)$ is an almost star-scale invariant kernel and $K_0$ is H\"older continuous on $\bbR^{2d}$ and positive  definite.

  \subsection{Phase transitions and phase diagrams for GMC}

Our main results concerns the asymptotic behavior of $M^{\gamma}_{\gep}$ in the specific range of $\gamma$ given in the abstract.
In order to properly motivate and present these results, it is necessary to introduce some context, and recall known facts about the phase diagram of the complex GMC.
 
 \subsubsection*{Phase transition at $|\alpha|=\sqrt{2d}$  for the real valued GMC}

 The question of the existence and identification of the limit  
 $$\lim_{\gep\to 0} M^{\alpha}_{\gep}(\cdot ),$$ 
 has first been considered in the work of Kahane in the eighties \cite{zbMATH03960673}, in the case when $\alpha\in \mathbb{R}$.
 The obtained limit in that case crucially depends on $\alpha$:
  when $|\alpha|<\sqrt{2d}$ - referred to as the \textit{subcritical case} - then  $M^{\alpha}_{\gep}$ converges in probability to a non-trivial limiting distribution
  (see for instance \cite[Theorem 1.1]{Natele} for a short and self contained proof, we refer to the introduction in \cite{Natele} for a detailed chronological account of results obtained for the subcritical case).

\medskip

 When $|\alpha|\ge \sqrt{2d}$, we have $\lim_{\gep\to 0} M^{\alpha}_{\gep}(f)=0$ and a rescaling procedure is needed in order to 
 obtain a non-trivial limit. The phenomenology is however different according to whether $|\alpha|=\sqrt{2d}$ ($\alpha$ critical) or $|\alpha|> \sqrt{2d}$ ($\alpha$ \textit{supercritical}).
 
 \medskip
 
 In the critical case ($\alpha=\pm \sqrt{2d}$), is has been shown, under fairly mild assumptions (see \cite{MR3262492,MR3215583,MR3613704,MR4396197} and Theorem \ref{critGMC} below) that $\sqrt{\log\left(1/\gep \right)}M_{\gep}^{\alpha}$
converges in probability to a non-trivial limit called the critical GMC.

 \medskip
 
 When $|\alpha|>\sqrt{2d}$, the results are less complete. So far the convergence has not been proved for $M^{\alpha}_{\gep}$ but only for an approximating martingale sequence $M^{\alpha}_t$ (see \eqref{defalfat}) in \cite{madaule2016}.  Besides this technical point, the most important differences with the case $|\alpha|\le \sqrt{2d}$ concerns the type of the convergence and the nature of limiting object. The convergence only holds  only \textit{in law}, and the limit is a \textit{purely atomic measure} (a measure supported by a countable set) see \cite[Corollary 2.3]{madaule2016}.

  \subsubsection*{Phase diagram for complex GMC}
  
  When $\gamma$ is allowed to assume complex value, the phase diagram becomes more intricate. The complex plane can be divided in three open regions with intersecting boundaries
  \begin{equation}\begin{split}
            \label{phasediagg}
 \cP_{\mathrm{I}}&:=\left\{\alpha+i \beta \ : \alpha^2+\beta^2<d\right\} \cup\left\{  \alpha+i \beta  \ : \ \alpha \in (\sqrt{d/2},\sqrt{2d}) \ ; \  |\alpha|+|\beta|<\sqrt{2d} \ \right\},\\
  \cP_{\mathrm{II}}&:= \left\{\alpha+i \beta \ : |\alpha|+|\beta|>\sqrt{2d} \ ; \  |\alpha|> \sqrt{d/2} \right\},\\
 \cP_{\mathrm{III}}&:= \left\{\alpha+i \beta \ : \alpha^2+\beta^2> d \ ; \ |\alpha| < \sqrt{d/2} \right\}.
 \end{split}
\end{equation}

\begin{figure}
\leavevmode
\epsfysize = 7.8 cm
\psfragscanon
 \psfrag{beta}{{\small $\beta$}}
 \psfrag{alpha}{{\small $\alpha$}}
 \psfrag{sqrd}{{\tiny $\sqrt{d}$}}
  \psfrag{sqrd/2}{{\tiny $\sqrt{d/2}$}}
       \psfrag{sqr2d}{{\tiny $\sqrt{2d}$}}
      \psfrag{msqr2d}{{\tiny $-\sqrt{2d}$}}
            \psfrag{psub}{$\mathcal P_{\mathrm{I}}$}   
            \psfrag{pII}{$\mathcal P_{\mathrm{II}}$}
             \psfrag{pIII}{$\mathcal P_{\mathrm{III}}$}
 \epsfbox{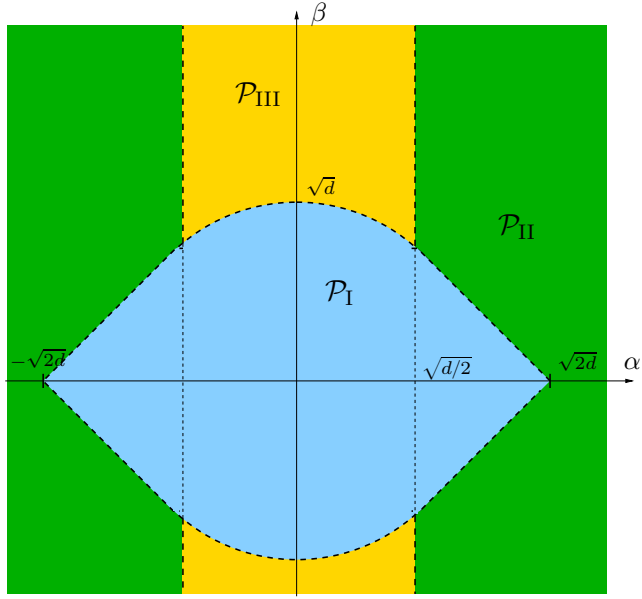}
 \caption{The phase diagram of the complex GMC in the complex plane. Each region correspond to a different limiting behavior for $M^{\gamma}_{\gep}$ in terms of renormalization factor, type of convergence and properties of the limit. In the present paper, we prove results concerning the asymptotic behavior on frontier of $\cP_{\mathrm I} \cup \cP_{\mathrm{III}}$ with $\cP_{\mathrm{II}}$. Results concerning convergence in $\cP_{\mathrm I} \cup \cP_{\mathrm{III}}$ where proved in \cite{junnila2019, lacoin2020} (for $\cP_{\mathrm I}$) and \cite{MR4413209} (for $\cP_{\mathrm{III}}\cup\cP_{\mathrm I/\mathrm{III}}$). The convergence in the region $\cP_{\mathrm{II}}$ remains a challenging conjecture. }
\end{figure}

This diagram first appeared in the context of complex Gaussian multiplicative cascade \cite{DES93}, and also serves to describe the behavior of other related models such as the complex REM \cite{KK14} or complex branching Brownian Motion \cite{HK15, HK18, MR3312433}.

\medskip
 
The region  $\cP_{\mathrm{I}}$ corresponds to the subcritical phase. For $\gamma\in  \cP_{\mathrm{I}}$ it has been proved \cite{junnila2019, lacoin2020} that $M^{\gamma}_{\gep}$ converges to a limit that does not depend on the mollifier $\theta$.

\medskip

The region  $\cP_{\mathrm{II}}$ corresponds to the supercritical phase, in which it is believed that $M^{\gamma}_{\gep}$ - after proper renormalization - converges \text{only in law} to a purely atomic random distribution. This conjecture is supported by rigorous results obtained in the case of Complex Branching Brownian Motion \cite{HK15, MR3312433}.

\medskip

Finally the region  $\cP_{\mathrm{III}}$ corresponds to yet another asymptotic behavior for $M^{\gamma}_{\gep}$. Like in $\cP_{\mathrm{II}}$,  $M^{\gamma}_{\gep}$ - properly rescaled - only  converges in law. The limit is given by a white noise whose intensity is random and is given by
the real valued GMC with parameter $2\alpha$ (which is subcritical, according to the definition of  $\cP_{\mathrm{III}}$).
A similar convergence result holds on the boundary between $\cP_{\mathrm{II}}$ and $\cP_{\mathrm{III}}$ that is  
$$ \cP_{\mathrm{II}/\mathrm{III}}:= \left\{\alpha+i \beta \ : \alpha^2+\beta^2= d \ ; \ |\alpha| < \sqrt{d/2} \right\}.$$
These convergence statements for $\gamma\in \cP_{\mathrm{III}} \cup \cP_{\mathrm{I}/\mathrm{III}}$ are proved in \cite{MR4413209}.

  \subsubsection*{The present contribution}
The aim of the present paper is to come closer to a completition of   the phase diagram by stating and proving convergence results for  $M_{\gep}^{\gamma}$ on the phase transition curves $\cP_{\mathrm{I}/\mathrm{II}}$ and $\cP_{\mathrm{I}/\mathrm{III}}$ as well as at the triple points $\cP_{\mathrm{I}/\mathrm{II}/\mathrm{III}}$.
In each case, the limit obtained does not depend on the regularization kernel $\theta$.
We leave as an open problem the challenging task of proving a convergence result in the frozen phase $\cP_{\mathrm{II}}$.

\section{Main results}

For simplicity of notation, we  consider, for the remainder of the paper and without loss of generality that $\gamma$ is in the upper-right quarterplane of $\bbC$, that is $\alpha,\beta\ge 0$.

\subsection{The boundary between phase I and II}

Our first result concerns the case when  $\gamma$ lies on the boundary between regions  I and II
 \begin{equation}\label{pundeux}
\cP_{\mathrm{I/II}}:= \{ \alpha+i \beta \ :  \alpha>\beta>0 \ ; \ \alpha+\beta=\sqrt{2d}  \}.
 \end{equation} 
Note that our definition of $\cP_{\mathrm{I/II}}$ excludes one point of the boundary which correspond to 
Critical Gaussian multiplicative chaos $\gamma=\sqrt{2d}$ (see Section \ref{criticalGMC} below).

\begin{theorem}\label{mainall}
 If $X$ is a centered Gaussian field whose covariance kernel $K$  has an almost star-scale invariant part, $\gamma \in \cP_{\mathrm{I/II}}$,
 and $f\in C_c(\bbR^d)$
then there exists a complex valued random variable $M^{\gamma}_{\infty}(f)$ such that for any choice of mollifier $\theta$ the following convergence holds in $L^p$ if $p\in\left[1,\sqrt{2d}/\alpha\right)$.
\begin{equation}\label{stableconv}
\lim_{\gep \to 0} M^{\gamma}_{\gep}(f)  
= M^{\gamma}_\infty(f).
 \end{equation}
 \end{theorem}

 The above result extends \cite[Theorem 2.2]{lacoin2020} which established convergence for $\gamma\in \mathcal P_{\mathrm I}$.
The method which we use to prove it however, completely differs from the one employed in \cite{lacoin2020}. In fact the method of proof that we employ in Section \ref{shortlapreuve}  balso provides an alternative and much shorter proof of \cite[Theorem 2.2]{lacoin2020}, with the additional benefit of establishing convergence in $L^p$ for an optimal range of $p$.

\begin{rem}
 We have chosen to denote the limit by $M^{\gamma}_\infty$ rather than $M^{\gamma}_0$. While the latter may seem a more natural choice, 
 it is already in use for the initial condition of the martingale GMC approximation introduced in Section \ref{martinX}
 (see for instance \eqref{forinstancemzero}).
\end{rem}

\begin{rem}
We have chosen to put the emphasis on the proof of the convergence of $M^{\gamma}_{\gep}(f)$  for all fixed  $f$, but it is true also that 
$M^{\gamma}_{\gep}(\cdot)$ converges as a random distribution. More precisely the convergence (in probability) of $M^{\gamma}_{\gep}$ in a local Sobolev space of negative index can in fact be deduced from the estimates obtained in the proof of \eqref{stableconv}. We include the argument in Appendix \ref{sobolev}. 
\end{rem}
 \subsection{The boundary between phase II and III, and the triple point}

 Our second result concerns the case when  $\gamma\in \cP'_{\mathrm{II/III}} $
 where 
  \begin{equation}\label{onetwothree}\begin{split}
  \cP_{\mathrm{II/III}}&:= \{ \sqrt{d/2}+i \beta \ :  \beta> \sqrt{d/2}  \},\\
   \cP'_{\mathrm{II/III}} &:= \cP_{\mathrm{II/III}}\cup\{ \sqrt{d/2}(1+i) \}
\end{split}
  \end{equation}
 In that case $M^{\gamma}_{\gep}$ needs to be rescaled in order to obtain a non-trivial limit. 
 The convergence holds  only in law. To describe the limit we need to introduce two notions:
 Critical Gaussian Multiplicative Chaos, and Gaussian White Noise with a random intensity.

 \subsubsection*{Critical GMC}\label{criticalGMC}
 
 As explained in the introduction critical Gaussian Multiplicative Chaos is obtained as the limit of $M_{\gep}^{\alpha}$ when $\alpha=\sqrt{2d}$.
 The value $\sqrt{2d}$ represent a threshold for the convergence of $M_{\gep}^{\alpha}$. The  convergence result below follows from a combination of 
 \cite[Theorem 5]{MR3215583} - which establishes the convergence for the martingale sequence $M^{\alpha}_t$ (see Section \ref{martinX})  and 
 \cite[Theorems 1.1 and 4.4]{MR3613704}  which establish that the limit is the same for the exponential of the mollified field $M_{\gep}^{\alpha}$. Alternative concise proofs of these results have been recently given in \cite{critix}.

 \begin{theorema}\label{critGMC}
 Let $X$ be a Gaussian random field with an almost-star scale invariant kernel.
There exists a locally finite random  measure  $M'$  with dense support and no atoms such that   
 for every $f\in C_c(\bbR^d)$ the following convergence holds in probability
 \begin{equation}
  \lim_{\gep\to 0} \sqrt{\frac{\pi\log\left(1/\gep \right)}{2}}M_{\gep}^{\sqrt{2d}}(f)=
  M'(f).
 \end{equation}

\end{theorema}
 
 \begin{rem}
    Note that we have set different conventions and that our $M '$ differs from that in   \cite[Theorem 5]{MR3215583} by a factor $\sqrt{\frac{2}{\pi}}$.
 \end{rem}

 \subsubsection*{Complex white noise with random intensity given by a Real GMC}

For $\gamma\in \cP'_{\mathrm {II}/\mathrm {III}}$ we define $\mathfrak M^{\gamma}$ to be a complex white noise with intensity measure given by 
 $M'(e^{|\gamma|^2L}\cdot)$
It is a random linear form which is constructed jointly with $X$, on an extended probability space.
Conditionally on $X$, for $f\in C^{\infty}_c(\cD)$,  $\mathfrak M^{\gamma}(f)$ is a complex Gaussian random variable, with independent real and imaginary parts, both with a variance equal to  
 $$ M'(e^{|\gamma|^2L} f^2)= \int_{\cD} e^{|\gamma|^2 L(x,x)}f(x)^2 M'(\dd x).$$
Formally, letting  $\bbP$ and $\bar \bbP$ denote respectively the law of $X$ and  the joint law  of $(X,\mathfrak M^{\gamma}(\cdot))$,
$\mathfrak M^{\gamma}(\cdot)$ is the random process indexed by $C_c(\bbR^d)$
which  satisfies
 for any $m,n\ge 1$,  $\rho_1,\dots,\rho_m, f_1,\dots,f_n \in C_c(\bbR^d)$
 and any bounded measurable function $F$ on $\bbC^{n+m}$
\begin{multline}\label{cgwn}
 \bar\bbE \left[ F\left( (\langle X, \rho_i \rangle)_{i=1}^m, ( \mathfrak M^{\gamma}(f_j))^n_{j=1}  \right)  \right]\\= \bbE\otimes \bE_n \left[  F\left( (\langle X, \rho_i \rangle)_{i=1}^m, \Sigma[\gamma,X, (f_j)_{j=1}^n]\cdot \cN_n  \right) \right] 
\end{multline}
where under $\bP_n$, $\cN_n$ is an $n$ dimensional vector whose coordinate are IID standard complex Gaussian variables, and 
$\Sigma[\gamma,X, (f_j)_{j=1}^n]$ is the positive definite square root of the Hermitian matrix
$$
 \left(M'(e^{|\gamma|^2 L}f_i\bar f_j)\right)_{i,j=1}^n.$$

\subsubsection*{The result}

Let us define the function $\ell_{\theta}$  on $\bbR^d$, obtained by convoluting $z\mapsto \log 1/|z|$ twice with $\theta$, that is
\begin{equation}\label{ltheta}
 \ell_{\theta}(z):= \int_{\bbR^d} \log\left(\frac{1}{|z+z_1-z_2|}\right) \theta(z_1) \theta(z_2) \dd z_1 \dd z_2,
\end{equation}
and set (recall \eqref{onetwothree})
\begin{equation}\label{defvgep}    v(\gep,\theta,\gamma):=
 \begin{cases}
 (2\pi \log (1/\gep))^{-1/4}\gep^{\frac{d-|\gamma|^2}{2}} \left(\int_{\bbR^d} e^{|\gamma|^2\ell_{\theta}(z)} \dd z\right)^{1/2} \quad &\text{ if } \gamma \in \cP_{\mathrm{II/III}},\\
    \sqrt{\Sigma_{d-1}} \left(\frac{2 \log (1/\gep)}{\pi}\right)^{1/4} \quad  &\text{ if }  \gamma= \sqrt{d/2}(i+1)
 \end{cases}
 \end{equation}
where   $\Sigma_d$ is the volume of the $d-1$ dimensional sphere. Note that $\lim_{\gep \to 0} v(\gep,\theta,\gamma)=\infty$ in all cases.

\begin{theorem}\label{finalfrontier}
Let $X$ be a Gaussian random field with an almost star-scale covariance.
Then given $\gamma \in \cP'_{\mathrm{II/III}}$, 
we have the following joint convergence in law 
\begin{equation}\label{daub}
\left(X, \frac{M^{\gamma}_{\gep}}{v(\gep,\theta,\gamma)}\right) \stackrel{\gep \to 0}{\Rightarrow} (X,  \mathfrak M^{\gamma}),
\end{equation}
\end{theorem}

\begin{rem}
The convergence in \eqref{daub} implies that  $v(\gep,\theta,\gamma)^{-1} M^{\gamma}_{\gep}$ does not converge in probability.
On the heuristic level, this can be explained as follows: The white noise that appears in the limit is the product of local fluctuations of $X_{\gep}$. These fluctuations are produced by high frequencies in the Fourier spectrum of $X$. The set of frequencies that produce the fluctuations diverges to infinity when $\gep\to 0$. 
This means that the randomness that produces the white noise become asymptotically independent of $X$ in the limit.
\end{rem}

\begin{rem}
The convergence \eqref{daub} means that for any collection $(\rho_i)^m_{i=1}$ and $(f_j)^n_{j=1}$
we have the convergence in law of the $\bbC^{m+n}$ valued vector 
\begin{equation}
\lim_{\gep \to 0}\left((\langle X,\rho_i\rangle)^m_{i=1}, \left(\frac{M^{\gamma}_{\gep}(f_j)}{v(\gep,\theta,\gamma)}\right)^n_{j=1}\right)=\left((\langle X,\rho_i\rangle)^m_{i=1}, \left(\mathfrak M^{\gamma}(f_j)\right)^n_{j=1}\right).
\end{equation}
The convergence can also be shown to hold in a space of distribution.
 More precisely, there exists a modification of the process $\mathfrak M^{\gamma}$ taking values in the local Sobolev space $H^{-u}_{\mathrm{loc}}(\bbR^d)$ with $u>d/2$ and $\frac{M^{\gamma}_{\gep}(f_j)}{v(\gep,\theta,\gamma)}$ converges in law in that space.
See Appendix \ref{sobolev}.

\end{rem}

Since both $X$ and  $M^{\gamma}_{\gep}$ are linear forms, the convergence of finite dimensional marginals follows from that of one dimensional marginals (this can simply be checked using Fourier transform and Lévy Theorem). 
More precisely, we only need to prove the convergence  for every $f\in C_c(\bbR^d)$ and $\go\in[0,2\pi)$ of the real valued variable ($\mathfrak{Re}$ denotes the real part)
 \begin{equation}
 M^{\gamma}_{\gep}(f,\go):= \mathfrak{Re}\left( e^{-i\go}  M^{\gamma}_{\gep}(f)\right)
 \end{equation} 
Hence Theorem \ref{finalfrontier} can be reduced to the proof of the following statement
 \begin{proposition}\label{main}
Under the assumption of Theorem \ref{finalfrontier}, given $\rho, f\in C_c(\bbR^d)$, $\go\in[0,2\pi)$, we have
\begin{equation}\label{vfssa}
 \lim_{\gep \to 0} \bbE \left[e^{i \langle X,\rho \rangle + i \frac{M^{\gamma}_{\gep}(f,\go)}{v(\gep,\theta,\gamma)}} \right]=
 \bbE \left[ e^{i \langle X,\rho\rangle -\frac{1}{2}M'(e^{|\gamma|^2L}|f|^2)} \right].
 \end{equation}
 
\end{proposition}
\noindent The r.h.s.\ in \eqref{vfssa} of  corresponds to the Fourier transform of $(X,   \mathfrak M^{\gamma})$ (cf. \eqref{cgwn})
$$  \bbE \left[ e^{i \langle X,\rho\rangle -\frac{1}{2}M'(e^{|\gamma|^2L}|f|^2)} \right]=  \bar{ \bbE} \left[ e^{i \langle X,\rho\rangle + i \mathfrak M^{\gamma}(f)}\right],$$
and the convergence of the Fourier transform implies that of finite dimensional marginals. 
More detailed justifications are exposed in \cite[Section 1.2]{MR4413209}.

\section{The martingale approximation for GMC} \label{martinX}

Before getting to the technical core of the paper, we need one more introductory section to present an essential tool which is used 
in  the proof of both Theorem \ref{mainall} and Theorem \ref{finalfrontier}: the martingale decomposition of the field $X$.
Under the almost star-scale assumption for $K$, besides mollification, there is another natural way to approximate the $\log$-correlated field $X$ by 
 a smooth field. Extending the probability space, one can define a martingale sequence of smooth fields $(X_t)_{t\ge 0}$ that converges to $X$.

This allows for another approach to the construction of GMC, considering the exponential of the martingale approximation of $X$
(see \eqref{defgepp}) which we call $M^{\gamma}_t$ (see Remark \ref{kiconflito} concerning the conflict of notation). Convergence results for $M^{\gamma}_t$ which are a analogous to  Theorem \ref{mainall} and \ref{finalfrontier}
are also presented in this section.
In section \ref{randomCLT} we introduce an important technical tool which is used to prove Theorem \ref{finalfrontier}.
The result (a central limit Theorem proving convergence to a Gaussian with random variance) may find applications in other context, so it is stated in a rather general setup.

\subsection{The martingale decomposition of $X$} \label{martindecoco}

Given $K$ with an almost star-scale invariant part, and using the decomposition \eqref{iladeuxstar} for $\bar K$, we set
 $Q_t(x,y):=  \kappa(e^{t'}(x-y))$
where $t'$ is defined as the unique positive solution of 
\begin{equation}\label{dumbequation}
t'-\frac{\eta_1}{\eta_2}(1-e^{-\eta_2 t'})=t.
\end{equation}
We set 
\begin{equation}\label{kkttt}\begin{split}
 K_t(x,y)&:=K_0(x,y)+ \int^t_0 Q_s(x,y) \dd s \\
 &=K_0(x,y) + \int_{0}^{t'} (1-\eta_1 e^{-\eta_2 s}) \kappa(e^{s}(x-y))\dd s=:K_0(x,y)+ \bar K_t(x,y).
\end{split}\end{equation}
Note that we have $\lim_{t\to \infty} K_t(x,y)=K(x,y)$.
We define $(X_{t}(x))_{x\in \bbR^d, t\ge 0}$ to be a centered Gaussian field with covariance given by 
(using the  notation $a\wedge b:=\min(a,b)$, $a\vee b:=\max(a,b)$)
\begin{equation}\label{covofxt}
 \bbE[X_t(x)X_s(y)]= K_{s\wedge t}(x,y).
\end{equation}
Since  $(s,t,x,y)\mapsto K_{s\wedge t}(x,y)$ is H\"older continuous, 
the field admits a continuous modification.
We let  
$\mathcal F_t:= \sigma\left( (X_s(x))_{ x\in \bbR^d ,s\in [0,t]}\right)$
denote the natural filtration associated with $X_{\cdot}(\cdot)$.
The process $X$ indexed by $C_c(\bbR^d)$ and defined by
$\langle X,\rho\rangle= \lim_{t\to \infty} \int_{\bbR^d} X_t(x) \rho(x)\dd x,$
is a centered Gaussian field with covariance, so that $X_t$ is an approximation sequence for a $\log$-correlated field with covariance $K$.
We also define $\bar X_{\cdot}:= X_{\cdot}-X_0$. Recalling \eqref{kkttt} we have 
\begin{equation}\label{barbarx}
 \bbE[\bar X_t(x)\bar X_s(y)]= \bar K_{s\wedge t}(x,y).
\end{equation}
An important observation is that since $\bar K_t(x):=\bar K_t(x,x)=t$,  for any fixed $x\in \bbR^d$, the process $(\bar X_t(x))_{t\ge 0}$ is a standard Brownian Motion. We also introduce  the field $X_{t,\gep}$ which is the mollification of $X_{t}$, that is 
$$ X_{t,\gep}(x):= \int_{\bbR^d} \theta_{\gep}(x-z)X_t(z) \dd z= \bbE\left[ X_{\gep}(x) \ | \ \cF_t\right].$$
We let $K_{t,\gep}(x,y)$ denote the covariance of the field $X_{t,\gep}$ and 
$K_{t,\gep,0}(x,y)$ the cross-covariance of $X_{t,\gep}$ and $X_t$
\begin{equation}\label{covconvo}\begin{split}
 K_{t,\gep}(x,y)&:= \bbE[X_{t,\gep}(x)X_{t,\gep}(y)]= \int_{\bbR^d}\theta_{\gep}(x-z_1)\theta_{\gep}(y-z_2)K_t(z_1,z_2)\dd z_1 \dd z_2,\\
  K_{t,\gep,0}(x,y)&:= \bbE[X_{t,\gep}(x)X_{t}(y)]= \int_{\bbR^d}\theta_{\gep}(x-z)K_t(z,y)\dd z.
  \end{split}
\end{equation}
The quantity $\bar K_{t,\gep}$ is defined similarly and we use the notation $Q_{t,\gep}$ and $Q_{t,\gep,0}$ the corresponding mollified versions of $Q_t$.

\subsection{The martingale approximation for the GMC}

We define the distribution   $M^{\gamma}_{t}$ by setting for $f\in C_c(\bbR^d)$  
 \begin{equation}\label{defalfat}
    M^{\gamma}_{t}(f):=\int_{\bbR^d} f(x)  e^{\gamma X_{t}(x)- \frac{\gamma^2}{2} K_t(x)} \dd x.
 \end{equation}
Using the independence of the increments of $X$, it is elementary to check that 
$M_t(f)$ is an $(\cF_t)$-martingale. 
We also define 
\begin{equation}\label{defgepp}
     M^{\gamma}_{t,\gep}(f):=\int_{\bbR^d} f(x)  e^{\gamma X_{t,\gep}(x)- \frac{\gamma^2}{2} K_{t,\gep}(x)} \dd x= \bbE\left[ M^{\gamma}_{\gep}(f) \ | \ \cF_t\right].
\end{equation}

\subsection{A few properties of the covariance kernels}

We introduce some technical notation and estimates that are going to be of use throughout the article.
Let us first not that if a kernel $K$ has an almost star-scale invariant part
then it can be written in the form \eqref{fourme}.
Indeed, if $K$ satisfies \eqref{iladeuxstar} then the function $L$ defined for $x\ne y$ by
\begin{equation}
 L(x,y):= K(x,y)+\log |x-y|,
\end{equation}
can be extended to a continuous function on $\bbR^{2d}$. 
Note that we have 
\begin{equation}\label{ladefdejota}
 L(x)=\lim_{y\to x}\left(K(x,y)+\log |x-y|\right)= K_0(x)-  \mathfrak j
\end{equation}
where the difference term  $\mathfrak j$ does not depend on $x$ and can be computed explicitely
\begin{equation}
 \mathfrak j:= \lim_{z\to 0}\left(\log (1/|z|)- \bar K(0,z) \right)= \frac{\eta_1}{\eta_2}+\int^{\infty}_0  \left(1-\tilde \kappa (e^{-s})\right)\dd s<\infty.
 \end{equation}
The above comes from the fact that
 $$ \log (1/|z|)- \bar K(0,z)= \int^{\log(1/|z|)}_0 (1-Q_{\log(1/|z|)-u}(0,z)) \dd u$$
and the fact that the integrand on the r.h.s.\ converges to $1-\tilde \kappa (e^{\frac{\eta_1}{\eta_2}-s})$.
Lastly one can observe that the following identity holds
\begin{equation}\label{defbarl}
\bar{\ell}(z):=\lim_{t\to \infty} \left(\bar K_t(0,e^{-t}z)-t\right) 
= \lim_{t\to \infty} \int_0^{t} (\kappa(e^{s'-t}z)-1) \dd s
=\int_0^{\infty} (\kappa(e^{\frac{\eta_1}{\eta_2}-u}z)-1) \dd u,
\end{equation}
where in the integral in $s$, $s'$ is related to $s$ via \eqref{dumbequation}.
To obtain the third equality, one simply observe that $s'=s+\eta_1/\eta_2+o(1)$ in the large $s$ limit and make the change of variable $u=t-s$.
Note that $\bar \ell(z)$ is a continuous negative function and that for any $|z|\ge e^{-\frac{\eta_1}{\eta_2}}$ we have $\bar \ell(z)=\log \frac{1}{|z|}-\mathfrak j.$

\medskip

To conclude this subsection, we gather in a a technical lemma a couple of useful estimates concerning $K_t$, $Q_t$ and their variant.

\begin{lemma}\label{petipanda}
Given $R>0$, there exists a constant $C_R$ such that for any $x,y\in B(0,R)$, $t>0$ and $\gep\in [0,1]$ 
\begin{equation}\label{samdwich}
 \left| K_{t,\gep}(x,y)-\log\left(\frac{1}{\max(e^{-t},\gep,|x-y|)}\right)\right|\le C_R.
\end{equation}
The bound \eqref{samdwich} remains valid with $K_{t,\gep}$ replaced by  $K_t$ (with $\gep=0$), $K_{t,\gep,0}$, $\bar K_{t,\gep}$ etc...\\
\noindent We also have 
\begin{equation}\label{zaam}
 \int_{\bbR^d} Q_t(x,y) =  \int_{\bbR^d} Q_{t,\gep}(x,y)\dd y= \int_{\bbR^d} Q_{t,\gep,0}(x,y)\dd y   \le C e^{-dt},  
\end{equation}
and
\begin{equation}\label{zuum}
0 \le t-\bar K_{t,\gep}(x,y)\le C \left(e^{t}(|x-y|+\gep)\right)^2
\end{equation}

\end{lemma}
The estimates above can be proved rather directly from the definition.
A detailed proof of \eqref{samdwich} is provided in \cite[Appendix A.3]{critix}.
The bound \eqref{zaam} follows directly from the definition of $Q_t$ given above \eqref{dumbequation} and the fact that $|t-t'|$ is uniformy bounded.
The upper bound in \eqref{zuum} can be obtained by integrating (in time and space) the inequality
$$1-Q_t(z_1,z_2)\le C [e^{t'}|z_1-z_2|]^2$$ which follows directly from the Taylor expansion at second order of $\kappa$.

\begin{rem}\label{kiconflito}
 There is an obvious conflict of notation between $K_t$ introduced above  and $K_{\gep}$ introduced in \eqref{labig} 
 and the same can be said about $X_t$ and $M^{\gamma}_t$. This should not cause any confusion since we keep using the letter 
 $\gep$ for quantities related to the mollified field $X_{\gep}$ and latin letters for quantities related to the martingale approximation $X_t$. 
\end{rem}

\subsection{Convergence results  for the martingale approximation  }\label{lesresults}

An intermediate step to prove  Theorem \ref{mainall} and Theorem \ref{finalfrontier} is to show that similar results hold for the martingale approximation $M^{\gamma}_{t}$. These results present of course an interest in their own right.

\subsubsection*{The case  of $\gamma\in \cP_{\mathrm{I/II}}$}

\begin{proposition}\label{martinmain}
When $\gamma\in \mathcal P_{\mathrm{I/II}}$, the martingale $M^{\gamma}_t(f)$ is bounded $L^p$ for $p\in [1,\sqrt{2d}/|\alpha|) $. 
As a consequence the limit
\begin{equation}\label{artinlim}
\lim_{t\to \infty}M^{\gamma}_t(f)=:M^{\gamma}_{\infty}(f)
\end{equation}
exists almost surely. The convergence holds in $L^p$ and the limit is non-trivial.
\end{proposition}

\begin{rem}
The martingale limit in \eqref{artinlim} is the same as the limit of $M^{\gamma}_{\gep}$ appearing in Theorem \ref{mainall} (this is the reason why we use the same notation), we have 
\begin{equation}\label{artinlim2}
\lim_{t\to \infty}M^{\gamma}_t(f)=\lim_{\gep\to 0}M^{\gamma}_{\gep}(f).
\end{equation}
This observation is important, since it establishes that the limit in \eqref{stableconv} does not depend on the choice of the mollifier.
\end{rem}

%

\subsubsection*{The case $\gamma\in \cP'_{\mathrm{II/III}}$ }

In order to state the convergence in law result for $M^{\gamma}_t$, we need to introduce a normalization factor $v(t,\gamma)$ (analogous to  \eqref{defvgep} for the mollified case).
Let us set 
\begin{equation}
 v(t,\gamma)=\begin{cases} e^{\frac{|\gamma^2|\mathfrak j}{2}} \left(\frac{1} {2\pi t}\right)^{1/4}e^{\frac{(|\gamma|^2- d) t}{2}}\left(\int_{\bbR^d} e^{|\gamma|^2 \bar \ell (z)}\dd z\right)^{1/2} , &\quad \text{ if } |\gamma|^2> d\\
               \sqrt{\Sigma_{d-1}} \left(\frac{2 t}{\pi}\right)^{1/4} \
              &\quad \text{ if } |\gamma|^2= d.
              \end{cases}
\end{equation}
and define 
 \begin{equation}
 M^{\gamma}_{t}(f,\go):= \mathfrak{Re}\left( e^{-i\go}  M^{\gamma}_{t}(f)\right).
\end{equation}
The following analogue of Proposition \ref{main} holds.
 
\begin{proposition}\label{finalfrontierpropt}
If $X$ is an almost star-scale invariant field and $\gamma\in \cP'_{\mathrm{II/III}}$ we have
for any $\rho, f\in C_c(\bbR^d)$
\begin{equation}\label{vfssat}
 \lim_{t \to 0} \bbE \left[e^{i \langle X,\rho \rangle + i \frac{M^{\gamma}_{t}(f,\go)}{v(t,\gamma)}} \right]=
 \bbE \left[ e^{i \langle X,\rho\rangle -\frac{1}{2}M'(e^{|\gamma|^2L}|f|^2)} \right].
 \end{equation}
 As a consequence we have the following convergence in law (in the sense of finite dimensional marginals)
 \begin{equation}
  \left(X, \frac{M^{\gamma}_{t}}{v(t,\gamma)} \right)\quad \stackrel{t\to \infty}{\Longrightarrow} \quad (X, \mathfrak{M}^{\gamma}).
 \end{equation}

\end{proposition}

\subsection{CLT towards a Gaussian with random variance} \label{randomCLT}
We conclude this section by introducing a technical results which is essential to prove the convergence of a sequence of variable towards a Gaussian with random intensity in Theorem \ref{finalfrontier}. We provide the result and its proof in a reasonably high level of generality since it may find application in other contexts. 

\medskip

Consider $(\cF_t)_{t\ge 0}$ a filtration and $(W_n)_{n\ge1}$  a sequence of real valued random variables in $L^1$. We introduce for each $n\ge 1$ the martingale
\begin{equation}
W_{n,t}:= \bbE\left[W_n \ | \ \cF_t \right].
\end{equation}
We assume that
the martingale $W_{n,t}$ admits a modification which is continuous in $t$ for every $n\ge 1$
We prove  that $W_n$ converges to to a Gaussian with random variance if the quadratic variation of 
$(W_{n,t})_{t\ge 0}$ satisfy a law of large number and a couple of additional technical assumptions.
The result generalizes a similar CLT established for a single martingale process  (see \cite[Theorem 5.50, Chap. VIII-Section 5c]{jacodsh} or \cite[Theorem 2.5]{MR4413209}).

\begin{theorem}\label{zuperzlt}
 
Let us assume that  and that there exists a non-negative valued random-variable $Z$ which is such that the three following convergences in probability hold 
\begin{equation}\label{toutlesass}
 \lim_{n\to \infty} \frac{\langle W_n\rangle_{\infty}}{v^2(n)}=Z , \quad \forall t\ge 0  \lim_{n\to \infty} \frac{\langle W_n\rangle_{t}}{v^2(n)}=0, \text{ and } \lim_{n\to \infty} \frac{W_{n,0}}{v(n)}=0.   
\end{equation}
Then $X_n/v(n)$ converges in distribution towards a random Gaussian with variance given by $Z$, that is to say that for any $\cF_{\infty}$ 
bounded measurable $H$ we have 
\begin{equation}
 \lim_{n\to \infty} \bbE\left[ H e^{i \xi W_n/v(n)}\right]= \lim_{n\to \infty} \bbE\left[ H e^{-\frac{\xi^2 Z}{2}}\right]
\end{equation}
This is equivalent to saying that for any $\cF_{\infty}$ random variable $Y$
we have the following convergence in law
$$ (Y,W_n)\Longrightarrow (Y, \sqrt{Z} \cN)$$  
where $\cN$ is a standard Gaussian which is independent of $Z$ and $Y$.

\end{theorem}

\begin{rem}
 We believe that with adequate assumption on the size of the jumps, the result may extend to the case where $(W_{n,t})_{t\ge 0}$ is a càd-làg martingale, with the quadratic variation is replaced by the predictable bracket. Since we have no application in that setup, we restricted ourselves to the continuous case where the proof is technically simpler.
\end{rem}

\begin{rem}
 In Section \ref{proofofmain} we apply Theorem \ref{zuperzlt} for a sequence of variables indexed by $\gep\in (0,1)$ (namely $M^{\gamma}_{\gep}(f,\go)$) in the limit when $\gep\to 0$ rather than $n\ge 1$ and $n\to \infty$. These setups are equivalent.
\end{rem}

\subsection{Organization of the paper}
The remainder of the paper is organized as follows
\begin{itemize}
 \item In Section \ref{shortlapreuve} we prove all the statements concerning convergence in $\cP_{\mathrm{I/II}}$.
Section \ref{pmmain} is devoted to the proof of Proposition \ref{martinmain}. The more technical proof of Theorem \ref{mainall}, which uses Proposition \ref{martinmain} as in imput is displayed in Section \ref{pofmainall}.
\item The statements concerning $\gamma\in \cP'_{\mathrm{II/III}}$, namely Proposition \ref{finalfrontierpropt} and Proposition \ref{main}, while relying on relatively simple ideas, 
require a certain amount of technical computations.  
In Section \ref{poffff} we prove Proposition \ref{finalfrontierpropt}, in Section \ref{proofofmain}  Proposition \ref{main}.
\item In Section \ref{zuperproof}, we present the proof of Theorem \ref{zuperzlt}.
\end{itemize}
A significant amount of material is presented in appendices.
\begin{itemize}
 \item In Appendix \ref{teikos}, we prove a couple of auxilliary results used in Section \ref{poffff}/\ref{proofofmain}.
\item In Appendix \ref{sobolev}, we present and prove an extension of our main results, that is, the convergence of $M^{\gamma}_{\gep}(\cdot)$ as a distribution. After identifying the right 
topology, the proof mostly boils down to repeating the computation made in Section \ref{shortlapreuve}
(for $\gamma\in \cP_{\mathrm{I/II}}$) and Section \ref{proofofmain} (for $\gamma \in \cP'_{\mathrm{II/III}}$).
\item In Appendix \ref{beyondstar}, we explain how our results can be extended to the case of 
a (sufficiently regular) $\log$-correlated Gaussian field defined on an arbitrary open domain $\cD\subset \bbR^d$.
\item In Appendix \ref{taikos}, we present a relatively short proof of Lemma \ref{teknikos} for the sake of completeness. It the same as the one presented in \cite[Lemma 3.15]{LRV15}, except that we include a short proof of Lemma \ref{lemmadtwo} instead of relying on the branching random walk literature where more general results have been shown, albeit with much longer proofs (see for instance \cite{Hushi,mad17}). 

\end{itemize}

\subsubsection*{A comment on notation}
Throughout the paper, we use the letter $C$ for a generic positive constant when we need to compare two quantities. It may depend on some parameters (for instance on $\gamma$ or on the kernel $K$) but never on the variable $t$ or $\gep$.  
The value of $C$ might change from one equation to the other withing the same proof. 
We use $C'$ and $C''$ if we need several constants in the same display.

\section{Proof of convergence results on for  $\gamma\in \cP_{\mathrm{I/II}}$}\label{shortlapreuve}

In this section we prove Proposition \ref{martinmain} and Theorem \ref{mainall}.
The first one is easier, recall that due to the martingale property of $M^{\gamma}_t$, it is sufficient to show that 
 the sequence is bounded  in $L^p$ to prove convergence. This is performed in Section \ref{pmmain}. We rely on the Burkeholder-Davis-Gundy (BDG) inequality,  compute the quadratic variation of the martingale and studying its moment of order $p/2$.

\medskip

In Section \ref{pofmainall}, we adapt the same method to estimate  the $L^p$ norm of $M^{\gamma}_\gep-M^{\gamma}_\infty$. More precisely  the BDG inequality for the martingale
$(M^{\gamma}_{t,\gep}-M^{\gamma}_t)_{t\ge 0}$, and show that the moment of order $p/2$ of its quadratic variation is uniformly small in $t$.


\subsection{Proof of Proposition \ref{martinmain}}\label{pmmain}

Recalling that $\sqrt{2d}/\alpha>1$, we are going to prove that $M^{\gamma}_t(f)$ is bounded in $L^p$ for 
\begin{equation}\label{boundonp}
p\in\left( \sqrt{8d}/(3\alpha)\vee 1, \sqrt{2d}/\alpha\right).
\end{equation}
In the whole paper, when $M_t$ is a complex valued continuous martingale, we use the notation $\langle M\rangle_t$ to denote the the bracket between $M$ and $\bar M$. It is the predictable 
process such that $|M_t|^2-\langle M\rangle_t$ is a local martingale.
Using  Burkeholder-Davis-Gundy (BDG) inequality for $M^{\gamma}_t(f)$, there exists a constant $C_p$ such that for every $t>0$
\begin{equation}\label{forinstancemzero}
 \bbE\left[ \left| M^{\gamma}_t(f)\right|^p \right]\le  
 C_p \left( \bbE[ \langle M^{\gamma}(f)\rangle^{p/2}_t]+ \bbE\left [ |M^{\gamma}_0(f)|^p \right]\right).
\end{equation}
We have
\begin{equation}\label{unibound}
\bbE\left[ |M^{\gamma}_0(f)|^2 \right]= \int_{\bbR^{2d}} e^{|\gamma|^2 K_0(x,y)} f(x)\bar f(y) \dd x \dd y<\infty.
\end{equation}
Since $p<2$ by assumption, Jensen's inequality implies that $\bbE\left[ |M^{\gamma}_0(f)|^p \right]<\infty$. 
Using Itô calculus, we obtain an explicit expression for the quadratic variation 
\begin{equation}\label{deriv}
  \langle M^{\gamma}(f)\rangle_\infty = |\gamma|^2\int^\infty_0 A_t \dd t
\end{equation}
where 
\begin{equation}\label{ladefdea}
 A_t:= \int_{\bbR^{2d}} f(x)\bar f(y) Q_t(x,y)e^{ \gamma X_t(x)+ \bar \gamma X_t(y) - \frac{\gamma^2}{2} K_t(x)- \frac{\bar \gamma^2}{2} K_t(y)} \dd x \dd y.
\end{equation}
Note that $A_t$ is real and positive. 
From \eqref{forinstancemzero}, we deduce that  $M^{\gamma}_t(f)$ is bounded in $L^p$ if
$\bbE\left[ (\int^{\infty}_0 A_t \dd t)^{p/2}\right]<\infty$.
To bound $A_t$ from above, we take the modulus of the integrand in \eqref{ladefdea} and using the assumption that  $\beta=\sqrt{2d}-\alpha$ ($\gamma\in \mathcal P_{\mathrm{I/II}}$) we obtain that  
\begin{equation}
 A_t\le \int_{\bbR^{2d}} |f(x)f(y)| Q_t(x,y)e^{ \alpha( X_t(x)+  X_t(y)) + \frac{2d-2\sqrt{2d}\alpha}{2}\left( K_t(x)+ K_t(y)\right)} \dd x \dd y.
\end{equation}
Then using the inequality $ab\le\frac{a^2}{2}+\frac{b^2}{2}$
with 
$$ a= |f(x)| e^{ \alpha X_t(x)+\frac{2d-2\sqrt{2d}\alpha}{2} K_t(x)} \quad  \text{ and } \quad   b= |f(y)| e^{ \alpha X_t(y)+\frac{2d-2\sqrt{2d}\alpha}{2} K_t(y)}  $$
and symmetry in $x$ and $y$, we have 
\begin{equation}
A_t\le \int_{\bbR^{2d}} |f(x)|^2 Q_t(x,y)e^{ 2\alpha X_t(x) + (2d-2\sqrt{2d}\alpha) K_t(x)} \dd x \dd y.
\end{equation}
We use \eqref{zaam} to integrate over $y$ and  \eqref{samdwich} to replace replace $K_t(x)$ by $t$ (at the cost of multiplicative constant) and  we have
\begin{equation} \label{praims}
A_t\le C e^{dt}  \int_{\bbR^{d}} |f(x)|^2 e^{ 2\alpha(X_t(x) -\sqrt{2d} t)} \dd x.
\end{equation}
Now, as $\alpha>\sqrt{d/2}$, we have by $p/2<1$ by assumption.
We can use thus the following inequality (valid for an arbitrary collection of positive real numbers $(a_i)_{i\in I}$ and $q\in(0,1)$)  
\begin{equation}\label{subadd}
\left( \sum_{i\in I} a_i\right)^{q}\le \sum_{i\in I} a_i^{q},
\end{equation}
with $q=p/2$. In the remainder of the paper, we simply say ``by subadditivity'' when using \eqref{subadd}. 
Using \eqref{subadd} and Jensen's inequality
we have 
\begin{equation}\label{truukz}
\bbE\left[ \left(\int^{\infty}_0 A_t \dd t\right)^{p/2}\right]
\le \sum_{n\ge 0} \bbE\left[ \left(\int^{n+1}_n A_t \dd t\right)^{p/2} \right]\le \sum_{n\ge 0} \bbE\left[ \left(\int^{n+1}_n \bbE[ A_s \ | \ \mathcal F_n] \dd s\right)^{p/2} \right].
\end{equation}
Averaging with respect to  $(X_s-X_n)$ we obtain from \eqref{praims}
\begin{equation}\label{chuppa}
\int^{n+1}_n\bbE\left[ A_s \ | \ \cF_n\right]\le  Ce^{dn} \int_{\bbR^{2d}} |f(x)|^2 e^{2\alpha \left( X_n(x)- \sqrt{2d}n \right)} \dd x=: CB_n.
\end{equation}
As  $p>\sqrt{8d}/3\alpha$ by assumption,
we can conclude using the estimate in Lemma \ref{teknikos} below for the fractional moments of $B_n$ (the assumption on $p$ makes the r.h.s.\ of \eqref{ookk} summable in $n$).
More precisely, we deduce from \eqref{truukz},\eqref{chuppa} and \eqref{ookk} that $\bbE\left[ \left(\int^{\infty}_0 A_t \dd t\right)^{p/2}\right]<\infty$.
\begin{lemma}\label{teknikos}
 For $\alpha>\sqrt{d/2}$ and $p< \sqrt{2d}/\alpha$ we have 
 \begin{equation}\label{ookk}
  \bbE\left[ B^{p/2}_n\right] \le C n^{-\frac{3\alpha p}{\sqrt{8d}}}(\log n)^6.
 \end{equation}
\end{lemma}
 \noindent This result is a weaker version of \cite[Lemma 3.15]{LRV15}. 
 We provide, for the commodity of the reader a self-contained of Lemma \ref{teknikos}  in Appendix \ref{taikos}.

\subsection{Proof of Theorem \ref{mainall}}\label{pofmainall}

We  prove Theorem \ref{mainall} in the setup where our probability space contains a martingale approximation $(X_t)_{t\ge 0}$ of the field $X$ with covariance \ref{covofxt}. More precisely we show that  $M^{\gamma}_{\gep}(f)$ converges to the same limit as $M^{\gamma}_{t}(f)$. 
Working in an enlarged probability space entails by no mean a loss of generality since the validity of the statement ``the sequence $(M^{\gamma}_{\gep}(f))_{\gep\in(0,1]}$ is Cauchy in $L^p$'' is entirely determined by the distribution of $(X_{\gep}(x))_{x\in \bbR^d, \gep\in(0,1]}$.

\begin{proposition}\label{propopop}
Given $\gamma \in \cP_{\mathrm{I/II}}$ and $p\in [1,\sqrt{2d}/\alpha)$ we have
\begin{equation}\label{pipif}
 \lim_{\gep \to 0} \sup_{t>0}\bbE\left[ |(M^{\gamma}_{t}-M^{\gamma}_{t,\gep})(f)   |^p \right]=0.
\end{equation}
As a consequence  the following convergence holds in $L^p$
\begin{equation}\label{popof}
  \lim_{\gep \to 0} M^{\gamma}_{\gep}(f) =   M^{\gamma}_{\infty}(f) 
\end{equation}

\end{proposition}
\begin{proof}
Let us first show indicate how \eqref{popof} follows from \eqref{pipif}.
We observe that 
\begin{equation}\label{isil2}
\bbE\left[ |(M^{\gamma}_{t,\gep}- M^{\gamma}_{\gep})(f)|^2 \right]= \int^{\infty}_0 f(x)\bar f(y)\left( e^{|\gamma|^2 K_{\gep}(x,y)}-e^{|\gamma|^2 K_{t,\gep}(x,y)}\right)\dd x \dd y.
 \end{equation}
Since  $\lim_{t\to \infty} K_{t,\gep}(x,y)=K_{\gep}(x,y)$, using dominated convergence 
the r.h.s.\ tends to $0$ when $t\to \infty$ and  thus
$\lim_{t\to \infty} M^{\gamma}_{t,\gep}(f)=M^{\gamma}_{\gep}(f)$ in $L^2$, and hence also in $L^p$. Using Proposition \ref{martinmain} 
 we thus have the following convergence in $L^p$  
 $$\lim_{t\to \infty}(M^{\gamma}_{t}-M^{\gamma}_{t,\gep})(f)= (M^{\gamma}_{\infty}-M^{\gamma}_{\gep})(f).$$  
 Taking the limit when $\gep$ to zero, we obtain that
\begin{equation}
 \lim_{\gep\to 0}  \bbE\left[ |(M^{\gamma}_{\infty}-M^{\gamma}_{\gep})(f)   |^p \right]=\lim_{\gep\to 0}     \lim_{t\to \infty}\bbE\left[ |(M^{\gamma}_{t}-M^{\gamma}_{t,\gep})(f)   |^p \right].
     \end{equation}
     and we conclude using \eqref{pipif}.
     
     \medskip
     
\noindent To prove \eqref{pipif}, we assume that \eqref{boundonp} holds.
Then using the BDG inequality  (we omit the dependence in $f$ for ease of reading). We have for every $t\ge 0$
\begin{equation}\label{l223}
 \bbE[  |M^{\gamma}_{t}-M^{\gamma}_{t,\gep}  |^p]\le C_p \bbE[ \langle M^{\gamma}-M^{\gamma}_{\cdot,\gep}\rangle^{p/2}_\infty+   |M^{\gamma}_{0}-M^{\gamma}_{0,\gep}|^p].
\end{equation}
The reader can then check by an explicit calculation of the second moment that  
\begin{equation}\label{l222}
\lim_{\gep\to 0}\bbE\left[|M^{\gamma}_{0}(f)-M^{\gamma}_{0,\gep}(f)  |^p\right]\le \lim_{\gep\to 0}\bbE\left[|M^{\gamma}_{0}(f)-M^{\gamma}_{0,\gep}(f)  |^2\right]^{p/2} =0
\end{equation}
Hence in view of \eqref{l223}-\eqref{l222}, to prove \eqref{pipif} we need to show that  
\begin{equation}\label{aproev}
\lim_{\gep\to 0} \bbE[ \langle M^{\gamma}-M^{\gamma}_{\cdot,\gep}\rangle^{p/2}_\infty]=0.
\end{equation}
Expanding the product, using Itô calculus ($\mathfrak{Re}$ denotes the real part)  we obtain 
\begin{equation}\label{decomp}
 \langle M^{\gamma}-M^{\gamma}_{\cdot,\gep}\rangle_\infty = |\gamma|^2\int^\infty_{0} \left(A_t- 2 \mathfrak{Re}
\left(A^{(1)}_{t,\gep}\right)+ A^{(2)}_{t,\gep}\right)\dd t.
\end{equation}
where, $A_t$ is defined in \eqref{praims}, and recalling \eqref{covconvo}, $A^{(1)}_{t,\gep}$ and  $A^{(2)}_{t,\gep}$ are defined by
\begin{equation}\begin{split}
 A^{(1)}_{t,\gep}:=  \int_{\bbR^{2d} } f(x)\bar f(y) Q_{t,\gep,0}(x,y) e^{ \gamma X_t(x)+\bar \gamma X_{t,\gep}(y)- \frac{\gamma^2}{2} K_t(x) - \frac{\bar \gamma^2}{2} K_{t,\gep}(y)}
 \dd x \dd y,
 \\
  A^{(2)}_{t,\gep}:=\int_{\bbR^{2d} } f(x)\bar f(y) Q_{t,\gep}(x,y) e^{ \gamma X_{t,\gep}(x)+\bar \gamma X_{t,\gep}(y)- \frac{\gamma^2}{2} K_{t,\gep}(x) - \frac{\bar \gamma^2}{2} K_{t,\gep}(y)}
 \dd x \dd y.
 \end{split}
\end{equation}
We are going to reduce the proof of \eqref{aproev} to that of two convergence statements concerning $A^{(i)}_{t,\gep}$ for $i\in \{1,2\}$
(the first being valid for any fixed $r>0$)
\begin{align}\label{pprr1}
 \lim_{\gep\to 0} \sup_{t\in[0,r]} \bbE\left[ |A_t-A^{(i)}_{t,\gep}|\right] &=0  \quad \text{ for } i \in \{1,2\}.\\ \label{pprr2}
 \lim_{r\to \infty}  \sup_{\gep\in(0,1]}\bbE\left[  \left( \int^{\infty}_r |A^{(i)}_{t,\gep}| \dd t \right)^{p/2}\right]&=0 \quad \text{ for } i \in \{1,2\}.
\end{align}
Before proving \eqref{pprr1}-\eqref{pprr2} let us explain how \eqref{aproev} is deduced from it.
Note that \eqref{pprr2} is also valid for $A_t$ (this can be extracted from the proof in  Section \ref{pmmain}).
Given $\delta>0$, using subadditivity \eqref{subadd}, and \eqref{pprr2}   we can find $r_{\delta}$ such that for every  $\gep>0$
\begin{multline}\label{stepain}
\bbE\left[ \left(\int^{\infty}_{r_{\delta}} \left(A_t- 2 \mathfrak{Re}
\big(A^{(1)}_{t,\gep}\big)+ A^{(2)}_{t,\gep}\right)\dd t\right)^{p/2}\right]\\ 
\le
\bbE\left[ \left(\int^{\infty}_{r_{\delta}} A_t\dd t\right)^{p/2} +   \left(\int^{\infty}_{r_{\delta}} 2 |
A^{(1)}_{t,\gep}|\dd t\right)^{p/2} + \left(\int^{\infty}_{r_{\delta}} A^{(2)}_{t,\gep}\dd t\right)^{p/2}\right]
\le \delta/2.
\end{multline}
Now using first Jensen's inequality and then \eqref{pprr1} (recall that $A_t$ is real valued) we can find $\gep_\delta$ such that for every $\gep\in (0,\gep_\delta)$ 
\begin{multline}\label{stepaine}
\bbE\left[ \left( \int^{r_{\delta}}_0 \left(A_t- 2 \mathfrak{Re}
\big(A^{(1)}_{t,\gep}\big)+ A^{(2)}_{t,\gep}\right)\dd s\right)^{p/2}\right]
 \le  \left(\int^{r_{\delta}}_0 \bbE\left[A_t-2 \mathfrak{Re} 
\big(A^{(1)}_{t,\gep}\big)+ A^{(2)}_{t,\gep}\right]\dd s\right)^{p/2}\\\le 
 \left(\int^{r_{\delta}}_0 \bbE\left[ 2|A_t- 
A^{(1)}_{t,\gep}|+ |A^{(2)}_{t,\gep}-A_t|\right]\dd s\right)^{p/2}
\le \delta/2.
\end{multline}
Using subadditivity again we deduce from  \eqref{stepain}-\eqref{stepaine} that if $\gep\in (0,\gep_{\delta})$ we have
\begin{equation}
 \bbE\left[ \left( \int^{r_{\delta}}_0 \left(A_t- 2 \mathfrak{Re}
\big(A^{(1)}_{t,\gep}\big)+ A^{(2)}_{t,\gep}\right)\dd t\right)^{p/2}\right]\le \delta,
\end{equation}
which (recalling \eqref{decomp}) concludes the proof of \eqref{aproev}. 
Let us now prove \eqref{pprr1}-\eqref{pprr2}.
The proof  \eqref{pprr1} follows from a rather pedestrian but rather cumbersome computation of the $L^2$ norm of $(A_t-A^{(i)}_{t,\gep})$.
The following lemma summarizes the key points of this computation.

\begin{lemma}\label{explayk}
Consider the following:
\begin{itemize}
 \item Let $(\mathcal X,\mu)$ be a measured space and $\mathcal T$ be a set of indices.
 \item Let $Z_{t,\gep}(\cdot)$, $t\in \mathcal T$, $\gep\in (0,1]$ be a collection of complex valued Gaussian processes defined on $\mathcal X$.
 We set 
 \begin{equation}\label{lescovs}
 G_{t,\gep}(x,y):= \bbE[  Z_{t,\gep}(x) Z_{t,\gep}(y) 
 ]  \ \ \text{ and } \ \    H_{t,\gep}(x,y):= \bbE[  Z_{t,\gep}(x)\bar Z_{t,\gep}(y)
 ].
 \end{equation}
 \item  Let $Z_t$ be defined on the same probability space in such a way that $(Z_t,Z_{t,\gep})$ is jointly Gaussian. We let $G_t$ and $H_t$ be defined as in \eqref{lescovs} and set
 $$H_{t,\gep,0}(x,y):= \bbE[Z_{t,\gep}(x)\bar Z_{t}(y)].$$
 
\item Let $g_{t,\gep}$ and $g_{t}$ be deterministic functions $\cX\to \bbR$.
\end{itemize}
 We assume that:
 \begin{itemize}
  \item [(i)] The covariance functions are uniformly bounded, that is 
  $$\suptwo{t\in\cT}{\gep\in (0,1]}\sup_{x,y \in \cX} \max\left(H_{t,\gep}(x,y),H_t(x,y),H_{t,\gep,0}(x,y)\right)<\infty.$$
  \item [(ii)]  There exists a $\mu$-integrable function $h$ such that for every $t\in \cT$  and $\gep\in (0,1]$
  $$\forall x\in \cX, \quad \max(|g_{t,\gep}(x)|,|g_{t}(x)|)\le h(x)$$
  \item[(iii)] That for every $t\in \mathcal T$, we have the following pointwise convergence
    \begin{equation}
                     \lim_{\gep \to 0} g_{t,\gep}=g_t, \quad 
 \text{ and } \quad \lim_{\gep\to 0} H_{t,\gep}= \lim_{\gep\to 0} H_{t,\gep,0}=
  H_t 
                    \end{equation}

 \end{itemize}
Then setting
$$ W_{t,\gep}:= \int_{\cX} g_{t,\gep}(x)e^{Z_{t,\gep}(x)-\frac{1}{2}G_{t,\gep}(x)} \mu(\dd x) \quad \text{ and } \quad W_{t}:= \int_{\cX} g_{t}(x)e^{ Z_{t}(x)-\frac{1}{2}G_t(x,x)}  \mu(\dd x).$$
We have 
\begin{equation}
 \lim_{\gep\to 0}\sup_{t\in \mathcal T} \bbE\left[ |W_{\gep}- W_{t,\gep}|^2\right]=0.
\end{equation}

\begin{proof}[Proof of Lemma \ref{explayk}]
 The proof is actually much shorter than the statement. We have 
 \begin{multline}
  \bbE\left[ |W_{t,\gep}- W_{t}|^2\right]=\int_{\cX^2}\bigg( g_{t,\gep}(x)\bar g_{t,\gep}(y)e^{H_{t,\gep}(x,y)}-2\mathfrak{Re}\left(  g_{t,\gep}(x)\bar g_{t}(y)e^{H_{t,\gep,0}(x,y)}\right)\\
  + g_{t}(x)\bar g_{t}(y)e^{H_{t}(x,y)}\bigg)\mu(\dd x)\mu(\dd y)
 \end{multline}
and using our assumptions we can apply dominated convergence.
\end{proof}

\begin{proof}[Proof of \eqref{pprr1}]
We consider the case $i=2$ but the other one is identical. We set $\cX=\bbR^{2d}$, $\mu$ is Lebesgue measure, $\cT=[0,r]$, and
\begin{equation*}\begin{split}
Z_{t,\gep}(x,y)&= \gamma X_{t,\gep}(x)+ \bar \gamma X_{t,\gep}(y),\\ Z_{t}(x,y)&= \gamma X_{t}(x)+ \bar \gamma X_{t}(y), \\
g_{t,\gep}(x,y)&= Q_{t,\gep}(x,y)f(x)\bar f(y) e^{|\gamma|^2 K_{t,\gep}(x,y)},\\
g_{t}(x,y)&= Q_{t}(x,y)f(x)\bar f(y) e^{|\gamma|^2 K_{t}(x)}.
\end{split}\end{equation*}
Then the assumptions of Lemma \ref{explayk} are immediate to check.
\end{proof}

\end{lemma}

\noindent We now provide the details for the proof of \eqref{pprr2}  $i=2$ (the  case $i=1$ is similar). Let us set $n_0(\gep)=\lceil \log (1/\gep)\rceil$ and assume (without loss of generality) that $r$ is an integer and is smaller than $n_0$.  Using - as in the proof of Proposition \ref{propopop} - subadditivity  \eqref{subadd} and  Jensen's  inequality we obtain 
\begin{multline}
\bbE\left[ \left(\int^{\infty}_r |A^{(2)}_{t,\gep}| \dd s\right)^{p/2}\right]\le  \sum^{n_0}_{n=r} \bbE\left[\left(
 \int^{n+1}_n |A^{(2)}_{t,\gep}| \dd t\right)^{p/2}\right]
  \\ \le   \sum^{n_0-1}_{n=r} \bbE\left[\left( \int^{n+1}_n\bbE\left[
 |A^{(2)}_{t,\gep}|  \ | \ \cF_n \right]\dd t\right)^{p/2}\right]
 +  \bbE\left[\left( \int^{\infty}_{n_0}\bbE\left[
 |A^{(2)}_{t,\gep}|  \ | \ \cF_{n_0} \right]\dd t\right)^{p/2}\right].
\end{multline}
Proceeding as in \eqref{chuppa}, we obtain that if $t\in [n,n+1)$, $n\in \lint r, n_0-1\rint$, or 
$t\ge n_0$, $n=n_0$, we have (using Lemma \ref{petipanda} to replace the covariance $K_{n,\gep}(x)$ by $n$)
\begin{equation}\label{elepetitpoulet}
\bbE\left[
 |A^{(2)}_{t,\gep}|  \ | \ \bar \cF_n \right]\le C \int_{\bbR^{2d}} |f(x)|^2Q_{s,\gep}(x,y) e^{ 2\alpha(X_{n,\gep}(x)-\sqrt{2d}n) +2dn}
 \dd x \dd y.
\end{equation}
Using \eqref{zaam} to integrate over $y$ and
setting 
\begin{equation}
B^{(2)}_{n,\gep}:=\int_{\bbR^d} |f(x)|^2 e^{ 2\alpha(X_{n,\gep}(x)-\sqrt{2d}n) +dn}\dd x
\end{equation}
we obtain that 
\begin{equation}\label{hipss}
 \bbE\left[ \left(\int^{\infty}_r |A^{(2)}_{t,\gep}| \dd s\right)^{p/2}\right]
 \le  C \sum^{n_0}_{n=r} \bbE\left[ (B^{(2)}_{n,\gep})^{p/2} \right].
\end{equation}
Using Jensen's inequality for the probability $\theta_{\gep}(y-x)\dd y$, we can replace the mollification acting on $X_n$ in the exponential by one acting of $|f|^2$, we have
$$e^{ 2\alpha X_{n,\gep}(x)}\le \int_{\bbR^d} \theta_{\gep}(x-y) e^{ 2\alpha X_{n}(y) } \dd y$$ 
which after multiplying by $|f(x)|^2$ and integrating with respect to $x$ implies that 
\begin{equation}\label{baykz}
 B^{(2)}_{n,\gep}\le \int_{\cD} \left(|f|^{2}\ast\theta_{\gep}\right)(y)   e^{ 2\alpha \left(X_{n}(y)-\sqrt{2d}n \right)+dn}\dd y.
\end{equation}
Since  $|f|^{2}\ast\theta_{\gep}\le \|f\|^2_{\infty}\ind_{\{|x|\le R+1\}}$ if $f$ is supported in $B(0,R)$,  we can conclude using Lemma \ref{teknikos}, that 
$$ \bbE\left[ (B^{(2)}_{n,\gep})^{p/2} \right]\le  C n^{-\frac{3\alpha p}{\sqrt{8d}}} $$
for a constant which does not depend on $\gep$. Recalling that $p>\sqrt{8d}/3\alpha$ (cf. \eqref{boundonp}) we obtain combining\eqref{elepetitpoulet},  \eqref{hipss} and \eqref{baykz} that 
\begin{equation}
  \bbE\left[ \left(\int^{\infty}_r |A^{(2)}_{t,\gep}| \dd s\right)^{p/2}\right]\le C r^{1-\frac{3\alpha p}{2\sqrt{2d}}}.
\end{equation}
This concludes the proof of \eqref{pprr2}, and thus of Proposition \ref{propopop}.
\end{proof}

\section{Proof of Proposition \ref{finalfrontierpropt}}\label{poffff}

\subsection{Reduction to a statement concerning the total variation}

Using \cite[Theorem 2.5]{MR4413209} (which is a simpler version of Theorem \ref{zuperzlt} displayed above) we can reduce the proof of \eqref{vfssat} to the following convergence statement about the quadratic variation of the martingale.

\begin{proposition}\label{lawlarnum}
We have the following
 \begin{equation}\label{forallgo}
 \lim_{t\to \infty} v(t,\gamma)^{-2}\langle M^{\gamma}(f,\go)\rangle_t= M'(e^{|\gamma|^2 L}|f|^2).
\end{equation}
\end{proposition}

\begin{proof}[Proof of Proposition \ref{finalfrontierpropt} from Proposition \ref{lawlarnum}]
 We simply apply  \cite[Theorem 2.5]{MR4413209} to the martingale  $M^{\gamma}_t(f,\go)$.
\end{proof}

Setting, for notational simplicity $W_t:= M^{\gamma}_t(f)$. Recall that for a complex value martingale such as $W_t$
we use the notation $\langle W\rangle_t$ for the bracket between $W$ and its conjugate.
Using bilinearity of the martingale brackets we have

\begin{equation}\label{cbibi}
\langle M^{\gamma}(f,\go)\rangle_t= \frac{1}{2} \left( \langle W\rangle_t + \mathfrak{Re}(e^{-2i\go}\langle  W,W\rangle_t)\right)
\end{equation}
Hence to prove \eqref{forallgo}, it is sufficient to prove that following convergences hold in probability.

\begin{equation}\label{twostuffs}\begin{split}
\lim_{t\to \infty} v(t,\gamma)^{-2}\langle W\rangle_t&= 2 M'(e^{|\gamma|^2 L}|f|^2),\\
\lim_{t\to \infty}v(t,\gamma)^{-2} \langle W,W\rangle_t&= 0.
\end{split}\end{equation}
The expression for the bracket of $W_t$ can be obtained by using Itô calculus (recall \eqref{deriv})
More precisely we have
\begin{equation}\label{lezexprezion}
 \langle W\rangle_t=|\gamma|^2\int^t_0 A_s \dd s \quad \text{ and } \quad  \langle W,W\rangle_t= \gamma^2\int^t_0 B_s \dd s,
\end{equation}
where $A_t$ is defined in \eqref{ladefdea} and
\begin{equation}\label{ladefdeb}
 B_t:= \int_{\bbR^{2d}} f(x) f(y) Q_t(x,y)e^{ \gamma (X_t(x)+  X_t(y)) - \frac{\gamma^2}{2} (K_t(x)+ K_t(y))} \dd x \dd y.
\end{equation}
Now using \eqref{lezexprezion} our first idea is to deduce \eqref{twostuffs} from a convergence statement concerning $A_t$ and $B_t$.
A really important point here is that while $A_t$, properly rescaled,  converges  to $M'(e^{|\gamma|^2L} f)$ \textit{in probability},  this type of convergence is not  sufficient to say something about the integral $\int^t_0 A_s \dd s$.
A convenient framework to work with integrals is $L^1$ convergence, but the issue we encounter is that $A_t$ certainly does not converge in $L^1$ (we have  $\bbE\left[ |M'(e^{|\gamma|^2L} f)| \right]=\infty$ when $f$ is non trivial). 

\medskip

\noindent To bypass this problem, restrict ourselves to likely family of event and prove $L^1$ convergence for the restriction.
Recalling the definition of $\bar X$ \eqref{barbarx}, given $q\ge 0$ and $R>0$, $t\ge 0$ and $x\in \bbR^d$  we introduce the events
\begin{equation}\label{lezaq}\begin{split}
\cA_{t,q}(x)&:= \left\{ \max_{s \in[0,t]}  (\bar X_s(x)-\sqrt{2d} s)<  q\right\},\\ \mathcal A_{q,R}&:= \left\{ \sup_{s \ge 0, |x|\le R}  (\bar X_s(x)-\sqrt{2d} t)<  q\right\}=\intertwo{x\in B(0,R)}{t\ge 0} \cA_{t,q}(x).
 \end{split}
\end{equation}
A very important fact, which is a direct consequence of  \cite[Proposition 19]{MR3262492} (see also \cite[Proposition 2.4]{critix} for a concise proof).
\begin{lemma}\label{leventkile}
 We have for any fixed $R>0$
 \begin{equation}
\lim_{q\to \infty} \bbP\left[\mathcal A_{q,R}\right]=1
\end{equation}

\end{lemma}
\noindent We introduce (we drop the dependence in $\gamma$ in most displays to make them easier to read)
\begin{equation}\label{tetat}
 \phi(t)=\phi(t,\gamma):=\sqrt{\frac{2} {\pi (t\vee 1)}} e^{|\gamma^2|\mathfrak j}\left(\int_{\bbR^d} Q_t(0,z) e^{|\gamma^2| \bar K_t(0,z)} \dd z \right),
\end{equation}
which plays the role of a rescaling function for $A_t$.
Our main technical result in this section is the proof that $A_t/\phi(t)$ converges in $L^1$ towards $M'(e^{|\gamma|^2 L}|f|^2)$ after restriction to the event $\mathcal A_{q,R}$. 

\begin{proposition}\label{zincpoint2}
The following convergences hold for any $q\ge 0$ and any $R$ such that $\Supp(f) \subset B(0,R)$

\begin{equation}\label{easygo}
 \lim_{t\to \infty} \bbE\left[ | A_t/\phi(t)- M'(e^{|\gamma|^2 L}|f|^2)|\ind_{\mathcal A_{q,R}}\right]=0,
\end{equation}

\begin{equation}\label{easycome}
  \lim_{t\to \infty} \bbE\left[ \left| B_t/\phi(t)\right|\ind_{\mathcal A_{q,R}}\right]=0,
\end{equation}
and the above quantities are finite for every $t\ge 0$.

\end{proposition}

To show that Proposition \ref{zincpoint2} implies the convergence stated in Proposition \ref{lawlarnum},
we need to ensure that the rescaling by $\phi(t)$  matches that proposed for $\langle W\rangle_t$ (which is $v(t,\gamma)^2$) after integrating with respect to time. This is the purpose of the following lemma.

\begin{lemma}\label{replacement}
 We have for any $|\gamma|\ge d$
 \begin{equation}
  \lim_{t\to \infty} \frac{|\gamma|^2\int^t_0 \phi(s)\dd s}{2 v(t,\gamma)^2}=1.
 \end{equation}

\end{lemma}
The  proof of Lemma \ref{replacement} is presented in Appendix \ref{prepla}. Note that the goal of the lemma is only to obtain a more presentable expression for $v(t,\gamma)$ since without it, we can still prove that Proposition \ref{lawlarnum} and hence Proposition \ref{finalfrontierpropt} are valid with $v$ replaced by $\bar v(t,\gamma):=|\gamma| \sqrt{(\int^t_0 \phi(s)\dd s)/2}$.

\begin{proof}[Proof of Proposition \ref{lawlarnum}]
As we have seen, it is sufficient to prove \eqref{twostuffs}.
We provide the details concerning the convergence of  $\langle W \rangle_t$ (the first line in \eqref{twostuffs}) but that of
$\langle W, W\rangle_t$ can be obtained exactly in the same manner. Using \eqref{lezexprezion} and Jensen's inequality  we have
\begin{multline}\label{ghit}
\bbE\left[ \left|\frac{\langle W\rangle_t}{|\gamma|^2\int^t_0  \phi(s)\dd s}- M'(e^{|\gamma|^2L}|f|^2)\right|\ind_{\cA_{q,R}}\right]\\ \\ \le \frac{\int^t_0  \phi(s)\bbE\left[\left |\frac{A_s}{\phi(s)}-M'(e^{|\gamma|^2L}|f|^2)  \right|\ind_{\cA_{q,R}}\right] \dd s}{\int^{t}_0 \phi(s)\dd s}. 
\end{multline}
Observing that $\int^{\infty}_0 \phi(s)\dd s =\infty$, the r.h.s.\ of \eqref{ghit}
is simply a weighted Cesaro mean and thus we deduce from Proposition \ref{zincpoint2} and more precisely from \eqref{easygo} that
\begin{equation}
 \lim_{t\to \infty}\bbE\left[ \left|\frac{\langle W\rangle_t}{|\gamma|^2\int^t_0  \phi(s)\dd s}- M'(e^{|\gamma|^2L}|f|^2)\right|\ind_{\cA_{q,R}}\right]
 =0
\end{equation}
Since this holds for every $q>0$ we obtain that the following convergence  holds in probability 
(the replacement of  $|\gamma|^2\int^t_0 \phi(s)\dd s$ by $2v(t,\gamma)^2$ simply comes from Lemma \ref{leventkile}) that
$$\lim_{t\to 0}\left|\frac{\langle W\rangle_t}{2v(t,\gamma)^2}- M'(e^{|\gamma|^2L}|f|^2)\right|\ind_{\bigcup_{q\ge 1}\cA_{q,R}}=0$$
which, since the event in the indicator has probability one (cf. Lemma \ref{leventkile}) is the desired conclusion.
\end{proof}

\subsection{Restricted convergence in $L^2$ for the critical GMC}

Before starting the proof of Proposition \ref{zincpoint2}, we recall a result which play a key role  in the proof, the $L^2$ convergence of $M^{\sqrt{2d}}_t(g)$ towards $M'(g)$ when considering the restriction to the event $\cA_{q,R}$. This also implies convergence in $L^1$ which is what we require for the proof of Proposition \ref{zincpoint2}. The result can be deduced from the $L^2$ convergence of the truncated version of $M^{\sqrt{2d}}_t(g)$ which is proved in \cite{critix}.

\begin{lemma}\label{lelemma}
 We have for any $g$ in $C_c(\bbR^d)$ such that $\Supp(g) \subset B(0,R)$  and any $q>0$
 \begin{equation}
  \lim_{t\to \infty} \bbE\left[  \left|\sqrt{\frac{\pi t}{2}}M^{\sqrt{2d}}_t(g) - M'(g)\right|^2 \ind_{\cA_{q,R}} \right]=0,
 \end{equation}
and $\bbE\left[  |M'(g)|^2\ind_{\cA_{q,R}}  \right]<\infty$.
\end{lemma}

\begin{proof}
The fact that $\bbE\left[  |M'(g)|^2\ind_{\cA_{q,R}}  \right]<\infty$ is a simple consequence of the convergence since for any fixed $t$,
$ \bbE[ |M^{\sqrt{2d}}_t(g)|^2 ]<\infty$.
We set (recall \eqref{lezaq})
\begin{equation}\label{ziouup}
   M^{\sqrt{2d},(q)}_t(g):= \int g(x)e^{\sqrt{2d} X_t(x)-dK_t(x)}\ind_{\cA_{t,q}(x)}  \dd x,
 \end{equation}
 From \cite[Proposition 4.1]{critix},  there exists an $L^2$ variable  $\bar D^{(q)}_{\infty}(g)$ such that
 \begin{equation}\label{ziouupp}
  \lim_{t\to \infty}\bbE\left[  \left|\sqrt{\frac{\pi t}{2}}M^{\sqrt{2d},(q)}_t(g)-    \bar D^{(q)}_{\infty}(g)\right|^2 \right]=0.
 \end{equation}
It satisfies
$\bar D^{(q)}_{\infty}(g)=M'(g)$ on the event  $\cA_{q,R}$.
More precisely
 \cite[Proposition 4.1]{critix} is only stated in the special case where $g$ is an indicator function (to keep notation light) but the proof for $g\in C_c(\bbR^d)$ is identical.
On the event $ \cA_{q,R}$ we have   $M^{\sqrt{2d},(q)}_t(g)= M^{\sqrt{2d}}_t(g)$. Hence
 \begin{multline}
    \limsup_{t\to \infty} \bbE\left[  \left|\sqrt{\frac{\pi t}{2}}M^{\sqrt{2d}}_t(g) - M'(g)\right|^2 \ind_{\cA_{q,R}} \right]\\
    =     \limsup_{t\to \infty} \bbE\left[  \left|\sqrt{\frac{\pi t}{2}}M^{\sqrt{2d},(q)}_t(f)-    \bar D^{(q)}_{\infty}(g)\right|^2 \ind_{\cA_{q,R}} \right]=0.
 \end{multline}
where the last equality follows from \eqref{ziouupp}.
 \end{proof}

\subsection{Organizing the proof of Proposition \ref{zincpoint2}}

The two convergences rely on similar ideas, we focus on \eqref{easygo} which is the more delicate of the two.
The main idea is that since the integrand in the definition of $A_t$
 $$A_t:= \int_{\bbR^{2d}} f(x)\bar f(y) Q_t(x,y)e^{ \gamma X_t(x)+ \bar \gamma X_t(y) - \frac{\gamma^2}{2} K_t(x)- \frac{\bar \gamma^2}{2} K_t(y)} \dd x \dd y,$$
 vanishes when $|x-y|\ge e^{-t}$ (due to the presence of the multiplicative  $Q_t(x,y)$), the value of the integral should not be much affected much if one changes
 $\bar f(y)$, $X_t(y)$ and $K_t(y)$ by $\bar f(x)$, $X_t(x)$ and $K_t(x)$ in the expression.
 
 \medskip
 
 The quantity obtained after this replacement is, up to a multiplicative factor, of the form
 $M^{\sqrt{2d}}_t(g)$ (recall that $\gamma+\bar \gamma=\sqrt{2d}$)  for some function $g$. 
 Hence we should be able to conclude the proof of the convergence statement using  Lemma \ref{lelemma}.
 
 \medskip

 While this idea is relatively simple, it requires several steps to be implemented.
We set
\begin{equation}\label{param}
K^*_t(x,y):= K_0(x)+ \bar K_t(x,y) \quad \text{and} \quad r=r(t):= t- \log \log t
\end{equation}
(we are assuming that $t> e$ so that $0\le r\le t$).
We introduce the quantity $\tilde A_t$ which will appear after all our ``replacement'' steps have been performed, it is defined by
\begin{equation}\begin{split}
\tilde A_t&:=  \int_{\bbR^{2d}} Q_t(x,y) e^{|\gamma|^2 K^*_t(x,y)} |f(x)|^2 e^{\sqrt{2d} X_r(x) - dK_r(x)} \dd x   \dd y \\
&=  \left(\int_{\bbR^d}Q_t(0,z) e^{|\gamma|^2 \bar K_t(0,z)} \dd z\right)    \int_{\bbR^{d}}e^{|\gamma|^2 K_0(t)} |f(x)|^2 e^{\sqrt{2d} X_r(x) - dK_r(x)} \dd x    \\
&=  \phi(t)\sqrt{ \frac{\pi t}{2}}\int_{\bbR^{d}}e^{|\gamma|^2 L(x)} |f(x)|^2 e^{\sqrt{2d} X_r(x) - dK_r(x)} \dd x =
\phi(t)\sqrt{ \frac{\pi t}{2}}  M^{\sqrt{2d}}_{r}(e^{|\gamma|^2 L} |f|^2).
\end{split}\end{equation}
As a direct consequence of Lemma  \ref{lelemma} (since $r=t-o(t)$ the presence of $\sqrt{t}$ instead of $\sqrt{r}$ does not affect the convergence),
we have 
\begin{equation}
\lim_{t\to \infty} \bbE\left[ \left| \frac{\tilde A_t}{\phi(t)}- M'(e^{|\gamma|^2L}|f|^2) \right|\ind_{\cA_{q,R}} \right]=0.
\end{equation}
With this observation the proof of \eqref{easygo} reduces to showing that 
\begin{equation}\label{laprouvait}
\lim_{t\to 0}  \frac{1}{\phi(t)} \bbE\left[|A_t-\tilde A_t| \ind_{\cA_{q,R}} \right]=0.
\end{equation}
This requires some care but before going in the depth of the proof, let us explain the  heuristic behind \eqref{laprouvait}.
Note that $\tilde A_t$ is obtained from $A_t$ with two simple modifications:

\begin{itemize}
 \item We have replaced $f(x)\bar f(y)$ by $|f(x)|^2$.  
 \item In the exponential, we have replaced  $\gamma X_t(x)+ \bar \gamma X_t(y)$  by $\sqrt{2d} X_r(x)= (\gamma + \bar \gamma) X_r(x)$ \textit{and} 
$- \frac{\gamma^2}{2} K_t(x)- \frac{\bar \gamma^2}{2} K_t(y)$ by  
 $- dK_r(x) + |\gamma|^2 K^*_t(x,y).$
\end{itemize}
The first modification is rather straightfoward, we are integrating close to the diagonal so that $\bar f(y)$ is close to $\bar f(x)$. 
For the second modification, the idea is that replacing  $X_t(y)$ with $X_t(x)$ (and $t$ with $r$) should not yield big modifications \textit{provided that} we change the normalization to keep the expectation of the exponential unchanged (or almost so). In our case we have
\begin{equation}\begin{split}
                 \bbE\left[ e^{ \gamma X_t(x)+ \bar \gamma X_t(y) - \frac{\gamma^2}{2} K_t(x)- \frac{\bar \gamma^2}{2} K_t(y)}\right]&=e^{|\gamma|^2 K_t(x,y)},\\
                 \bbE\left[ e^{\sqrt{2d} X_r(x) - dK_r(x)+|\gamma|^2 K^*_t(x,y)} \right] &= e^{|\gamma|^2 K^*_t(x,y)}.
                \end{split}
\end{equation}
and, on the considered domain of integration, $K^*_t(x,y)$ and $K_t(x,y)$ are very close since $|x-y|\le e^{-t}$ when $Q_t(x,y)\ne 0$.
The proof of \eqref{laprouvait} requires three distinct steps which are detailed in the next subsection.

\subsection{The proof of \eqref{laprouvait}}

\subsubsection*{Step 1: Changing the deterministic prefactor in the integrand}

The integrand of $A_t$ and $\tilde A_t$ have different expectations. Our first step aims to fix this by replacing $\bar f(y)$ by $\bar f(x)$ in $A_t$ and doing a small modification in the exponential factor.
We set 
\begin{equation}
 A^{(1)}_t:= \int_{\bbR^{2d}} |f(x)|^2 Q_t(x,y) e^{ \gamma X_t(x)+ \bar \gamma X_t(y) + \frac{\gamma^2}{2}K_t(x)+ \frac{\bar \gamma^2}{2}K_t(y) +|\gamma|^2 \left(K_0(x)-K_0(x,y)\right)}\dd x \dd y
\end{equation}
We are going to prove that 
\begin{equation}\label{stepone}
 \lim_{t\to \infty} \phi(t)^{-1}\bbE\left[ |A_t-A^{(1)}_t|\ind_{\cA_{q,R}}\right]=0  
\end{equation}
Since $f$ and $K_0$ are uniformly continuous on the support of $f$ and $\Supp(f) \subset B(0,R)$, there exists a positive function $\delta$ with $\lim_{t\to \infty}\delta(t)=0$, such that for $|x-y|\le e^{-t}$
setting 
$$ F(x,y):=f(x)\bar f(y)-   |f(x)|^2 e^{|\gamma|^2 \left(K_0(x)-K_0(x,y)\right)}$$
we have
\begin{equation}
|F(x,y)|\le \delta(t) \ind_{B(0,R)}(x)\ind_{B(0,R)}(y)
\end{equation}
Hence  we obtain (since $\alpha=\sqrt{d/2}$, we have $\mathfrak{Re}(\gamma^2)=d-|\gamma|^2$)
 \begin{equation}\begin{split}\label{onlefaitunefois}
 |A_t-A^{(1)}_t|& = \left| \int_{\bbR^{2d}} Q_t(x,y) F(x,y) e^{ \gamma X_t(x)+ \bar \gamma X_t(y) + \frac{\gamma^2}{2}K_t(x)+ \frac{\bar \gamma^2}{2}K_t(y)}\dd x\dd y  \right| \\ 
 &\le \int_{\bbR^{2d}}  Q_t(x,y)|F(x,y)| e^{ \sqrt{\frac{d}{2}} (X_t(x)+ X_t(y)) + (|\gamma|^2-d)\frac{K_t(x)+K_t(y)}{2}} \dd x \dd y \\
  & \le \delta(t) \int_{B(0,R)^2}  Q_t(x,y) e^{ \sqrt{\frac{d}{2}} (X_t(x)+ X_t(y)) + (|\gamma|^2-d)\frac{K_t(x)+K_t(y)}{2}} \dd x \dd y \\
   &\le  \delta(t)\int_{B(0,R)^2}  Q_t(x,y)  e^{ \sqrt{2d} X_t(x) + (|\gamma|^2-d)K_t(x)} \dd x \dd y
 \end{split}\end{equation}
where the first inequality is simply obtained by taking the modulus of the integrand and in the third one we simply used $$Z(x)Z(y)\le \frac{1}{2}(Z(x)^2+Z(y)^2)$$ with  $Z(x)=e^{ \sqrt{\frac{d}{2}} X_t(x) + (|\gamma|^2-d)\frac{K_t(x)}{2}}$ and then  symmetry in $x$ and $y$. 
Then we observe that ($\gl$ denotes the Lebesgue measure) since $\cA_{t,q}(x)\subset \cA_{q,R}$ we have
\begin{equation}\begin{split}\label{tapz}
  \bbE\left[  e^{ \sqrt{2d} X_t(x) -dK_t(x)}\ind_{\cA_{q,R}} \right]
  &\le \bbE\left[ e^{\sqrt{2d}X(x)-dK_{t}(x)}\ind_{\cA_{t,q}(x)} \right]\\
  &= P[ \forall s\in[0,t], B_s\le q]\le  \sqrt{\frac{2}{\pi t}}q
\end{split}\end{equation}
where in the last line, we used Cameron-Martin formula (see Proposition \ref{cameronmartinpro} in the appendix) and the fact that $(\bar X_t(x))_{t\ge 0}$ is a standard Brownian Motion.
The last inequality is simply Lemma \ref{stupid}.
Combining \eqref{onlefaitunefois} and \eqref{tapz} and the fact that $K_0$ is bounded,  we have
\begin{equation}
 \bbE\left[ |A_t-A^{(1)}_t|\ind_{\cA_{q,R}}\right]\le \frac{C \delta(t)}{\sqrt{t}} \int_{B(0,R)^2}  Q_t(x,y)  e^{|\gamma|^2K_t(x)} \dd x\dd y \le C' \delta(t)\phi(t).
\end{equation}

\qed

\subsubsection*{Step 2: Taking conditional expectation}
Recalling the definition of $r(t)$  \eqref{param} we set
\begin{equation}
 A^{(2)}_t:= \bbE[A^{(1)}_t \ | \ \cF_r]\
\end{equation}
For this step of the proof (and only this step), we are going to assume that $K_0\equiv 0$ (and hence $X_0\equiv 0)$. Treating the case where $X_0$ is a non-trivial field does not present any extra difficulty besides the challenge of making the equations fit within the margins.
This assumption allows to replace $K_t(x)$ and $K_t(y)$ by $t$, and we get the following simplification for the expression of $A^{(1)}_t$.
\begin{equation}
 A^{(1)}_t:= \int_{\bbR^{2d}} |f(x)|^2 Q_t(x,y) e^{ \gamma X_t(x)+ \bar \gamma X_t(y) + (|\gamma|^2-d)t}\dd x \dd y
\end{equation}
Then we have
\begin{equation}\label{convenzion}
 A^{(2)}_t=\int_{\bbR^{2d}} |f(x)|^2 Q_t(x,y) e^{ \gamma  X_r(x)+ \bar \gamma  X_r(y)+(|\gamma|^2-d)r+ |\gamma|^2 K_{[r,t]}(x,y)  } \dd x \dd y,
\end{equation}
where $K_{[r,t]}=K_t-K_r$ (in the remainder of the paper, we use this convention for other quantities indexed by $t$). 
We are going to show that 
\begin{equation}\label{steptwo}
 \lim_{t\to \infty} \phi(t)^{-1}\bbE\left[ |A^{(1)}_t-A^{(2)}_t|\ind_{\cA_{q,R}}\right]=0  
\end{equation}
Recalling \eqref{lezaq} we define
\begin{equation}\begin{split}
 \bar A^{(1)}_t&:= \int_{\bbR^{2d}} |f(x)|^2 Q_t(x,y) e^{ \gamma X_t(x)+ \bar \gamma X_t(y) + (|\gamma|^2-d)t}\ind_{A_{r,q}(x)}\dd x \dd y,\\
  \bar A^{(2)}_t&:=\int_{\bbR^{2d}} |f(x)|^2 Q_t(x,y) e^{ \gamma  X_r(x)+ \bar \gamma  X_r(y)+(|\gamma|^2-d)r+ |\gamma|^2 K_{[r,t]}(x,y)  }\ind_{A_{r,q}(x)} \dd x \dd y.
\end{split}\end{equation}
Since  on $\cA_{q,R}$, $A^{(i)}_t$ and $\bar A^{(i)}_t$ coincide,
We have 
\begin{equation}
 \bbE\left[ (A^{(2)}_t-A^{(1)}_t)^2 \ind_{\cA_{q,R}}\right]\le \bbE\left[ (\bar A^{(2)}_t-\bar A^{(1)}_t)^2  \right]
\end{equation}
and thus we can prove that  \eqref{steptwo} holds by showing that 
\begin{equation}\label{shokk}
 \lim_{t\to \infty} \phi(t)^{-2}\bbE\left[ (\bar A^{(2)}_t-\bar A^{(1)}_t)^2  \right]=0.
\end{equation}
To bound $\bbE[ (\bar A^{(2)}_t-\bar A^{(1)}_t)^2  ]$ we expand the square, making it an integral on  $\bbR^{4d}$. We set
$$\xi(x,y):= |f(x)|^2 Q_t(x,y) e^{-(|\gamma|^2-d)t} \left( e^{ \gamma X_s(x)+ \bar \gamma X_s(y) }-\bbE \left[e^{  \gamma X_s(x)+ \bar\gamma X_s(y))} \ | \ \mathcal F_r \right]\right) \ind_{A_{r,q}(x)}.$$
We have 
\begin{equation}\label{barvario}
  \bbE\left[ (\bar A^{(2)}_t-\bar A^{(1)}_t)^2  \right]= \int_{\bbR^{4d}}\bbE\left[ \xi(x_1,y_1)\bar \xi(x_2,y_2)\right] \dd x_1 \dd y_1 \dd x_2 \dd y_2.
\end{equation}
As the range of correlation of the increment field $X_{[r,t]}:=X_t-X_r$ is smaller that $e^{-r}$ have,
whenever $|x_1-x_2|\ge 3 e^{-r}$
\begin{equation}
 \bbE\left[ \xi(x_1,y_1)\bar \xi(x_2,y_2) \ | \ \cF_r \right]=0.
\end{equation}
Hence we only need to integrate  the r.h.s.\ of \eqref{barvario} on the set $|x_1-x_2|\le 3 e^{-r}$.
In that case we use 
\begin{equation}
  \bbE\left[ \xi(x_1,y_1)\bar \xi(x_2,y_2) \right]\le   \bbE\left[ |\xi(x_1,y_1)|^2 \right]^{1/2} \bbE\left[|\xi(x_2,y_2)|^2 \right]^{1/2}.
\end{equation}
and 
\begin{equation}
   \bbE\left[ |\xi(x,y)|^2 \right]= |f(x)|^4 Q_t(x,y)^2 e^{2(|\gamma|^2-d)t}\bbE\left[ e^{\sqrt{2d}(X_t(x)+X_t(y))} \ind_{A_{r,q}(x)}\right]
\end{equation}
Using Cameron-Martin formula and the fact that $(\bar X_t(x))_{t\ge 0}$ is a standard Brownian motion we have
\begin{equation*}
 \bbE\left[ e^{\sqrt{2d}(X_t(x)+X_t(y))} \ind_{A_{r,q}(x)}\right] = e^{2d(t+K_t(x,y))}\bP\left[ \forall u\in [0,r], B_u\le q-K_u(x,y) \right].
\end{equation*}
Using \eqref{samdwich} (and then Lemma \ref{stupid}) we obtain for a constant $q'>q$ 
\begin{equation*}
 e^{-4dt}\bbE\left[ e^{\sqrt{2d}(X_t(x)+X_t(y))} \ind_{A_{q,r}(x)}\right]\le  \bP\left[ \forall u\in [0,r], B_u\le q'-\sqrt{2d}u  \right] \le  C r^{-3/2} e^{-d r}.
\end{equation*}
Altogether , setting 
$h(\bx,\by,t):= \ind_{\{|x_1-x_2|\le 3e^{-r}\}}|f(x_1)f(x_2)|^2 Q_t(x_1,y_1) Q_t(x_2,y_2)$ (recall that $r$ is a function of $t$) we obtain that for $t$ sufficiently large
\begin{multline}\label{patron}
   \bbE\left[ (\bar A^{(2)}_t-\bar A^{(1)}_t)^2  \right]\le C r^{-3/2}  e^{2(|\gamma|^2+d)t-dr}\int_{\bbR^{4d}}h(\bx,\by,t)  \dd \bx \dd \by\\ \le C' t^{-3/2} e^{2|\gamma|^2t-2dr}
   \le C'' t^{-1/2} e^{2d(t-r)}  \phi(t)^2\le t^{-1/4} \phi(t)^2.
\end{multline}
To get the second  inequality, simply observe that $h$ is smaller than a constant times the indicator of the set
$\{ |x_1|\le R, |x_2-x_1|\le 3e^{-r}, |y_i-x_i|\le e^{-t}, i=1,2 \}$, which has volume of order 
$e^{-d(r+2t)}$. The third inequality is a consequence of \eqref{aaazimp} (see the computation in the Appendix, while the last inequality follows from the the fact that with our choice of parameters \eqref{param} we have $t-r=o(\log t)$.

\subsubsection*{Step 3: Comparing $A^{(2)}_t$ and $\tilde A_t$}

Finally, we  show that 
\begin{equation}\label{stepthree}
 \lim_{t\to \infty} \phi(t)^{-1}\bbE\left[ |A^{(2)}_t-\tilde A_t|\ind_{\cA_{q,R}}\right]=0  
\end{equation}
 which together with \eqref{stepone}-\eqref{steptwo}, concludes the proof of \eqref{laprouvait}.
We introduce another smaller time parameter, namely $\bar r=t/2$ and define $\hat X(x,y)= X_{\bar r}(x)+ X_{[\bar r, r]}(y)$.
We want to replace $X_r(y)$ by $X_r(x)$ in the exponential with an intermediate steps, so we set
\begin{equation}\begin{split}
 Z_1(x)&:= \sqrt{2d} X_r(x),\\
 Z_2(x,y)&:= \gamma X_{r}(x)+\bar \gamma \hat X(x,y),\\
  Z_3(x,y)&:= \gamma X_{r}(x)+\bar \gamma  X_r(y).
\end{split}\end{equation}
The reader can check that we have
\begin{equation}\begin{split}
 \tilde A_t&:= \int_{\bbR^{2d}} |f(x)|^2 Q_t(x,y) e^{|\gamma|^2 K^*_t(x,y)} e^{ Z_1(x) -\frac{1}{2}\bbE[Z_1(x)]} \dd x \dd y,\\
 A^{(2)}_t&:= \int_{\bbR^{2d}} |f(x)|^2 Q_t(x,y) e^{|\gamma|^2 K^*_t(x,y)} e^{ Z_3(x.y) -\frac{1}{2}\bbE[Z_3(x,y)]} \dd x \dd y. 
 \end{split}
\end{equation}
In order to prove \eqref{stepthree} taking absolute value inside the integrand, 
we have
\begin{multline}
 \bbE\left[ |A^{(2)}_t-\tilde A_t|\ind_{\cA_{q,R}}\right]\le \maxtwo{|x|\le R}{|x-y|\le e^{-t}}\bbE\left[  \Big|e^{ Z_1(x) -\frac{\bbE[Z^2_1]}{2}}-  e^{ Z_3 -\frac{\bbE[Z^2_3]}{2}}\Big|\ind_{\cA_{q,R}}\right]\\
 \times \int_{\bbR^{2d}} |f(x)|^2 Q_t(x,y) e^{|\gamma|^2 K^*_t(x,y)} \dd x \dd y.
\end{multline}
Since the integral is of order  $e^{(|\gamma|^2-d)t}$ (cf. \eqref{samdwich}),
which is the  same order as $\phi(t) \sqrt{t}$ (cf. \eqref{aaazimp}), the estimate
\eqref{stepthree} boilds down to proving
\begin{equation}\label{individix}
 \lim_{t\to \infty}  \sqrt{ t} \maxtwo{|x|\le R}{|x-y|\le e^{-t}}\bbE\left[  \Big|e^{ Z_1(x) -\frac{\bbE[Z^2_1]}{2}}-  e^{ Z_3 -\frac{\bbE[Z^2_3]}{2}}\Big|\ind_{\cA_{q,R}}\right]=0.
\end{equation}
To prove \eqref{individix} we start with the decomposition
\begin{multline}\label{scroutch}
 \bbE\left[  \Big|e^{ Z_1(x) -\frac{\bbE[Z^2_1]}{2}}-  e^{ Z_3 -\frac{\bbE[Z^2_3]}{2}}\Big|\ind_{\cA_{q,R}}\right]\\
 \le  \bbE\left[  \Big|e^{ Z_1 -\frac{\bbE[Z_1^2]}{2}}-  e^{ Z_2 -\frac{\bbE[Z_2]}{2}}\Big|\ind_{\cA_{\bar r,q}(x)}\right]+  \bbE\left[ \Big| e^{ Z_3 -\frac{\bbE[Z^2_3]}{2}}- e^{ Z^2_2 -\frac{\bbE[Z^2_2]}{2}}\Big|\right]
\end{multline}
(this is just the triangle inequality and replacing $\cA_{q,R}$ with a larger event $\cA_{\bar r,q}(x)$) and show that each term is $o(t^{-1/2})$.
We start with the second one.
From Lemma \ref{ulikeit}, we have
\begin{multline}\label{hit1}
 \bbE\left[  e^{ Z_3 -\frac{\bbE[Z^2_3]}{2}}- e^{ Z^2_2 -\frac{\bbE[Z^2_2]}{2}}|\right] \le  C\sqrt{\bbE[ |Z_3-Z_2|^2]}\\
 \\ = C|\gamma| \sqrt{\bbE[(X_{\bar r}(x)-X_{\bar r}(y))^2]}
\le C' e^{- ct},
\end{multline}
where we have used that 
$$\bbE[(X_{\bar r}(x)-X_{\bar r}(y))^2]= 2(\bar r- \bar K_{\bar r}(x,y))+(K_0(x)+K_0(y)-2K_0(x,y)).$$ 
The second part of the sum is smaller than $|x-y|^c$ since $K_0$ is H\"older continuous and the first part is smaller than $|x-y|^2 e^{2\bar r}$ (from \eqref{zuum}), both are exponentially small in $t$.
For the  first term in \eqref{scroutch} we factorize the part that is $\mathcal F_{\bar r}$ measureable and use independence to obtain 
\begin{multline}\label{fector}
  \bbE\left[  |e^{ Z_1 -\frac{\bbE[Z_1^2]}{2}}-  e^{ Z_2 -\frac{\bbE[Z_2]}{2}}|\ind_{A_{q,\bar r}(x)}\right]
  \\=   \bbE\left[  e^{ \sqrt{2d} X_{\bar r}(x)- d K_{\bar r}(x)}\ind_{A_{q,\bar r}(x)}\right]   \bbE\left[  |e^{ Z'_1 -\frac{\bbE[(Z'_1)^2]}{2}}-  e^{ Z'_2 -\frac{\bbE[(Z')^2_2]}{2}}|\right],
\end{multline}
where $Z'_i= Z_i - \sqrt{2d} X_{\bar r}(x)$. Using Cameron-Martin formula and Lemma \ref{stupid}, we have
\begin{equation}\label{hit2}
 \bbE\left[  e^{ \sqrt{2d} X_{\bar r}(x)- d K_{\bar r}(x)}\ind_{A_{\bar r,q}(x)}\right]=
 \bP\left[ \forall s\in [0,\bar r], B_s\le q \right]\le  \sqrt{\frac{2}{\bar r\pi}} q .
\end{equation}
The factor $\bar r^{-1/2}$ is sufficient to cancel the $\sqrt{t}$ in \eqref{individix} and we just have to show that the second factor in \eqref{fector} is small.
From Lemma \ref{ulikeit} we have
\begin{multline}\label{hit3}
  \bbE\left[  |e^{ Z'_1 -\frac{\bbE[(Z'_1)^2]}{2}}-  e^{ Z'_2 -\frac{\bbE[(Z'_2)^2]}{2}}|\right]\le \sqrt{\bbE\left[  |Z'_1 -Z'_2|^2\right]}\\=|\gamma| \sqrt{\bbE\left[|X_{[\bar r,r]}(x)- X_{[\bar r,r]}(y)|^2\right]}\le  C e^{r}|x-y|\le Ce^{r-t},
 \end{multline}
 where the penultimate inequality can be deduced from \eqref{zuum}.
 The combination of \eqref{hit1}-\eqref{hit2} and \eqref{hit3} concludes the proof of \eqref{individix}. \qed
 
 \subsubsection*{Bonus step: the case of $B_t$}
To conclude let us sketch rapidly the proof of \eqref{easycome}.
We can repeat the argument of step $2$ to show that 
\begin{equation}
 \lim_{t\to \infty} \phi(t)^{-2}\bbE[ |B_t- \bbE[B_t \ | \  \cF_r]|^2 ]=0.
\end{equation}
Then it is rather direct to check that 
\begin{equation}
  \lim_{t\to \infty} \phi(t)^{-1} \bbE\left[|\bbE[B_t \ | \  \cF_r]|\ind_{\cA_{q,R}} \right]=0.
\end{equation}
More precisely we have
\begin{equation*}
\bbE[B_t \ | \  \cF_r]= \int_{\bbR^{2d}} f(x) f(y) Q_t(x,y)e^{ \gamma (X_r(x)+  X_r(y)) + \frac{\gamma^2}{2} (2K_{[r,t]}(x,y)- K_r(x)- K_r(y))} \dd x \dd y.
 \end{equation*}
 Taking the absolute value of the integrand, using \eqref{samdwich} to evaluate $K_t$ and $K_{[r,t]}$, then the inequality $ab\le (a^2+b^2)/2$ and symmetry, and finally \eqref{zaam}
 \begin{equation}\begin{split}
  |\bbE[B_t \ | \  \cF_r]|&\le C e^{(|\gamma^2|-d)(2r-t)} \int_{\bbR^{2d}}|f(x) f(y)| Q_t(x,y)e^{ \sqrt{d/2}(X_r(x)+  X_r(y))}\dd x \dd y\\
 & \le C e^{(|\gamma^2|-d)(2r-t)} \int_{\bbR^{2d}}|f(x)|^2 Q_t(x,y)e^{ \sqrt{2d}X_r(x)}\dd x \dd y
 \\ &\le C' e^{(|\gamma^2|-d)(2r-t)-dt} \int_{\bbR^{2d}}|f(x)|^2 e^{ \sqrt{2d}X_r(x)}\dd x
 \end{split}\end{equation}
Hence we have 
\begin{equation}
 \bbE\left[|\bbE[B_t \ | \  \cF_r]|\ind_{\cA_{q,R}} \right]
 \le e^{(|\gamma^2|-d)(2r-t)-dt} \int_{\bbR^{2d}}|f(x)|^2 \bbE\left[ e^{ \sqrt{2d}X_r(x)} \ind_{\cA_{r,q}(x)}\right]\dd x.
\end{equation}
 Using Cameron Martin formula, \eqref{samdwich} and Lemma \ref{stupid} (recall that $r\sim t$) we obtain that 
 \begin{equation}
  \bbE\left[ e^{ \sqrt{2d}X_r(x)} \ind_{\cA_{r,q}(x)}\right]\le C t^{-1/2} e^{dr}.
 \end{equation}
Overall using \eqref{aaazimp} we have
$  \phi(t)^{-1}\bbE\left[|\bbE[B_t \ | \  \cF_r]|\ind_{\cA_{q,R}} \right]\le C  e^{-|\gamma^2|(t-r)}.$ 
\qed

 \section{Proof of Proposition  \ref{main}}\label{proofofmain}

 \subsection{Organization of the proof}
 Like for the proof of Theorem \ref{mainall}, we  assume  that our probability space contains a martingale approximation sequence $(X_t)_{t\ge 0}$ of the field $X$, with covariance given by \eqref{covofxt}. For the same reason as the one exposed at the beginning of Section \ref{pofmainall} this entails no loss of generality.

The main idea is to apply Theorem \ref{zuperzlt} (for the filtration corresponding to $(X_t)$) to the family $M^{\gamma}_{\gep}(f,\go)$
with rate $v(\gep,\theta,\gamma)$ and with the variable $Z$ being equal to $ M'(e^{|\gamma|^2L}|f|^2)$. Hence need to check that the martingale $M^{\gamma}_{t,\gep}(f,\go):= \bbE\left[ M^{\gamma}_{\gep}(f,\go) \ | \ \cF_t\right]$ satisfy all the requirements in \eqref{toutlesass}.
Setting
$W^{(\gep)}_{t}:= M^{\gamma}_{t,\gep}$ (recall \eqref{defgepp}), and using the bilinearity of the martingale bracket like in \eqref{cbibi} we obtain
\begin{equation}\label{cbibi2}
\langle M^{\gamma}_{\cdot,\gep}(f,\go)\rangle_t= \frac{1}{2} \left( \langle W^{(\gep)}\rangle_t + \mathfrak{Re}(e^{-2i\go}\langle  W^{(\gep)},W^{(\gep)}\rangle_t)\right).
\end{equation}
The requirements concerning the quadratic variation of $M^{\gamma}_{t,\gep}(f,\go)$ can be obtained as consequences of the following,

\begin{proposition} \label{lapastry}
The following convergences hold
\begin{equation}\begin{split}\label{daci}
 \lim_{\gep \to 0}  \bbE\left[ \left|\frac{\langle W^{(\gep)} \rangle_{\infty}}{2v(\gep,\theta,\gamma)^2} -  M'(e^{|\gamma|^2L}|f|^2) \right|\ind_{\mathcal A_{q,R}}\right]&=0,\\
  \lim_{\gep \to 0}  \bbE\left[ \left|\frac{\langle W^{(\gep)},W^{(\gep)}  \rangle_{\infty}}{v(\gep,\theta,\gamma)^2} \right|\ind_{\mathcal A_{q,R}}\right]&=0.
 \end{split}\end{equation}
Furthermore we have for any fixed $t$ we have 
\begin{equation}\label{daco}
\sup_{\gep\in (0,1)}\bbE[\langle W^{(\gep)}\rangle_t ]<\infty.
\end{equation}

\end{proposition}

\noindent Proposition \ref{lapastry} is proved in the next subsection, let us first show how our main results can be deduced from it. 

\begin{proof}[Proof of Proposition \ref{main}]
We must check that the three requirements in \eqref{toutlesass} are satisfied since the result follows then from  Theorem \ref{zuperzlt}.
Given that $\lim_{\gep\to 0}v(\gep,\theta,\gamma)=\infty$, it is sufficient for the second and third requirements to show that that the sequences
$(M^{\gamma}_{0,\gep}(f,\go))_{\gep\in(0,1)}$, and $(\langle M^{\gamma}_{\cdot,\gep}(f,\go)\rangle_t)_{\gep\in (0,1)}$ (for a fixed $t$) are tight.
The sequences are in fact uniformly bounded in $L^1$. We have
\begin{equation}
 \sup_{\gep\in(0,1)}\bbE[|M^{\gamma}_{0,\gep}(f,\go)| ]\le  \sup_{\gep\in (0,1)} \bbE[|M^{\gamma}_{0,\gep}(f)| ]<\infty.
\end{equation}
Indeed  taking the absolute value of the integrand, we have 
\begin{equation}
 \bbE[|M^{\gamma}_{0,\gep}(f)| ]\le\int_{\bbR^d} \bbE\left[  f(x)e^{\sqrt{d/2} X_{0,\gep}(x)+\frac{\beta^2-(d/2)}{2}K_{0,\gep}(x)}\right]\dd x= \int_{\bbR^d} f(x)e^{\beta^2 K_{0,\gep}(x)}\dd x,
\end{equation}
and the uniform bound follows from \eqref{samdwich}.
From \eqref{cbibi2} we have $\langle M^{\gamma}_{\cdot,\gep}(f,\go)\rangle_t\le \langle W^{(\gep)}\rangle_t$ and thus the uniform boundedness in $L^1$
is  consequence of \eqref{daco}. Let us now turn to the first and main requirement in \eqref{toutlesass}.
The convergences in \eqref{daci} imply the following convergence in probability 
 \begin{equation}\begin{split}
  \lim_{\gep\to 0}\frac{\langle W^{(\gep)} \rangle_{\infty}}{2v(\gep,\theta,\gamma)^2}\ind_{
 \bigcup_{q\ge 1 } \mathcal A_{q,R} }&=  M'(e^{|\gamma|^2L}|f|^2),\\
  \lim_{\gep\to 0}\frac{\langle W^{(\gep)}, W^{(\gep)} \rangle_{\infty}}{v(\gep,\theta,\gamma)^2}\ind_{
 \bigcup_{q\ge 1 } \mathcal A_{q,R} }&= 0.
\end{split} \end{equation}
Using Lemma \ref{leventkile} and \eqref{cbibi2}, we conclude that $$\lim_{\gep \to 0}v(\gep,\theta,\gamma)^{-2}\langle M^{\gamma}_{\cdot,\gep}(f,\go)\rangle_\infty=M'(e^{|\gamma|^2L}|f|^2)$$ in probability.
\end{proof}

\noindent 

As another preliminary step to our proof, we reduce the convergence statement in Proposition \ref{lapastry} to a convergence of the 
derivative of the  martingale brackets.
Using Itô calculus we obtain that for $T\in [0,\infty]$,
$$   \langle W^{(\gep)} \rangle_{T}=\int^{T}_0  A_{t,\gep} \dd t \  \text{ and }   \  \langle W^{(\gep)} \rangle_{T}=\int^{T}_0  B_{t,\gep} \dd t $$
 where
 \begin{equation}\begin{split}
  A_{t,\gep}&:=\int_{\bbR^{2d} } f(x)\bar f(y) Q_{t,\gep}(x,y) e^{ \gamma X_{t,\gep}(x)+\bar \gamma X_{t,\gep}(y)- \frac{\gamma^2}{2} K_{t,\gep}(x) - \frac{\bar \gamma^2}{2} K_{t,\gep}(y)}
 \dd x \dd y,\\
  B_{t,\gep}&:=\int_{\bbR^{2d} } f(x) f(y) Q_{t,\gep}(x,y) e^{ \gamma X_{t,\gep}(x)+ \gamma X_{t,\gep}(y)- \frac{\gamma^2}{2}\left( K_{t,\gep}(x) +K_{t,\gep}(y)\right)}
 \dd x \dd y.
\end{split}\end{equation}
Similarly to what has been done in Proposition \ref{zincpoint2}, we are going to show that, with appropriate renormalizations and restrictions, $A_{t,\gep}$ and $B_{t,\gep}$ converge in $L^1$ to  $M'(e^{|\gamma|^2 L} |f|^2)$ and $0$ respectively.
To this end we introduce a couple of parameters (recall \eqref{covconvo})
\begin{equation}\begin{split}\label{pararara}
 \bar t(t,\gep)&:= t\wedge \log (1/\gep)\\
 \phi(t,\gep)&:=  \sqrt{\frac{2}{\pi (\bar t \vee 1)}} e^{|\gamma|^2 \mathfrak j} \left(\int_{\bbR^{d} }  e^{|\gamma|^2 \bar K_{t,\gep}(0,z)} Q_{t,\gep}(0,z) \dd z\right).
\end{split}\end{equation}
The quantity $r$ will on
Our aim is to prove the following 
\begin{proposition} \label{l3333}
\begin{equation}\label{lesgross}\begin{split}
 \limtwo{\gep \to 0}{t\to \infty}  \bbE\left[ \left|\frac{A_{t,\gep}}{\phi(t,\gep)} - M'(e^{|\gamma|^2 L} |f|^2)  \right|\ind_{\mathcal A_{q,R}}\right]&=0,\\
  \limtwo{\gep \to 0}{t\to \infty}  \bbE\left[ \left|\frac{B_{t,\gep}}{\phi(t,\gep)} \right|\ind_{\mathcal A_{q,R}}\right]=0,
 \end{split}\end{equation}
 and for any $T<\infty$
 \begin{equation}\label{lespetitst}
   \suptwo{t\in [0,T]}{\gep\in(0,1)} \bbE\left[ |A_{t,\gep}| \right]<\infty. 
 \end{equation}

\end{proposition}

 \begin{rem}
Let us underline that  $\limtwo{\gep \to 0}{t\to \infty} F(t,\gep) =0$ means that
  that there exists $t_0(\delta)$ and $\gep_0(\delta)$ such that $|F(t,\gep)|\le \delta$ when $t\ge t_0$ AND $\gep\in(0,\gep_0)$. 
  This is a stronger statement than both $\lim_{\gep \to 0}\lim_{t\to \infty} F(t,\gep)=0$ or  $\lim_{t\to \infty}\lim_{\gep \to 0} F(t,\gep)=0$
 \end{rem}

 \noindent Clearly \eqref{lespetitst} implies \eqref{daco}. To deduce \eqref{daci} from \eqref{lesgross}, we need to check that 
 renormalizing factor $2 v(\gep,\theta,\gamma)^2$ corresponds to the integral of $\phi(t,\gep)$.
 This is the content of the following lemma whose proof is presented in Appendix \ref{prepla2}.
 
 \begin{lemma}\label{replacement2}
  We have for any $|\gamma|\ge d$
 \begin{equation}
  \lim_{\gep\to 0} \frac{|\gamma|^2\int^\infty_0 \phi(t,\gep)\dd t}{2 v(\gep,\theta,\gamma)^2}=1
 \end{equation}
\end{lemma}

\noindent We can now complete the proof of Proposition \ref{lapastry} using Proposition \ref{l3333}
 
\begin{proof}[Proof of Proposition \ref{lapastry}]
From Lemma \ref{replacement2} it is sufficient to prove the convergence of

\begin{multline}
 \bbE\left[ \left|\frac{\langle W^{(\gep)} \rangle_{\infty}}{\int^{\infty}_{0} |\gamma|^2\phi(t,\gep)} - M'(e^{|\gamma|^2 L} |f|^2)\right|\ind_{\mathcal A_{q,R}} \right]
 \\
 \le \frac{1}{\int^{\infty}_{0}  \phi(t,\gep)\dd t} \int^{\infty}_0  \phi(t,\gep) \bbE\left[  \left| \frac{A_{t,\gep}}{\phi(t,\gep)}-M'(e^{|\gamma|^2 L} |f|^2)\right|\ind_{\mathcal A_{q,R}} \right] \dd t.
\end{multline}
Let us fix $\delta>0$, and let $T$ and $\gep_0$ be such that for all $t>T$ and $\gep<\gep_0$ we have 
\begin{equation}
 \bbE\left[  \left| \frac{A_{t,\gep}}{\phi(t,\gep)}-M'(e^{|\gamma|^2 L} |f|^2)\right|\ind_{\mathcal A_{q,R}} \right]\le \frac{\delta}{2}
\end{equation}
In the integral we can distinguish the contribution from $[0,T]$ from the rest.
We have from \eqref{lespetitst} and the fact that $\phi(t,\gep)$ is bounded from below
\begin{equation}
\suptwo{t\in [0,T]}{\gep\in (0,\gep_0)} \bbE\left[  \left| \frac{A_{t,\gep}}{\phi(t,\gep)}-M'(e^{|\gamma|^2 L} |f|^2)\right|\ind_{\mathcal A_{q,R}}  \right]<\infty.
\end{equation}
As a consequence, since  $\int^{\infty}_{0} \phi(t,\gep)\dd t$ diverges when $\gep\to 0$,   taking $\gep_1$ sufficiently small we have forall $\gep\in(0,\gep_1)$ 
\begin{equation}
 \frac{1}{\int^{\infty}_{0} \phi(t,\gep)} \int^{T}_0  \phi(t,\gep) \bbE\left[   \frac{A_{t,\gep}}{\phi(t,\gep)}-M'(e^{|\gamma|^2 L} |f|^2) \right] \dd t\le \frac{\delta}{2},
\end{equation}
which implies that for $\gep\le \gep_0\wedge \gep_1$ we have 
\begin{equation}
 \bbE\left[ \left|\frac{\langle W^{(\gep)} \rangle_{\infty}}{\int^{\infty}_{0} \phi(t,\gep)\dd t} - M'(e^{|\gamma|^2 L} |f|^2)\right|\ind_{\mathcal A_{q,R}} \right]\le \delta
 \end{equation}
and thus conclude the proof.

\end{proof}

 \subsection{Proof of Proposition \ref{l3333}}
Let us start with the proof of \eqref{lespetitst}.
Noting that $A_{t,\gep}$ is positive,
we have 
\begin{equation}
  \bbE\left[ A_{t,\gep}\right]=\int_{\bbR^{2d}} f(x)\bar f(y) Q_{t,\gep}(x,y) e^{|\gamma|^2 K_{t,\gep}(x,y)} \dd x \dd y.
\end{equation}
We can just use \eqref{samdwich} and bound $Q_{t,\gep}(x,y)$ by $1$ and $ K_{t,\gep}(x,y)$ by 
$T+C$ to conclude.
For the proof of the convergence of $A_{t,\gep}$ we proceed exactly as for the proof of Proposition \ref{zincpoint2}. We assume that $\bar t(t,\gep)> e$ (recall \eqref{pararara}) and set 
\begin{equation}\label{defder}
r=r(t,\gep):=\bar t- \log\log\bar t,
\end{equation}
Setting $K^*_{t,\gep}(x,y):=K_0(x)+\bar K_{t,\gep}(x,y)$,  we 
define
\begin{equation}\begin{split}
\tilde A_{t,\gep}&:=  \int_{\bbR^{2d} } |f(x)|^2 Q_{t,\gep}(x,y)e^{|\gamma|^2 K^*_{t,\gep}(x,y)}    e^{ \sqrt{2d} X_{r}(x)- 2d K_{r}(x)}
 \dd x \dd y\\
\\&= \phi(t,\gep) M^{\sqrt{2d}}_r(e^{|\gamma|^2 L}|f|^2)
\end{split}\end{equation}
 Since $\limtwo{\gep\to 0}{t\to\infty} r(t,\gep)=\infty$, we obtain, as a direct consequence of Lemma \ref{lelemma} that 
 \begin{equation}
\limtwo{\gep\to 0}{t\to\infty} 
\bbE\left[  \left| \frac{\tilde A_{t,\gep}}{\phi(t,\gep)}-M'(e^{|\gamma|^2 L} |f|^2)\right|\ind_{\mathcal A_{q,R}} \right]=0.
 \end{equation}
Our task is thus to prove that 
\begin{equation}
\limtwo{\gep\to 0}{t\to\infty} \phi(t,\gep)^{-1}  \bbE\left[ |\tilde A_{t,\gep}-A_{t,\gep}| \ind_{\mathcal A_{q,R}} \right]=0.
\end{equation}
Like for the proof of \eqref{laprouvait} in the previous section, we proceed in three steps.

 \subsubsection*{Step 1: Changing the deterministic prefactor in the integrand}
 
 Set 
 \begin{equation*}
 A^{(1)}_{t,\gep}:=  \int_{\bbR^{2d} } |f(x)|^2 Q_{t,\gep}(x,y)e^{ \gamma X_{t,\gep}(x)+\bar \gamma X_{t,\gep}(y)- \frac{\gamma^2}{2} K_{t,\gep}(x) - \frac{\bar \gamma^2}{2} K_{t,\gep}(y)+|\gamma|^2\left(K_0(x)-K_{0,\gep}(x,y)\right)}
 \dd x \dd y
 \end{equation*}
 Let us prove that 
 \begin{equation}\label{tepoine}
\limtwo{\gep\to 0}{t\to\infty} \phi(t,\gep)^{-1}  \bbE\left[ |A^{(1)}_{t,\gep}-A_{t,\gep}| \ind_{\mathcal A_{q,R}} \right]=0.
\end{equation}
Let 
 As a direct consequence of the continuity of $f$ and of $K_0$, if one sets 
 \begin{equation}
\suptwo{|x|,|y|\le R}{|x-y|\le e^{t}+2\gep} \left| f(x)\bar f(y)  -|f(x)|^2  e^{ |\gamma|^2\left(K_0(x)-K_{0,\gep}(x,y)\right)}\right|=:\delta(\gep,t),
 \end{equation}
 we have $\limtwo{\gep\to 0}{t\to\infty} \delta(\gep,t)=0$.
 Since $Q_{t,\gep}(x,y)=0$ when $|x-y|\ge e^{t}+2\gep$ repeating the computation \eqref{onlefaitunefois} and using \eqref{zaam} we get 
 \begin{equation}\begin{split}\label{croppy}
 |A^{(1)}_{t,\gep}-A_{t,\gep}| &\le \delta(\gep,t) \int_{B(0,R)^2}Q_{t,\gep}(x,y) e^{ \sqrt{2d} X_{t,\gep}(x)+ (|\gamma|^2-d)K_{t,\gep}(x)}\dd x \dd y \\
 &\le C e^{-d t}\delta(\gep,t) \int_{B(0,R)} e^{ \sqrt{2d} X_{t,\gep}(x)+ (|\gamma|^2-d)K_{t,\gep}(x)} \dd x 
 \end{split}
  \end{equation}
Using Cameron-Martin formula, we obtain (assuming $|x|,|y|\le R$ and $|x-y|\le e^{t}+2\gep$) for a constant $q'>q$
 \begin{equation}\begin{split}\label{beatiful}
\bbE\left[ e^{ \sqrt{2d} X_{t,\gep}(x)-dK_{t,\gep}(x)}\ind_{\cA_{q,R}}\right]&\le \bbE\left[ e^{ \sqrt{2d} X_{t,\gep}(x)+ (|\gamma|^2-d)K_{t,\gep}(x)}\ind_{\cA_{\bar t,q}(x)}\right] \\
  &=  P\left[ \forall s\in [0,\bar t], B_s \le  \sqrt{2d} (s-\bar K_{s,\gep,0}(x)) + q\right]\\
  &\le   P( \forall s\in [0,\bar t], B_s \le q')\le C (\bar t\vee 1)^{-1/2} e^{\bar t|\gamma|^2}
  \end{split}
 \end{equation}
where in the second  inequality we have used \eqref{samdwich} to estimate covariances.
Hence, using \eqref{lasimptotz} we deduce from \eqref{croppy} that 
\begin{equation*}
  \bbE\left[ |A^{(1)}_{t,\gep}-A_{t,\gep}| \ind_{\mathcal A_{q,R}} \right] \le  C \delta(\gep,t)\bar t^{-1/2}  e^{\bar t|\gamma|^2-dt}  \le C' \delta(\gep,t) \phi(t,\gep).
\end{equation*} 
 This concludes the proof of \eqref{tepoine}.
 \qed 
  \subsubsection*{Step 2: Taking conditional expectation}

  We set  
  \begin{equation}\label{a2a2def}
  A^{(2)}_{t,\gep}:= \bbE\left[ A^{(1)}_{t,\gep} \ | \ \cF_r \right]
  \end{equation} 
  and we are going to prove 
   \begin{equation}\label{tepoine2}
\limtwo{\gep\to 0}{t\to\infty} \phi(t,\gep)^{-2}  \bbE\left[ |A^{(2)}_{t,\gep}-A^{(1)}_{t,\gep}|^2 \ind_{\mathcal A_{q,R}} \right]=0.
\end{equation}
  Like what we did in the previous section, we assume here that $K_0\equiv 0$ to simplify the writing (but this does not affect the proof).
   In that case note that since $K_{t,\gep}(x)=  K_{t,\gep}(y)=\bar K_{t,\gep}(x)$, we can factorize the term. We have (recall that $ K_{[r,t],\gep}=  K_{t,\gep}- K_{r,\gep}$)
  
   \begin{equation}\label{wert}
 A^{(2)}_{t,\gep}:=  \int_{\bbR^{2d} } |f(x)|^2 Q_{t,\gep}(x,y)e^{ \gamma X_{r,\gep}(x)+\bar \gamma X_{r,\gep}(y)+(|\gamma^2|-d)K_{r,\gep}(x)
 +|\gamma|^2 K_{[r,t],\gep}(x,y)}
 \dd x \dd y
 \end{equation}
Setting 
\begin{equation}\begin{split}
\zeta(x,y)&:=e^{\gamma X_{t,\gep}(x)+\bar \gamma X_{t,\gep}(y)+(|\gamma^2|-d)K_{t,\gep}(x)},\\
      \bar A^{(1)}_{t,\gep}&:=  \int_{\bbR^{2d} } |f(x)|^2 Q_{t,\gep}(x,y)\zeta(x,y)
 \ind_{\cA_{r,q}(x)}
 \dd x \dd y,\\    
      \bar A^{(2)}_{t,\gep}&:=  \int_{\bbR^{2d} } |f(x)|^2 Q_{t,\gep}(x,y)\bbE[\zeta(x,y)\ | \cF_r ]\ind_{\cA_{r,q}(x)}
 \dd x \dd y            
                \end{split}
\end{equation}
we realize that 
\begin{equation}
 \bbE\left[ |A^{(2)}_{t,\gep}-A^{(1)}_{t,\gep}|^2 \ind_{\mathcal A_{q,R}} \right]\le  \bbE\left[ |\bar A^{(2)}_{t,\gep}-\bar A^{(1)}_{t,\gep}|^2  \right]
\end{equation}
Now we set 
\begin{equation}
\xi_{t,\gep}(x,y):=|f(x)|^2 Q_{t,\gep}(x,y)\left( \zeta(x,y)-\bbE[\zeta(x,y) \ | \cF_r]\right)\ind_{A_{r,q}(x)}.
\end{equation}
We have 
\begin{equation}
 \bbE\left[ |\bar A^{(2)}_{t,\gep}-\bar A^{(1)}_{t,\gep}|^2  \right]\le 
 \int_{\bbR^{4d}} \bbE\left[ \xi(x_1,y_1)\bar \xi(x_2,y_2)\right] \dd x_1\dd x_2 \dd y_1 \dd y_2
\end{equation}
The range of the convariance of $X_{[r,t],\gep}$ is smaller than $e^{-r}+2\gep$, and 
$Q_{t,\gep}(x,y)$ vanishes when $|x-y|\ge e^{-t}+2\gep$. All of this implies that if 
if $|x_1-x_2|\ge 2e^{-r}$ (if $\gep$ is sufficiently small, then $e^{-r}$ is much larger than both $\gep$ and $e^{-t}$ cf. \eqref{pararara})
 then 
 \begin{equation}
   \bbE\left[ \xi(x_1,y_1)\bar \xi(x_2,y_2)  \ | \ \cF_r\right]=0.
 \end{equation}
When $|x_1-x_2|\le 2e^{-r}$ we can use Cauchy-Schwarz to bound the covariance.
We have 
 \begin{equation}
 \bbE\left[ |\xi(x,y)|^2 \right]\le |f(x)|^4 \left(Q_{t,\gep}(x,y)\right)^2
 \bbE[ |\zeta(x,y)|^2\ind_{A_{r,q}(x)}]
 \end{equation}
 and from Cameron-Martin formula, we have, for $|x-y|\le e^{-t}+2\gep$
 \begin{multline*}
   \bbE[ |\zeta(x,y)|^2\ind_{A_{r,q}(x)}]\\
   =e^{2|\gamma|^2 K_{t,\gep}(x)+ 2d K_{t,\gep}(x,y)}
   P\left(\forall s\in [0,r], B_s \le \sqrt{2d}(s-\bar K_{s,\gep,0}(x)-\bar K_{s,\gep,0}(y,x))+q\right)\\
   \le C e^{(|\gamma|^2+d)\bar t}   P\left(\forall s\in [0,r], B_s \le -\sqrt{2d}s+q'\right)
   \le C' e^{\bar t (2|\gamma|^2+d)} r^{-3/2}.
\end{multline*}
where in the first inequality we used \eqref{samdwich} which replace the $K_{t,\gep}$ and $\bar K_{s,\gep}$ by
$\bar t$ and $s$ respectively at the cost of an additive constant and in the second inequality we used Lemma \ref{stupid}.
Altogether we obtain that 
\begin{multline}
 \bbE\left[ |\bar A^{(2)}_{t,\gep}-\bar A^{(1)}_{t,\gep}|^2  \right]
 \le  C e^{\bar t (2|\gamma|^2+d)} r^{-3/2}\\
 \times \int_{\bbR^{4d}} \ind_{\{|x_1-x_2|\le 2e^{-r}\}}|f(x_1)|^2 |f(x_2)|^2 Q_{t,\gep}(x_1,y_1) Q_{t,\gep}(x_2,y_2)\dd \bx \dd \by\\
 \le C' e^{2|\gamma|^2\bar t+ d(\bar t-r)} r^{-3/2} \left(\int_{\bbR^d} Q_{t,\gep}(0,z)\dd z\right)^{2}
 \le C'' e^{d(\bar t-r)} r^{-1/2} \phi(t,\gep)^2.
\end{multline}
We conclude the proof of \eqref{tepoine2} by observing (recall \eqref{pararara}) that 
\begin{equation}
 \limtwo{\gep\to 0}{t\to\infty}  e^{d(\bar t-r)} r^{-1/2}=0.
\end{equation}
\qed

   \subsubsection*{Step 3: Final comparison}
  
  Finally to conclude we need to show that 
  \begin{equation}\label{tepoine3}
   \limtwo{\gep\to 0}{t\to\infty} \phi(t,\gep)^{-2}  \bbE\left[ |A^{(2)}_{t,\gep}-\tilde A_{t,\gep}|^2 \ind_{\mathcal A_{q,R}} \right]=0.
  \end{equation}
We set (recall that $X_{[s,t],\gep}=X_{t,\gep}-X_{s,\gep}$´)
  \begin{equation}\begin{split}
 Z_1(x)&:= \sqrt{2d} X_r(x),\\
 Z_2(x,y)&:= \sqrt{2d} X_{\bar r}(x)+\gamma X_{[\bar r,r],\gep}(x)+ \bar \gamma X_{[\bar r,r],\gep}(y),\\
  Z_3(x,y)&:= \gamma X_{r,\gep}(x)+\bar \gamma  X_{r,\gep}(y).
\end{split}\end{equation}
 The reader can check (using \eqref{a2a2def} rather than \eqref{wert} since the latter assumes $K_0\equiv 0$) that  
\begin{equation}\begin{split}
 \tilde A_{t,\gep}&:= \int_{\bbR^{2d}} |f(x)|^2 Q_{t,\gep}(x,y) e^{|\gamma|^2 K^*_{t,\gep}(x,y)} e^{ Z_1(x) -\frac{1}{2}\bbE[Z_1(x)]} \dd x \dd y,\\
 A^{(2)}_{t,\gep}&:= \int_{\bbR^{2d}} |f(x)|^2 Q_{t,\gep}(x,y) e^{|\gamma|^2 K^*_{t,\gep}(x,y)} e^{ Z_3(x,y) -\frac{1}{2}\bbE[Z_3(x,y)]} \dd x \dd y. 
 \end{split}
\end{equation}
In order to prove \eqref{tepoine3} taking absolute value inside the integrand using Jensen's inequality, 
we just have to obtain a uniform bound on the integrand, that is, to show that 
\begin{equation}
 \limtwo{t\to \infty}{\gep \to 0} \bar t^{1/2} \maxtwo{|x|\le R}{|x-y|\le e^{-t}+2\gep}\bbE\left[  \Big|e^{ Z_1(x) -\frac{\bbE[Z^2_1(x)]}{2}}-  e^{ Z_3(x,y) -\frac{\bbE[Z^2_3(x,y)]}{2}}\Big|\ind_{\cA_{q,R}}\right]=0.
\end{equation}
The restriction for $x$ and $y$ comes from the support of $f$ and $Q_{t,\gep}$ respectively.
To prove \eqref{individix} we start with the decomposition
\begin{multline}\label{decompopo}
 \bbE\left[  \Big|e^{ Z_1(x) -\frac{\bbE[Z^2_1]}{2}}-  e^{ Z_3 -\frac{\bbE[Z^2_3]}{2}}\Big|\ind_{\cA_{q,R}}\right]\\
 \le  \bbE\left[  \Big|e^{ Z_1 -\frac{\bbE[Z_1^2]}{2}}-  e^{ Z_2 -\frac{\bbE[Z_2]}{2}}\Big|\ind_{\cA_{\bar r,q}(x)}\right]+  \bbE\left[ \Big| e^{ Z_3 -\frac{\bbE[Z^2_3]}{2}}- e^{ Z^2_2 -\frac{\bbE[Z^2_2]}{2}}\Big|\right] 
\end{multline}
(the inequality is just the triangle inequality and replacing $\cA_{q,R}$ with a larger event), and show that each term is $o(\bar t^{-1/2})$.
Let us start with the second one. 
Using Lemma \ref{ulikeit}, we have
\begin{equation}\begin{split}
 \bbE&\left[  e^{ Z_3 -\frac{\bbE[Z^2_3]}{2}}- e^{ Z^2_2 -\frac{\bbE[Z^2_2]}{2}}|\right] \le  C\sqrt{\bbE[ |Z_3-Z_2|^2]}\\
  &= C \sqrt{\bbE[|\sqrt{2d} X_{\bar r}(x)-\gamma X_{\bar r,\gep}(x)-   \bar \gamma X_{\bar r,\gep}(y)|^2]}\\
&\le C' \left( \sqrt{\bbE\left[  |X_{\bar r}(x)- X_{\bar r,\gep}(x)|^2\right]} + \sqrt{\bbE\left[ |X_{\bar r}(x)- X_{\bar r,\gep}(y)|^2   \right]}\right)\\
&= C' \left( \left(K_{\bar r}(x)+ K_{\bar r, \gep}(x)- 2 K_{\bar r, \gep,0}(x)\right)^{1/2}+\left( K_{\bar r}(x)+ K_{\bar r, \gep}(x)- 2 K_{\bar r, \gep,0}(y,x)\right)^{1/2}\right)\\
&\le C'' \left(\gep+ |x-y|\right)^{c}\le C' e^{-c \bar t}
\end{split}\end{equation}
where to obtain the last line we have used \eqref{zuum} and the H\"older continuity of $K_0$.
For the  first term in \eqref{decompopo} we factorize the $\mathcal F_{\bar r}$-measureable part  and use independence to obtain 
\begin{multline}\label{sisquarantetroi}
  \bbE\left[  |e^{ Z_1 -\frac{\bbE[Z_1^2]}{2}}-  e^{ Z_2 -\frac{\bbE[Z_2]}{2}}|\ind_{\cA_{\bar r,q}(x)}\right]
  \\=   \bbE\left[  e^{ \sqrt{2d} X_{\bar r}(x)- d K_{\bar r}(x)}\ind_{\cA_{\bar r,q}(x)}\right]   \bbE\left[  |e^{ Z'_1 -\frac{\bbE[(Z'_1)^2]}{2}}-  e^{ Z'_2 -\frac{\bbE[(Z')^2_2]}{2}}|\right],
\end{multline}
where $Z'_i= Z_i - \sqrt{2d} X_{\bar r}(x)$. We have from Cameron-Martin Formula and Lemma \ref{stupid}
\begin{equation}\label{zyrup}
 \bbE\left[  e^{ \sqrt{2d} X_{\bar r}(x)- d K_{\bar r}(x)}\ind_{\cA_{\bar r,q}(x)}\right]=
 \bP\left[ \forall s\in [0,\bar r], B_s\le q \right]\le  C \bar r^{-1/2}.
\end{equation}
This is obviously $O(\bar t^{-1/2})$ so to conclude we only need to show that the other factor in \eqref{sisquarantetroi} goes to zero.  
We also have from Lemma \ref{ulikeit} and \eqref{zuum}
\begin{multline}
  \bbE\left[  |e^{ Z'_1 -\frac{\bbE[(Z'_1)^2]}{2}}-  e^{ Z'_2 -\frac{\bbE[(Z')^2_2]}{2}}|\right]\le  C
 \sqrt{ \bbE\left[  |Z'_1 -Z'_2|^2\right]}\\\le  C \left(  K_{[\bar r,r]}(x)+ K_{[\bar r,r], \gep}(x)+ K_{[\bar r,r],\gep}(y)- 2 K_{[\bar r,r], \gep,0}(x)- 2 K_{[\bar r,r], \gep,0}(y,x)\right)^{1/2}
  \\ \le C' e^{r}(|x-y|+\gep) \le C e^{(r-\bar t)}.
 \end{multline}
 This concludes the proof. \qed
      
   \subsubsection*{The convergence of $B_{t,\gep}$}
      
    To prove the second convergence in \eqref{lesgross}, it is sufficient again to show first that 
    
    \begin{equation}
   \limtwo{t\to \infty}
   {\gep\to 0} \varphi(t,\gep)^{-2}\bbE\left[ | B_{t,\gep}- \bbE[B_{t,\gep} \ | \ \cF_r]|^2 \ind_{\cA_{q,R}} \right] =0
    \end{equation}
    repeating the computation of step two,  and then prove that 
    $$   \limtwo{t\to \infty}{\gep \to 0}  \bbE[ |\bbE[B_{t,\gep} \ | \ \cF_r]|\ind_{\cA_{q,R}}]=0.$$ 
   We leave this part to the reader, since this is very similar to the computation performed at the end of Section \ref{poffff}.
      \qed

\section{Proof of Theorem \ref{zuperzlt}}\label{zuperproof}
\noindent We need to show that for any $H$ bounded and $\cF_{\infty}$-measurable and $\xi\in \bbR$ we have 
\begin{equation}\label{laconv}
 \lim_{n\to \infty} \bbE\left[  H \left(e^{i \xi \frac{W_n}{v(n)}}- e^{-\frac{\xi^2 Z}{2}} \right)\right]=0.
\end{equation}
We first assume that the collection of variables $v(n)^{-2}\langle W_n\rangle_{\infty}$ is uniformly essentially bounded, that is, that there exists $M$ such that for every $n\ge 1$
\begin{equation}\label{essbound}
\bbP\left[ v(n)^{-2}\langle W_n\rangle_{\infty}\ge M\right]=0
 \end{equation}
Note that this implies also that $\bbP\left[ Z\ge M\right]=0$.
 We assume, to simplify notation that $\xi=1$ (this entails no loss of generality).
We set $H_t:= \bbE\left[ H \ | \ \cF_t  \right]$ and $Z_t:=\bbE\left[ Z \ | \ \cF_t\right]$ we have
\begin{multline}
 \bbE\left[  H \left(e^{i  \frac{\xi W_n}{v(n)}}- e^{-\frac{\xi^2 Z}{2}} \right)\right]\\
 =\bbE\left[ H (e^{-\frac{Z}{2}}- e^{-\frac{Z_t}{2}})\right]+
 \bbE\left[  (H-H_t) \left(e^{i  \frac{W_n}{v(n)}}- e^{-\frac{ Z_t}{2}} \right)\right]+ \bbE\left[  H_t \left(e^{i  \frac{W_n}{v(n)}}- e^{-\frac{ Z_t}{2}} \right)\right]
 \\=: E_1(t,n)+E_2(t,n)+E_3(t,n).
\end{multline}
We prove the convergence \eqref{laconv} by showing that for $i=1,2,3$  
\begin{equation}\label{EEE}
\lim_{t\to \infty} \limsup_{n\to \infty}|E_i(t,n)|=0.
\end{equation}
Using the fact that $z\mapsto e^z$ is  $1$-Lipshitz (first line) and has modulus bounded by $1$  (second line) in $\{ z\in \bbC \ : \ \mathfrak{Re}(z)\le 0\}$ we have
\begin{equation}\begin{split}
|E_1(t,n)|&\le  \bbE\left[ |H| \left| e^{-\frac{Z}{2}}- e^{-\frac{Z_t}{2}}\right|\right]\le \frac{\|H\|_{\infty}}{2}\bbE\left[ |Z-Z_t|\right],\\
|E_2(t,n)|&\le \bbE\left[  |H-H_t| \left|e^{i  \frac{W_n}{v(n)}}- e^{-\frac{ Z}{2}} \right|\right]\le  2 \bbE\left[ |H-H_t|\right].
\end{split}\end{equation}
Since $H_t$ and $Z_t$ converge respectively to $H$ and $Z$ in $L^1$, \eqref{EEE} holds for $i=1,2$.
For $i=3$, we observe that for fixed $t$ the process 
$$M^{(n)}_u:= e^{\frac{iW_{n,t+u}-W_{n,t}}{v(n)}+\frac{\langle W_n\rangle_{t+u}- \langle W_n\rangle_{t}}{2 v(n)^2}-\frac{ Z_t}{2}} $$
 is a martingale for the filtration $(\mathcal G_u):= (\cF_{t+u})$, which converges in $L^1$ when $u\to \infty$. In particular we have 
 \begin{equation}
  \bbE\left[ e^{\frac{iW_n-W_{n,t}}{v(n)}+\frac{\langle W_n\rangle_{\infty}- \langle W_n\rangle_{t}}{2 v(n)^2}} \ | \ \cF_t \right]=1. 
 \end{equation}
Multiplying by $H_t e^{-\frac{ Z_t}{2}}$ and taking expectation we obtain that
\begin{equation}
\bbE\left[  H_t e^{\frac{iW_n-W_{n,t}}{v(n)}+\frac{\langle W_n\rangle_{\infty}- \langle W_n\rangle_{t}}{2 v(n)^2}-\frac{ Z_t}{2}}\right] =\bbE\left[  H_t e^{-\frac{ Z_t}{2}}\right]
\end{equation}
Hence we have (using that $\|H_t\|_{\infty}\le \|H\|_{\infty}$)
\begin{multline}\label{therhs}
 |E_3(t,n)|\le \bbE\left[ |H_t| \left| e^{i  \frac{W_n}{v(n)}}- e^{\frac{i(W_n-W_{n,t})}{v(n)}+\frac{\langle W_n\rangle_{\infty}- \langle W_n\rangle_{t}}{2 v(n)^2}-\frac{ Z_t}{2}} \right|  \right]\\ \le \|H\|_{\infty} \bbE\left[ \left| 1- e^{-\frac{iW_{n,t}}{v(n)}+\frac{\langle W_n\rangle_{\infty}- \langle W_n\rangle_{t}}{2 v(n)^2}-\frac{ Z_t}{2}} \right|  \right],
\end{multline}
From  \eqref{toutlesass}
we have the following convergence in probability for any fixed $t$
\begin{equation}
 \lim_{n\to \infty}-\frac{iW_{n,t}}{v(n)}+\frac{\langle W_n\rangle_{\infty}- \langle W_n\rangle_{t}}{2 v(n)^2}-\frac{ Z_t}{2}= \frac{Z-Z_t}{2}. 
\end{equation}
Using assumption \eqref{essbound}, taking the limit in the r.h.s.\ of \eqref{therhs} and using
dominated convergence, we obtain that  
\begin{equation}
 \limsup_{n\to \infty} |E_3(t,n)|\le \|H\|_{\infty} \bbE\left[ \left| 1- e^{\frac{Z-Z_t}{2}}\right|\right].
\end{equation}
Since $Z_t$ converges to $Z$ we can conclude that \eqref{EEE} also holds for $i=3$ using dominated convergence again (both variables are uniformly bounded).

\medskip

\noindent Let us now remove the boundedness assumption. 
Given $A>0$ we set $$T_{A,n}:=\inf\{ t \ : \ v(n)^{-2}\langle W_{n}\rangle_t=A\} \quad  \text{ and } \quad  W^A_n:= W_{n,T_{A,n}}.$$
Note that  $\bbE\left[ W^A_n \ | \ \cF_t\right]= W_{t\wedge T_{A,n}}$ so that (using the notation $\langle W^A_n \rangle_t$ to denote the quadratic variation of this martingale)
we have 
\begin{equation}
 \lim_{n\to \infty} v(n)^{-2}\langle W^A_n \rangle_{\infty} =Z\wedge A.
\end{equation}
Since we have proved \eqref{laconv} under the assumption \eqref{essbound} we know that for every $A>0$
\begin{equation}\label{boundedcase}
\lim_{n\to \infty}   \bbE\left[  H \left(e^{i  \frac{\xi W^A_n}{v(n)}}- e^{-\frac{ \xi^2(Z\wedge A)}{2}} \right)\right] =0
\end{equation}
From the convergence  assumption, we have 
\begin{equation}
\limsup_{n\to \infty} \bbP[T_{A,n}=\infty]\le \bbP\left[ Z\ge A\right]
\end{equation}
and hence 
$$\lim_{A\to  \infty}\liminf_{n\to \infty} \bbP\left[ W^A_n=W_n \right]=\lim_{A\to \infty}\bbP[Z\wedge A=Z]= 1.$$
As a consequence we can conclude using \eqref{boundedcase} that 
\begin{equation}
  \lim_{n\to \infty} \bbE\left[  H \left(e^{i \xi \frac{W_n}{v(n)}}- e^{-\frac{\xi^2 Z}{2}} \right)\right]=\lim_{A\to \infty}   \lim_{n\to \infty}   \bbE\left[  H \left(e^{i  \frac{\xi W^A_n}{v(n)}}- e^{-\frac{ \xi^2(Z\wedge A)}{2}} \right)\right]=0.
\end{equation}

\qed

\noindent
\textbf{Acknowledgements:}  This work was supported by a productivity grant from CNPq and a JCNE grant from FAPERJ.

\appendix

\section{Technical results and their proof}\label{teikos}

\subsection{Standard Gaussian tools}

We first display two standard tools which are used throughout the proof.
The first is the standard Cameron-Martin formula which describes how a Gaussian process is affected by an exponential tilt.

\begin{proposition}\label{cameronmartinpro}

 Let $(Y(z))_{z\in \cZ}$ be a centered Gaussian field indexed by a set $\cZ$. We let  $H$ denote its covariance and $\bP$ denote its law. 
Given $z_0\in \cZ$ let us define $\tilde \bP_{z_0}$ the probability obtained from $\bP$ after a tilt by $Y(z_0)$ that is
\begin{equation}
 \frac{\dd \tilde \bP_{z_0}}{\dd \bP}:= e^{Y(z_0)- \frac{1}{2} H(z_0,z_0)}
\end{equation}~
Under $\tilde \bP_{z_0}$, $Y$ is a Gaussian field with covariance $H$,
and mean  $\tilde \bbE_{z_0}[ Y(z)]=H(z,z_0).$
\end{proposition}

\noindent The second is a bound on the probability for a Brownian Motion to remain below a line.
Both estimates can be proved directly using the reflexion principle.

\begin{lemma}\label{stupid}
 
 Let $B$ be a standard Brownian Motion and let $P$ denote its distribution, setting $ \mathfrak g_{t}(a):= \int^{u_+}_0 e^{-\frac{z^2}{2t}} \dd z.$
 we have
 \begin{equation}
  P\left[ \sup_{s\in[0,t]} B_s \le a \right]=\sqrt{ \frac{2\pi}{t}}\mathfrak g_{t}(a)\le \sqrt{\frac{2\pi}{t}}a.
 \end{equation}
 Additionally for any $a,b> 0$ there exists $C_{a,b}$ such that f
  \begin{equation}
  P\left[ \sup_{s\in[0,t]} (B_s+bs) \le  a  \right]= \frac{1}{\sqrt{2\pi t}} \int e^{-\frac{u^2}{2t}}(1-e^{\frac{2a(a+u-bs)_+}{t}})\dd u \le C_{a,b} e^{-\frac{b^2t}{2}}t^{-3/2}.
 \end{equation}
%
%

\end{lemma}

\subsection{Comparing exponentiated Gaussians}

In or comparison of partition functions 

\begin{lemma}\label{ulikeit}
 
 Consider $(X_1,X_2,Y_1,Y_2)$ an $\bbR^4$ valued centered Gaussian vector 
 and set $X:=X_1+iX_2$ and $Y=Y_1+iY_2$.
 Assuming that 
 \begin{equation}
  \bbE[X^2_2]\le 1 \text{ and } \bbE[|X-Y|^2]\le 1
 \end{equation}
then there exits a constant $C$ such that
 \begin{equation}
  \bbE\left[ \big| e^{X-\frac{1}{2}\bbE[X^2]}-  e^{Y-\frac{1}{2}\bbE[Y^2]}\big|\right] \le C \bbE\left[ |X-Y|^2 \right]
 \end{equation}

 \end{lemma}

\begin{proof}
   We factorize $e^{X-\frac{1}{2}\bbE[X^2]}$, use the Cameron-Martin formula and rearrange the expectation terms in the exponential, we obtain 
  \begin{equation}\begin{split}
 \bbE\left[ \big|  e^{Y-\frac{\bbE[Y^2]}{2}}- e^{X-\frac{\bbE[X^2]}{2}}\big|\right]&=\bbE\left[ e^{X_1+\frac{\bbE[X^2_2]-\bbE[X^2_1] }{2}}\big|e^{Y-X+ \frac{\bbE[X^2]-\bbE[Y^2]}{2}}- 1 
  \big| \right]\\
  &=e^{\frac{\bbE[X^2_2]}{2}}\bbE\left[ \big|
  e^{Y-X- \frac{\bbE[(X-Y)^2]}{2}-i\bbE[X_2(Y-X)]}-1 \big| \right].
 \end{split} \end{equation}
 The prefactor is bounded (by assumption) by $e^{1/2}$. For the rest, setting $Z=Y-X$ (and letting $Z_1$ and $Z_2$ denote the real and imaginary part) we have using the triangle inequality
 \begin{equation}
  \bbE\left[ \big|  e^{Z- \frac{\bbE[Z^2]}{2}-i\bbE[X_2 Z]}-1 \big| \right]\le \big|  e^{i\bbE[X_2 Z]}-1  \big| \bbE\left[ \big|e^{Z- \frac{\bbE[Z^2]}{2}}\big|\right]+   \bbE\left[ \big|  e^{Z- \frac{\bbE[Z^2]}{2}}-1 \big| \right].
 \end{equation}
 For the first term, using that  $|\bbE[X_2 Z]|\le \sqrt{\bbE[X^2_2]\bbE[ |Z|^2]} \le \sqrt{\bbE[|Z|^2]}\le 1$, and that $|e^u-1|\le e|u|$ for $u\le 1$ and computing expectation, we obtain that 
 \begin{equation}
 \big| e^{i\bbE[X_2 Z]}-1  \big| \bbE\left[ \big|e^{Z- \frac{\bbE[Z^2]}{2}}\big|\right]\le e \sqrt{\bbE[|Z|^2]} e^{\frac{\bbE[Z^2_2]}{2}}\le e^{3/2} \sqrt{\bbE[|Z|^2]}.
  \end{equation}
For the second term we  have (using again $|e^u-1|\le e |u|$)
\begin{equation}
   \bbE\left[ \big|  e^{Z- \frac{\bbE[Z^2]}{2}}-1 \big| \right]\le \sqrt{   \bbE\left[ \big|  e^{Z- \frac{\bbE[Z^2]}{2}}-1 \big|^2 \right]}
   =\sqrt{ e^{\bbE[|Z|^2]}-1}\le e^{1/2} \sqrt{\bbE[|Z|^2]},
\end{equation}
which yields the desired result for $C= e^2+e$.
\end{proof}

\subsection{Proof of Lemma \ref{replacement}}\label{prepla}

Let us first compute the order of magnitude of $\phi(t)$.
Let us set for practical purpose $\bar \phi(t):= \int  Q_t(0,z) e^{|\gamma|^2 \bar K_t(0,z)} \dd z$.
Using \eqref{samdwich} (recall that $|z|\le e^{-t}$ on the integrand) and \eqref{zaam} we have
\begin{equation}\label{aaazimp}
\bar \phi(t)\asymp   e^{(|\gamma|^2-d)t}  \quad   \text{ and } \quad \phi(t)\asymp  t^{-1/2}e^{(|\gamma|^2-d)t}
\end{equation}
As a consequence when $|\gamma|^2>d$ most of the integral is carried by $[t-\sqrt{t},t]$ and  we have
  \begin{equation}\begin{split}\label{lestrucsdefous}
\int^t_0 \phi(s)\dd s&=(1+o(1))\int^t_{t-\sqrt{t}} \phi(s)\dd s\\& = (1+o(1))\int^t_{t-\sqrt{t}} \sqrt{\frac{s\vee1}{t}}\phi(s)\dd s 
\\&=(1+o(1))\sqrt \frac{2}{\pi t} e^{|\gamma|^2 \mathfrak j}  \int^t_{0} \bar \phi(s)\dd s.
 \end{split} \end{equation}
 We observe that 
 $$|\gamma|^2 Q_s(0,z) e^{|\gamma|^2 \bar K_s(0,z)}=\partial_s\left(e^{|\gamma|^2\bar K_s(0,z)} \right).$$
Using Fubini  
and integrating w.r.t.\ time and making a change of variable we have
  \begin{equation}\begin{split}\label{compiout}
|\gamma|^2 \int^t_{0} \bar \phi(s)\dd s&=   \int_{\bbR^d}  \left(e^{|\gamma|^2 \bar K_t(0,z)}-1\right) \dd z\\& =  e^{(|\gamma|^2-d)t}\int_{\bbR^d}  \left(e^{|\gamma|^2 \left(\bar K_t(0, e^{-t}z)-t\right)}- e^{-|\gamma|^2 t}\right) \dd z.
 \end{split} \end{equation}
The integrand in the second line is bounded above by $(|z|\vee 1)^{-|\gamma|^2}$. This is obvious for $|z|\ge e^t$ since the integrand vanishes, and when $|z|\le e^{t}$ this can be obtainded from \eqref{defbarl}. Furthermore it converges to $e^{|\gamma|^2 \bar \ell(z)}$ and we obtain using dominated convergence that 
\begin{equation}
\lim_{t\to \infty} \frac{ |\gamma|^2 \int^t_{0} \bar \phi(s)\dd s}{ e^{(|\gamma|^2-d)t}}= \int_{\bbR^d}  e^{|\gamma|^2  \bar \ell(z)} \dd z,
\end{equation}
which combined with \eqref{lestrucsdefous}, proves the lemma in the case $|\gamma|^2>d$.
When $|\gamma|^2=d$, we observe that using,  as in \eqref{compiout}, a change of variable and dominated convergence, we 
have 
\begin{equation}\label{sh}
 \lim_{s\to \infty} \bar\phi(s)= \int_{\bbR^d} \kappa(e^{\frac{\eta_1}{\eta_2}}z) e^{d \bar \ell(z)} \dd z.
\end{equation}
On the other hand we have from \eqref{compiout}
\begin{equation}
 \int^t_{0} \bar \phi(s)\dd s=  \int_{\bbR^d}  \left(e^{d \left(\bar K_t(0, e^{-t}z)-t\right)}- e^{-dt}\right) \dd z
\end{equation}
We have, for $1\le |z|\le e^t$, 
\begin{equation}\begin{split}
\bar K_t(0, e^{-t}z)&= \bar K(0,e^{-t}z)=K(0,e^{-t}z)-K_0(0,e^{-tz})\\
&= t+ \log \frac{1}{|z|}+ (L-K_0)(0,e^{-tz}).
\end{split}\end{equation}
Since $(L-K_0)(0,e^{-t}|z|)= -\mathfrak j+ \delta\left(e^{-t}|z|\right)$ where $\delta(u)$ tends to zero when $u\to 0$  we can deduce that 
\begin{equation}\begin{split}
 \int_{\bbR^d}  \left(e^{d \left(\bar K_t(0, e^{-t}z)-t\right)}- e^{-dt}\right) \dd z
 &=(1+o(1)) \int \ind_{\{1\le |z| \le e^t\}}  e^{d \left(\bar K_t(0, e^{-t}z)-t\right)} \dd z\\
 &=(1+o(1)) e^{-d\mathfrak j} \int \ind_{\{1\le |z| \le e^t\}} |z|^{-d} \dd z\\
 &= (1+o(1)) e^{-d\mathfrak j} \Sigma_{d-1} t.
\end{split}\end{equation}
Since $\bar\phi(s)$ converges, we deduce that its limit equals its Cesaro limit and thus
\begin{equation}
  \lim_{s\to \infty} \bar\phi(s)= e^{-d\mathfrak j} \Sigma_{d-1},
\end{equation}
which implies in turn that
\begin{equation}
\int^t_0  \sqrt{\frac{2}{\pi(s\wedge 1)}} e^{d\mathfrak j} \bar\phi(s)=(1+o(1))2 \sqrt{\frac{2t}{\pi}} \Sigma_{d-1},
\end{equation}
and concludes the proof of the lemma. \qed

\subsection{Proof of Lemma \ref{replacement2}}\label{prepla2}

Let us again start with the case $|\gamma|^2>d$. As in the proof of Lemma \ref{replacement}, 
we can compute the asymptotic of $\phi(t,\gep)$ (when $t$ and $\gep$ goes to infinity and zero respectively)
\begin{equation}\label{lasimptotz}
 \phi(t,\gep)\asymp \bar{t}^{-1/2} e^{|\gamma|^2 \bar t-dt}.
\end{equation}
Since $ \suptwo{t\in[0,T]}{\gep\in(0,1)}\phi(t,\gep)<\infty$ for every finite $T$, this implies that the integral $\int^\infty_0 \phi(t,\gep)$ is mostly carried by values of $s$ around $\log(1/\gep)$ (say $\pm \sqrt{\log (1/\gep)}$). For this reason, we can replace the term $(\bar t\vee 1)^{-1/2}$ by $(\log 1/\gep)^{-1/2}$.
\begin{equation}\label{comiccio}
|\gamma|^2 \int^\infty_0 \phi(t,\gep)\dd t=  (1+o(1))\sqrt{\frac{2}{\pi (\log 1/\gep)}} e^{|\gamma|^2 \mathfrak j} 
\int^{\infty}_0\int_{\bbR^{d} }  |\gamma|^2e^{|\gamma|^2 \bar K_{t,\gep}(0,z)} Q_{t,\gep}(0,z) \dd z 
\end{equation}
Using Fubini and integrating with respect to time as in \eqref{compiout} we have 
\begin{equation}
\int^{\infty}_0\int_{\bbR^{d} }  |\gamma|^2e^{|\gamma|^2 \bar K_{t,\gep}(0,z)} Q_{t,\gep}(0,z) \dd z = \int_{\bbR^{d} }\left( e^{|\gamma|^2 \bar K_{\gep}(0,z)} -1\right) \dd z
\end{equation}
We then perform a change of variable for $z$
\begin{equation}
 \int_{\bbR^{d} }  \left( e^{|\gamma|^2 \bar K_{\gep}(0,z)} -1\right) \dd z= \gep^{d-|\gamma|^2} \int_{\bbR^{d} }  \left( e^{|\gamma|^2 \left(\bar K_{\gep}(0,\gep z)+ \log (\gep)\right)} - \gep^{|\gamma|^2}\right)\dd z
\end{equation}
Next we observe that 
$$\bar K_{\gep}(0,\gep z)+ \log (\gep)=\ell_{\theta}(z)+ \left(L_{\gep}(0,\gep z)-K_\gep(0,\gep z)\right).$$
Using dominated convergence as $\gep$ goes to zero (the integrand is bounded above by $(|z|\vee 1)^{-|\gamma|^2}$ and recalling \eqref{ladefdejota}
we obtain that 
\begin{equation}\label{fino}
 \lim_{\gep \to 0}\int_{\bbR^{d} }  \left( e^{|\gamma|^2 \left(\bar K_{\gep}(0,\gep z)+ \log (\gep)\right)} - \gep^{|\gamma|^2}\right)\dd z
 = e^{-|\gamma|^2 \mathfrak j} \int_{\bbR^d} e^{-|\gamma|^2 \ell_{\theta}(z) }\dd z.
\end{equation}
The combination of \eqref{comiccio}-\eqref{fino} concludes the proof  in the case $|\gamma|^2>d$.
For the case $|\gamma|^2=d$, based on \eqref{lasimptotz} we know that setting $T_{\gep}=\log (1/\gep)-\sqrt{\log (1/\gep)}$

\begin{equation}
\int^\infty_0 \phi(t,\gep) \dd t =(1+o(1))\int^{T_{\gep}}_0 \phi(t,\gep) \dd t.
\end{equation}
Now in this range for $t$ it is tedious but not difficult to check that
\begin{equation}
 \lim_{\gep \to 0} \sup_{t\in[0, T_{\gep }]}\frac{\int_{\bbR^d}e^{|\gamma|^2 \bar K_{t,\gep}(0,z)} Q_{t,\gep}(0,z) \dd z   }{\int_{\bbR^d} e^{|\gamma|^2 \bar K_{t}(0,z)} Q_{t}(0,z) \dd z}=1.
\end{equation}
From this we obtain that 
\begin{equation}
 \int^\infty_0 \phi(t,\gep) \dd t =(1+o(1))\int^{T_{\gep}}_0 \phi(t,\gep) \dd t= (1+o(1))\int^{T_{\gep}}_0 \phi(t) \dd t
\end{equation}
and we can conclude using Lemma \ref{replacement}.

\section{The convergence of $M^{\gamma}_{\gep}$ as a distribution}\label{sobolev}

We have chosen for simplicity, to present our convergence results as convergence of a collection of random variables $M^{\gamma}_{\gep}(f)$ indexed by $C_c(\bbR^d)$. We can go further and prove that 
$M^{\gamma}_{\gep}(\cdot)$ converges as a distribution. For this purpose we need to recall the definition of local Sobolev/Bessel spaces.

\medskip

\noindent The  Bessel space   $H^{s,p}(\bbR^k)$, $s\in \bbR$ and $p\in [1,\infty]$ on $\bbR^k$ 
is defined by
\begin{equation}
 H^{s,p}(\bbR^k):= \{ \varphi\in \mathcal D'(\bbR^k) : \ (1+|\xi|^2)^{s/2} \hat \varphi(\xi)\in L^{p}(\bbR^k)\}
\end{equation}
where  $\mathcal D'(\bbR^k)$ is the space of distribution and $\hat \varphi(\xi)$ is the Fourier transform of $\varphi$  defined for $\varphi\in C^{\infty}_c(\bbR^k)$ by
$ \hat \varphi(\xi)= \int_{\bbR^{k}} e^{i\xi x} \varphi(x)\dd x.$
It is a Banach space when equiped with the norm
\begin{equation}
 \| f  \|_{H^{s,p}}=  \int_{\bbR^k}(1+|\xi|^2)^{ps/2} |\hat \varphi(\xi)|^{p} \dd \xi
\end{equation}
For $U\subset \bbR^k$ open, the local Bessel space $H^{s,p}_{\mathrm{loc}}(U)$ denotes the set of distributions which belongs to  $H^{s,p}(U)$ after multiplication by an arbitrary smooth function with compact support
\begin{equation}\label{locsob}
 H^{s,p}_{\mathrm{loc}}(U):= \left\{  \varphi\in \mathcal D' (U)  \ | \ \rho\varphi\in
 H^{s,p}(\bbR^d) \text{ for all } \rho\in C^{\infty}_c(U)\right\},
\end{equation}
where above $\rho\varphi$ is identified with its extension by zero on $\bbR^k$.
It is equiped with the topology generated by the family of seminorms $r_{\rho}$, $\rho\in  C^{\infty}_c(U)$ defined by $r_{\rho}(\varphi):=  \| \varphi\rho  \|_{H^{s,p}}$.
In the particular case where $p=2$ we write  $H^{s}(\bbR^k):=H^{s,2}(\bbR^k)$ which is a Hilbert space (and use the same convention  for the local spaces). The convergence result for $M^{\gamma}_{\gep}(\cdot)$ as a distribution for $\gamma \in \cP_{\mathrm{I/II}}$ is the following.

\begin{theorem}\label{mainalldis}
 If $X$ is a centered Gaussian field whose covariance kernel $K$  has an almost star-scale invariant part, $\gamma \in \cP_{\mathrm{I/II}}$, $p\in [1,\sqrt{2d}{\alpha})$ and $s<-\frac{d}{p}$,
then there exists  $M^{\gamma}_{\infty}\in H^{s,p}_{\mathrm{loc}}(\bbR^d) $ such that for every 
$\rho \in C^{\infty}_c(\bbR^d)$
\begin{equation}\label{convexx}
\lim_{\gep \to 0} \bbE\left[   \|M^{\gamma}_{\gep}(\rho \ \cdot)- M^{\gamma}_\infty(\rho \ \cdot)\|^p_{H^{s,p}}\right]=0   .
 \end{equation}
 In particular $M^{\gamma}_{\gep}$ converges to $M^{\gamma}_\infty$ in probability in the  $H^{s,p}_{\mathrm{loc}}(\bbR^d)$ topology.
 \end{theorem}

 \noindent Similarly in  $\cP'_{\mathrm{II/III}}$ the convergence in law holds also for the distribution.

 \begin{theorem}\label{finalfrontierdis}
Let $X$ be a Gaussian random field with an almost star-scale covariance.
Then given $\gamma \in \cP'_{\mathrm{II/III}}$, and $s<-\frac{d}{2}$
the following joint convergence in law for the  $H^{s}_{\mathrm{loc}}(\bbR^d)$ topology
\begin{equation}
\left(X, \frac{M^{\gamma}_{\gep}}{v(\gep,\theta,\gamma)}\right) \stackrel{\gep \to 0}{\Rightarrow} (X,  \mathfrak M^{\gamma}),
\end{equation}

\end{theorem}

\begin{rem}
 In the proof of Theorems \ref{mainalldis} and \ref{finalfrontierdis} presented below, we are going to assume that
our probability space contains a martingale sequence $(X_t)_{t\ge 0}$ of fields with covariance \eqref{covofxt} approximating $X$.
For reasons analogous to the one exposed at the beginning of Section \ref{pofmainall} this entails no loss of generality.
\end{rem}

\subsection{The case of  Theorem \ref{mainalldis}}

Let us  fix $\rho\in C^{\infty}_c(\bbR^d)$. We want to prove that $M^{\gamma}_{\gep}(\rho \ \cdot)$ converges   in $H^{s,p}(\bbR^d)$.
We first define the limit point.
 We set (without underlying the dependence in $\rho$ to keep the notation light)
\begin{equation}\begin{split}\label{furry}
 \hat M^{\gamma}_{\gep}(\xi)&= \int_{\bbR^{d}} \rho(x)e^{i\xi.x}  e^{\gamma X_{\gep}(x)-\frac{\gamma^2}{2}K_{\gep}(x)} \dd x,\\
 \hat M^{\gamma}_{\infty}(\xi)&=\lim_{\gep\to 0}  \hat M^{\gamma}_{\gep}(\xi).
\end{split}\end{equation}
We let $M^{\gamma}_\infty(\rho \ \cdot)$ denote the random distribution whose Fourier transform is given by  $\hat M^{\gamma}_{\infty}(\xi)$.
The proof of \eqref{convexx}, implies that $M^{\gamma}_\infty(\rho \ \cdot)\in H^{s,p}(\bbR^d)$ with probability one. To prove \eqref{convexx}, note that we have
\begin{equation}
  \bbE\left[ \|M^{\gamma}_{\gep}(\rho \ \cdot)- M^{\gamma}_{\infty}(\rho \ \cdot)\|^p_{H^{s,p}}\right]= \int_{\bbR^d} \bbE\left[ | (\hat M^{\gamma}_{\gep}- \hat M^{\gamma}_{\infty})(\xi)|^p\right] (1+|\xi|^2)^{\frac{ps}{2}}\dd \xi.
\end{equation}
Hence with our assumption on $s<-\frac{d}{p}$ it is sufficient to prove that 
\begin{equation}\begin{split}\label{aprouver}
\limsup_{\gep\to 0}\sup_{\xi \in \bbR^d} \bbE\left[ | (\hat M^{\gamma}_{\gep}- \hat M^{\gamma}_{\infty})(\xi)|^p\right]&<\infty,\\
\forall \xi \in \bbR^d, \ \lim_{\gep\to 0} \bbE\left[ | (\hat M^{\gamma}_{\gep}- \hat M^{\gamma}_{\infty})(\xi)|^p\right]&=0.
\end{split}\end{equation}
and we can then conclude using dominated convergence (the first line yields the domination).
The second line is simply \eqref{stableconv} with $f(x)=\rho(x)e^{i\xi. x}$.
For the first line it sufficient to prove that 
$$\limsup_{\gep\to 0}\sup_{\xi \in \bbR^d} \bbE\left[ | \hat M^{\gamma}_{\gep}(\xi)|^p\right]<\infty,$$ since the bound for $\hat M^{\gamma}_{\infty}(\xi)|$ can then be obtained by Fatou.
We set $V^{(\gep)}_t(\xi)=\bbE[ \hat M^{\gamma}_{\gep}  \ | \ \cF_t]$.
Using the BDG inequality we have 
\begin{equation}
 \bbE\left[ |\hat M^{\gamma}_{\gep}(\xi)|^p\right]\le C \bbE\left[  |V^{(\gep)}_0(\xi)|^p+
 \langle V^{(\gep)}(\xi)\rangle^{p/2}_{\infty}\right].
\end{equation}
Now we have  (recall that $p<\sqrt{2d}/\alpha\le 2$) 
\begin{equation}
 \bbE\left[  |V^{(\gep)}_0(\xi)|^2\right]
 = \int_{\bbR^{2d}} \rho(x)\rho(y) e^{i\xi.(x-y)} e^{|\gamma|^2 K_{0,\gep}(x,y)}\dd x \dd y 
\end{equation}
and we can conclude by replacing $e^{i\xi.(x-y)}$ by $1$ and  observing that since $K$ is continuous $K_{0,\gep}$ is uniformly bounded for $x$, $y$ in the support of $\rho$ and $\gep\in(0,1)$. For the quadratic variation part, 
we have 
\begin{equation}
 \langle V^{(\gep)}\rangle_\infty=|\gamma|^2\int^{\infty}_{0} A_{t,\gep}(\xi) \dd t,
\end{equation}
where, 
\begin{multline}
A_{t,\gep}(\xi):=\int_{\bbR^{2d}}\rho(x)\rho(y) Q_{t,\gep}(x,y) e^{i\xi(x-y)} e^{\gamma X_{t,\gep}(x)+\bar \gamma X_{t,\gep}(y)- \frac{\gamma^2}{2}K_{t,\gep}(x)-\frac{\bar \gamma^2}{2} K_{t,\gep}(y)}\dd x\dd y
\\ \le \int_{\bbR^{2d}}\rho(x)\rho(y)  Q_{t,\gep}(x,y) e^{\alpha\left( X_{t,\gep}(x)+ X_{t,\gep}(y)\right)+ \frac{\beta^2-\alpha^2}{2}\left(K_{t,\gep}+K_{t,\gep}(y)\right)}\dd x\dd y=: \bar A_{t,\gep},
\end{multline}
the inequality being obtain by taking the modulus of the integrand.
To conclude, we just need to prove that 
\begin{equation}
\sup_{\gep\in(0,1)} \bbE\left[\left( \int^{\infty}_0\bar A_{t,\gep}\dd t\right) ^{p/2} \right] <\infty
\end{equation}
For this part we can just repeat the computations made to prove \eqref{pprr2} in Section \ref{pofmainall}.

\subsection{The case of Theorem \ref{finalfrontierdis}}

Since the convergence of finite dimensional marginal has been established, we only need to prove 
tightness of the distribution of  $v(\gep,\theta,\gamma)^{-1}M^{\gamma}_{\gep}( \rho \  \cdot)$ in $H^{s}(\bbR^d)$ for every $\rho$.
For this, we  simply replicate the strategy presented in \cite{MR4413209}, with a minor twist. Since in our case, the Fourier transform in not in $L^2$, we need to consider a restriction to the event $\cA_{q,R}$ where $R$ is such that the support $\rho$ is contained in $B(0,R)$ (recall \eqref{lezaq}).
Keeping the notation introduced in \eqref{furry} for the Fourier transform, we are going to prove the following analogue of \cite[Lemma B.2]{MR4413209} (we use the notation $v(\gep)$ for $v(\gep,\theta,\gamma)$ for ease of reading)

\begin{lemma}
If the support of $\rho$ is included in $B(0,R)$ then 
the following holds for every $a\in \bbR^d$ with a constant $C$ which depends on $\rho$.
\begin{equation}\begin{split}
 \suptwo{\gep\in(0,1)}{\xi\in \bbR^{d}}\bbE\left[ v(\gep)^{-2} |\hat M^{\gamma}_{\gep}(\xi)|^2\ind_{\cA_{q,R}} \right]&<\infty,\\
  \suptwo{\gep\in(0,1)}{\xi\in \bbR^{d}}\bbE\left[ v(\gep)^{-2} |\hat M^{\gamma}_{\gep}(\xi+a)-\hat M^{\gamma}_{\gep}(\xi)|^2\ind_{\cA_{q,R}} \right] &\le C |a|^2,\\
 \end{split}\end{equation}
\end{lemma}

\begin{proof}
We introduce a martingale whose limit coincides with $\hat M^{\gamma}_{\gep}(\xi)$ on the event $\cA_{q,R}$.
Given $x\in \bbR^d$ and $q>0$ we set 
\begin{equation}\label{deftqx}
 T_{q}(x):=\inf\{ t> 0 \ : \ \bar X_t(x)=\sqrt{2d}t+q \}, 
\end{equation}
and define 
\begin{equation}
 N^{(\gep)}_t(\xi):= \int_{\bbR^{d}} \rho(x)e^{i\xi.x}  e^{\gamma X_{t\wedge T_q(x),\gep}(x)-\frac{\gamma^2}{2}K_{t\wedge T_q(x),\gep}(x)} \dd x.
\end{equation}
Since  $T_{q}(x)=\infty$ for all $x$ in the support of $\rho$ on the event $\cA_{q,R}$, we have 
\begin{equation}
  N^{(\gep)}_\infty(\xi)\ind_{\cA_{q,R}}   = \hat M^{\gamma}_{\gep}(\xi)\ind_{\cA_{q,R}}. 
\end{equation}
Hence it is sufficient to prove 
\begin{equation}\begin{split}
 \suptwo{\gep\in(0,1)}{\xi\in \bbR^{d}}\bbE\left[ v(\gep)^{-2} |N^{(\gep)}_\infty(\xi)|^2 \right]&<\infty,\\
  \suptwo{\gep\in(0,1)}{\xi\in \bbR^{d}}\bbE\left[ v(\gep)^{-2} |N^{(\gep)}_\infty(\xi+a)-N^{(\gep)}_\infty(\xi)|^2 \right] &\le C |a|^2.
 \end{split}\end{equation}
 Let us prove the only second inequality, since the first one is only easier.
We set for simplicity $W_t:= N^{(\gep)}_t(\xi+a)-N^{(\gep)}_t(\xi)$. We have $$
\bbE\left[|N^{(\gep)}_\infty(\xi+a)-N^{(\gep)}_\infty(\xi)|^2 \right]= \bbE[|W_\infty|^2]=\bbE[|W_0|^2]+\bbE\left[ \langle W \rangle_\infty \right].$$
We are going to prove a bound for each of the term in the r.h.s.\ . We have
\begin{equation}\label{doublezero}\begin{split}
 \bbE[|W_0|^2]&= \int_{\bbR^{2d}} \rho(x)\rho(y)\left(e^{i(\xi+a).x}- e^{i\xi.x}\right) \left(e^{-i(\xi+a).y}- e^{i\xi.y}\right)e^{|\gamma|^2 K_{0,\gep}(x,y)}\\
 &\le C |a|^2  \int \rho(x)\rho(y) |x||y| \dd x \dd y \le C'|a|^2.
\end{split}\end{equation}
where in the second line have taken the modulus of the integrand, and used the fact that 
the complex exponential is Lipshitz. 
To bound the expected value of the quadratic variation,
using Itô calculus, and observing that $\{T_q(x)< t\}=\cA_{t,q}(x)$ (recall \eqref{lezaq}) we obtain that 
\begin{equation}
  \langle W \rangle_\infty=|\gamma|^2 \int^{\infty}_0 U_t \dd t.
\end{equation}
where
\begin{multline}
  U_t:=\int_{\bbR^{2d}} \rho(x)\rho(y) Q_{t,\gep}(x,y)\left(e^{i(\xi+a).x}- e^{i\xi.x}\right) \left(e^{-i(\xi+a).y}- e^{i\xi.y}\right) \\ \times 
 e^{\gamma X_{t,\gep}(x)+\bar \gamma X_{t,\gep}(y)- \frac{\gamma^2}{2} K_{t,\gep}(x)- \frac{\bar \gamma^2}{2}K_{t,\gep}(y)}
 \ind_{\cA_{t,q}(x)\cap \cA_{t,q}(y)}\dd x .
\end{multline}
Taking the modulus in the integrand value everywhere inside the integral and
using the fact that the complex exponential is  Lipshitz fwe obtain 
 \begin{multline}
U_t \le  |a|^2 \int_{\bbR^{2d}} \rho(x)\rho(y) |x| |y| Q_{t,\gep}(x,y)\\
\times
 e^{\sqrt{d/2}(X_{t,\gep}(x)+ X_{t,\gep}(y))+ \frac{|\gamma|^2-d}{2} (K_{t,\gep}(x)+K_{t,\gep}(y))}  \ind_{\cA_{t,q}(x)\cap \cA_{t,q}(y)} \dd x \dd  y\\
 \le C |a|^2  \int_{\bbR^{2d}} \rho(x)^2 |x|^2Q_{t,\gep}(x,y) e^{\sqrt{2d}X_{t,\gep}(x)+ (|\gamma|^2-d)K_{t,\gep}(x)} \ind_{\cA_{t,q}(x)} \dd x\dd y,
 \end{multline}
 where the second line is obtained via the same step as \eqref{onlefaitunefois} ($ab\le a^2+b^2 /2$ and symmetry and in $x$ and $y$).
 Now recalling \eqref{beatiful} we have
 \begin{equation}\label{topz}
  \bbE\left[ e^{\sqrt{2d}X_{t,\gep}(x)-dK_{t,\gep}(x)} \ind_{\cA_{t,q}(x)}  \right]
\le C (\bar t\vee 1)^{-1/2}.
 \end{equation}
where to obtain the first  inequality, we used \eqref{samdwich} to show that $K_{s}(x,y)$ (and all similar terms) are well estimated by $\bar t$
for $s\in[0,\bar t]$.
Now we have 
\begin{equation}\begin{split}
 \bbE\left[ U_t\right] \le \frac{ C |a|^2 e^{-dt} }{\sqrt{t\vee 1}}\int_{\bbR^{2d}}  \rho(x)^2 |x|^2 Q_{t,\gep}(x,y) e^{|\gamma|^2 K_{t,\gep}(x,y)}
 \dd x
 \le C' |a|^2 \phi(t,\gep).
\end{split}\end{equation}
After integrating with respect to $t$ (recalling Lemma \ref{replacement2}) we obtain that 
\begin{equation}
\bbE[\langle W\rangle_{\infty}]\le C |a|^2 v(\gep)^2,
\end{equation}
for a constant $C$ which is independent of $\gep$ and $\xi$ and $a$, which combined with \eqref{doublezero}, concludes the proof.
\end{proof}

\section{Beyond star-scale invariance}\label{beyondstar}

The assumption that the kernel can be written in the form \eqref{iladeuxstar} may be felt as unnecessarily restrictive, since after all, 
given an open domain $\cD\subset \bbR^d$ and a positive definite Kernel  kernel $K: \ \cD^2\to (-\infty,\infty]$ that admits a decomposition of the form \eqref{fourme},  the mollified field $X_{\gep}$ can be defined on 
$$ \cD_{\gep}:= \{ x\in \cD \ :  \inf_{y\in \cD^{\cc}}|x-y| >2\gep  \}.$$
More precisely in that case the field $X$ is indexed by $C_c(\cD)$ the set of functions with compact support on $\cD$ (in \eqref{hatK}, $\bbR^{2d}$ is replaced by $\cD^2$), and $X_{\gep}$ remains defined by \eqref{convolulu} (here $\theta_{\gep}(x-\cdot)$, which for $x\in \cD_{\gep}$, has its support included in $\cD$, is identified with its restriction on $\cD$).

\medskip

It turns out that our results can be extended to the the general setup described above, only with an additional regularity assumption concerning the function $L$ present in \eqref{fourme}.
Given $U\subset \cD$, we say that the restriction of $K$ to $U$   \textit{has an almost star-scale invariant part}, if 
\begin{equation}\label{4star}
 \forall x,y\in U, \   K(x,y)=K_0(x,y)+ \bar K(x,y)
\end{equation}
where $\bar K$ is an almost-star scale invariant Kernel, and $K_0: U^2\to \bbR$ is positive definite and H\"older continuous.

\medskip

To extend the result we use the fact (proved in \cite{junnila2019}) that if $L$ is sufficiently regular then $K$ is \textit{locally} star-scale invariant in the sense defined above.
We state this result as a proposition. It can be directly derived from  \cite[Theorem 4.5]{junnila2019}.

\begin{proposition}\label{dapopo}
 If $K$ is a positive definite kernel on $\cD$ that can be written in the form \eqref{fourme} with 
 $L\in H^{s}_{\mathrm{loc}}(\cD^2)$ with $s>d$, then for every $z\in \mathcal D$, there exist $\delta_z>0$ such that the restriction of 
 $K$ to $B(z,\delta_z)$ has an almost star-scale invariant part.
\end{proposition}

To extend Theorem \ref{finalfrontier} we require another technical result, which states that 
with the same assumption as above, and $U$ an open set whose closure is included in $\cD$, $K$ can be approximated by a kernel with an almost star-scale invariant part defined on $U$.
This is the content of the following result, \cite[Lemma 2.1]{MR4413209}

\begin{proposition}\label{werp}
 Given $K$ a covariance kernel on $\cD$ of the form \eqref{fourme} with $L\in H^s_{\mathrm{loc}}(\cD^2)$ for $s>d$
 , $U$ a bounded open set whose closure satisfies   $\bar U\subset \cD$ and $\delta>0$,
 then there exists a kernel $K^{(\delta)}$  on $U$ satisfying \eqref{4star} such that 
 \begin{itemize}
  \item [(A)]  For all $x,y \in U, \quad   |K^{(\delta)}(x,y)-K(x,y)|\le \delta$.
\item [(B)] $\Delta^{(\delta)}(x,y)=K^{(\delta)}(x,y)-K(x,y)$ is a positive definite kernel on $U$.
 \end{itemize}

\end{proposition}
\begin{rem}
 More precisely, \cite[Lemma 2.1]{MR4413209} states that one can chose $\eta_1=0$ (recall \eqref{iladeuxstar}) for the almost-star scale invariant part of $K^{(\delta)}$, but this refinement is not required for our purpose.
\end{rem}

\subsection{The case of  Theorem \ref{mainall}}\label{looka}

The extension of the result to the case of a general $\log$-correlated field defined on a domain $\cD$ is the following.

\begin{theorem}\label{mainallall}
 If $X$ is a centered Gaussian field defined on $\cD$ whose covariance kernel $K$  can be written in
 the form \eqref{fourme} with 
 $L\in H^{s}_{\mathrm{loc}}(\cD^2)$ for  $s>d$,  and $f\in C_c(\bbR^d)$,
then there exists a complex valued random variable $M^{\gamma}_{\infty}(f)$ such that for any choice of mollifier $\theta$ the following convergence holds in $L^p$ if $p\in\left[1,\sqrt{2d}/\alpha\right)$.
\begin{equation}
\lim_{\gep \to 0} M^{\gamma}_{\gep}(f)  
= M^{\gamma}_\infty(f).
 \end{equation}
 \end{theorem}

\begin{proof}
This follows quite immediately via a localization argument using a partition of unity.
Let $f\in C_c(\cD)$ be fixed. Using Proposition \ref{dapopo}, we can cover the support of $f$ (which is compact) by finitely many Euclidean balls $B(z_i,\gep_i)$, $i\in I$  such that for every $i\in I$ the restriction of $K$ to $B(z_i,\gep_i)$ has an almost star-scale invariant part.
Using a partition of the unity, we can write $f:=\sum_{i\in I} f_i$ 
where $f_i$ is continuous with compact support included in $B(z_i,\gep_i)$.
Using Theorem \ref{mainall} for $K$ restricted to $B(z_i,\gep_i)$, we obtain that $M^{\gamma}_{\gep}(f_i)$ converges in $L^p$ for every $f_i$ and thus we obtain the convergence for  $M^{\gamma}_{\gep}(f)= \sum_{i\in I}M^{\gamma}_{\gep}(f_i)$.
\end{proof}

\subsection{The case of Theorem \ref{finalfrontier}}

To extend the result for $\gamma\in \cP'_{\mathrm{II/III}}$ it is sufficient to extend Proposition \ref{main}. In the statement below, we implicitely use the fact that  the critical multiplicative chaos $M'$ is well defined under our assumptions (see \cite[Theorem C.2]{critix} for a proof).

 \begin{proposition}\label{mainaaa}
 If $X$ is a centered Gaussian field defined on $\cD$ whose covariance kernel $K$  can be written in
 the form \eqref{fourme} with 
 $L\in H^{s}_{\mathrm{loc}}(\cD^2)$ for  $s>d$, given $\rho, f\in C_c(\cD)$, $\go\in[0,2\pi)$, we have
\begin{equation}\label{vfssaaaa}
 \lim_{\gep \to 0} \bbE \left[e^{i \langle X,\rho \rangle + i \frac{M^{\gamma}_{\gep}(f,\go)}{v(\gep,\theta,\gamma)}} \right]=
 \bbE \left[ e^{i \langle X,\rho\rangle -\frac{1}{2}M'(e^{|\gamma|^2L}|f|^2)} \right].
 \end{equation}
 
\end{proposition}

\begin{proof}

Given a fixed $f\in C_c(\cD)$, and $n\ge 1$, we chose $U$ which contains the support of $f$ and $K_n: U^2\to (-\infty,\infty]$ satisfying the assumptions of $K^{(\delta)}$ of  Proposition \ref{werp} with $\delta=1/n$. We let $Z_n$ be a centered Gaussian field indexed by $U$, independent of $X$ and with covariance  $\Delta_n=K_n-K$ and define $X_n$ a field indexed by $C_c(U)$ by setting $X_n=X+Z_n$. Note that $X_n$ has covariance $K_n$.
We let $M^{\gamma,n}_{\gep}$ and $M^{'}_n$ denote the mollified GMC and critical GMC associated with $X_n$. For simplicity, we define all the $(Z_n)_{n\ge 1}$ on the same probability space: the fields $Z_n$ form an independent sequence which is independent of $X$.  We let $\bbP$ denote the corresponding probability.
From Proposition \ref{main} we have for each $n\ge 1$
\begin{equation}
  \lim_{\gep \to 0} \bbE \left[e^{i \langle X_n,\rho \rangle + i \frac{M^{\gamma,n}_{\gep}(f,\go)}{v(\gep,\theta,\gamma)}} \right]=
 \bbE \left[ e^{i \langle X_n,\rho\rangle -\frac{1}{2}M^{'}_n(e^{|\gamma|^2L_n}|f|^2)} \right],
\end{equation}
where $L_n:= L+\Delta_n$.
In order to conclude, we need to show that (for any choice of $K^{(\delta)}$)
\begin{equation}
 \lim_{n\to \infty}\sup_{\gep\in (0,1)}\left|\bbE \left[e^{i \langle X_n,\rho \rangle + i \frac{M^{\gamma,n}_{\gep}(f,\go)}{v(\gep,\theta,\gamma)}} \right]-\bbE \left[e^{i \langle X,\rho \rangle + i \frac{M^{\gamma}_{\gep}(f,\go)}{v(\gep,\theta,\gamma)}} \right]\right|=0
\end{equation}
and that 
\begin{equation}
  \lim_{n\to \infty} \bbE \left[ e^{i \langle X_n,\rho\rangle -\frac{1}{2}M^{'}_n(e^{|\gamma|^2L_n}|f|^2)} \right]=\bbE \left[ e^{i \langle X,\rho\rangle -\frac{1}{2}M^{'}(e^{|\gamma|^2L^{(\delta)}}|f|^2)} \right].
\end{equation}
Note that it is sufficient to show that the difference between the terms in the l.h.s.\ and the r.h.s.\ tend to zero in probability (uniformly in $\gep$) that is 
\begin{equation}\begin{split}\label{inproba}
&\quad \quad \quad \lim_{n\to \infty}\bbE\left[  |\langle X_n,\rho\rangle-\langle X,\rho \rangle|\vee 1\right]=0,\\
&\lim_{n\to \infty}\bbE\left[  \left|M^{'}_n(e^{|\gamma|^2L_n}|f|^2) - M^{'}(e^{|\gamma|^2L}|f|^2)\right| \vee 1\right]=0,\\
 &\lim_{n\to \infty}\sup_{\gep\in (0,1)}\bbE\left[ \left| \frac{M^{\gamma,n}_{\gep}(f,\go)}{v(\gep,\theta,\gamma)}- \frac{M^{\gamma}_{\gep}(f,\go)}{v(\gep,\theta,\gamma)}\right|\vee 1\right]=0.
\end{split}\end{equation}
The first line is immediate via the computation of the $L^2$ norm (the convergence holds in $L^2$).
For the second line, we set $g_n= e^{|\gamma|^2L_n}|f|^2$ and $g= e^{|\gamma|^2L}|f|^2$.
Using the notational convention introduced in Section \ref{martindecoco}, we let $Z_{n,\gep}$ denote the mollification of $Z_n$ and $\Delta_{n,\gep}$ its covariance.
Letting $e^{\sqrt{2d} Z_{n,\gep}-d \Delta_{n,\gep}}$ denote the function $x \mapsto e^{\sqrt{2d} Z_{n,\gep}(x)-d \Delta_{n,\gep}(x)}$ we have 
\begin{multline}
 \bbE\left[\left(M^{\sqrt{2d},n}_{\gep}(g_n) - M^{\sqrt{2d}}_{\gep}(g)\right)^2 \ | \ X \right]
 =\bbE\left[ M^{\sqrt{2d}}_{\gep}(e^{\sqrt{2d} Z_{n,\gep}-d \Delta_{n,\gep}}g_n -g)^2  \ | \ X \right]
 \\= \int_{U^2} \left(e^{2d\Delta_{n,\gep}(x,y)}g_n(x)g_n(y)-2g_n(x)g(y)+g(x)g(y)\right)     M^{\sqrt{2d}}_{\gep}(\dd x) M^{\sqrt{2d}}_{\gep}(\dd y).
\end{multline}
From the assumption that $|\Delta_n(x,y)| \le 1/n$ (and thus $|L_n(x)-L(x)|\le 1/n$) we obtain that
$$ |e^{2d\Delta_{n,\gep}(x,y)}g_n(x)g_n(y)-2g_n(x)g(y)+g(x)g(y)| \le \frac{C g(x)g(y)}{n} $$
 and hence 
\begin{equation}
 \bbE\left[\left(M^{\sqrt{2d},n}_{\gep}(g_n) - M^{\sqrt{2d}}_{\gep}(g)\right)^2 \ | \ X \right]
 \le \frac{C}{n} M^{\sqrt{2d}}_{\gep}(g)^2
\end{equation}
Using Fatou after renormalization we obtain that 
\begin{equation}
 \bbE\left[\left(M'_{n}(g_n) - M'(g)\right)^2 \ | \ X \right]
 \le \frac{C}{n} M'(g)^2
\end{equation}
which implies the second line in \eqref{inproba}.
For the third line, we are going to Proposition \eqref{dapopo}. More precisely, we use a decomposition of $f=\sum_{i\in I}f_i$ where $f_i$ is continous with compact support included in $U_i$ and
the restriction of $K$ to $U_i$ has an almost star-scale invariant part.
We are going to prove that for each $i\in I$.
\begin{equation}
\lim_{n\to \infty}\sup_{\gep\in (0,1)}\bbE\left[ \left| \frac{M^{\gamma,n}_{\gep}(f_i,\go)}{v(\gep,\theta,\gamma)}- \frac{M^{\gamma}_{\gep}(f_i,\go)}{v(\gep,\theta,\gamma)}\right|\vee 1\right]=0.
\end{equation}
This operation shows that it is in fact sufficient to prove the third line of \eqref{inproba} assuming that $K$ is an almost star-scale invariant Kernel. We can thus further our probability space contains $(X_t)_{t\ge 0}$ a martingale sequence of fields with covariance $K_t$ (we adopt the notation of Section \ref{martindecoco}) approximating $X$.
We equip our space with the filtration
$$ \mathcal G_t:=\sigma( (X_s)_{s\in[0,t]}, (Z_n)_{n\ge 1}).$$
We recall the definition of $T_q(x)$ in \eqref{deftqx} and define
\begin{equation}\begin{split}
W^{(n,\gep)}_t:= \int_{U} (e^{\gamma Z_{n,\gep}(x)-\frac{\gamma^2}{2}\Delta_{n,\gep}(x) }-1) f(x)
e^{\gamma X_{t\wedge T_q(x),\gep}-\frac{\gamma^2}{2}K_{t\wedge T_q,\gep}(x) }\dd x
 \end{split}\end{equation}
Setting 
$$\cA_q:= \{ \forall x\in \Supp(f),\forall t>0, \ \bar X_t(x)\le \sqrt{2d}t+q\}$$
 We have from the definition
$$W^{(n,\gep)}_\infty\ind_{\cA_{q}}=|M^{\gamma,n}_{\gep}(f)- M^{\gamma}_{\gep}(f)|\ind_{\cA_{q}}$$
Hence we have (for any $\go\in [0,2\pi)$ since the projection on one axis reduces the modulus
\begin{equation}
 \bbE\left[|M^{\gamma,n}_{\gep}(f,\go)- M^{\gamma}_{\gep}(f,\go)|^2 \ind_{\cA_{q}}\right]
 \le \bbE[|W^{(n,\gep)}_\infty|^2]= \bbE[|W^{(n,\gep)}_0|^2]+\bbE[ \langle W^{(n,\gep)} \rangle_\infty].
\end{equation}
Since from Lemma \ref{leventkile} we have $\lim_{q\to \infty} \bbP[\cA_q]=1$, to prove the third line in \eqref{inproba}, it is sufficient to show that for any $q$ we have 
\begin{equation}\label{lazt}
 \lim_{n\to \infty} \sup_{\gep\in(0,1)} v(\gep,\theta,\gamma)^{-2}\left(\bbE[|W^{(n,\gep)}_0|^2]+\bbE[ \langle W^{(n,\gep)} \rangle_\infty]\right).
\end{equation}
For the first term, we have 
\begin{equation}
 \bbE[|W^{(n,\gep)}_0|^2]=\int_{U^2}f(x)\bar f(y)(e^{|\gamma|^2 \Delta_{n,\gep}(x,y)-1})e^{|\gamma|^2 K_{0,\gep}(x,y)}\dd x \dd y\le C n^{-1}
\end{equation}
where the inequality obtained taking the modulus of the integrand and using the fact that $|\Delta_n(x,y)|\le 1/n$ and the other terms are uniformly bounded.
The derivative of the bracket of $W^{(n,\gep)}$ is given by $|\gamma|^2$ times (recall that by \eqref{lezaq} we have 
$\cA_{t,q}(x)=\{T_q(x)\le t\}$)
\begin{multline}
 D_t:=\int_{\bbR^{2d}} Q_{t,\gep}(x,y)G_{n,\gep}(x)\bar G_{n,\gep}(y) \\
 \times e^{\gamma X_{t,\gep}+\bar \gamma X_{t,\gep}(y)-\frac{\gamma^2}{2}K_{t,\gep}(x)-\frac{\bar \gamma^2}{2} K_{t,\gep}(y) } \ind_{\cA_{t,q}(x)\cap \cA_{t,q}(y)}\dd x \dd y.
\end{multline}
with $G_{n,\gep}(x)=(e^{\gamma Z_{n,\gep}(x)-\frac{\gamma^2}{2}\Delta_{n,\gep}(x) }-1)$.
Repeating once more the computation in \eqref{onlefaitunefois}
we obtain that
\begin{equation}
 D_t\le \int_{\bbR^{2d}} Q_{t,\gep}(x,y)|G_{n,\gep}(x)|^2 
e^{\sqrt{2d} X_{t,\gep}+(|\gamma|^2-d)K_{t,\gep}(x) } \ind_{\cA_{t,q}(x)}\dd x \dd y.
\end{equation}

We define a martingale $\bar W^{(\gep)}$ and $\bar W^{(\gep)}_t$ by setting 
\begin{equation}
W^{(\gep)}_t:= \bbE\left[ M^{\gamma,n}_{\gep}(f,\go)- M^{\gamma}_{\gep}(f,\go) \ | \ \cG_t \right]
\end{equation}
Now we have 
$$ \bbE\left[|G_{n,\gep}(x)|^2 \right]=e^{|\gamma|^2 \Delta_{n,\gep}(x)}-1\le C n^{-1} $$
This the term is independent of the rest, thus using \eqref{topz} we obtain that 
\begin{equation}
  \bbE[D_t]\le C n^{-1} \int_{\bbR^{2d}} Q_{t,\gep}(x,y) e^{|\gamma|^2 K_{t,\gep}(x)} \dd x \dd y
  \le C' n^{-1} \phi(t,\gep).
\end{equation}
Integrating against $t$ we conclude that 
$$\bbE[\langle W^{(n,\gep)} \rangle_\infty]\le C n^{-1}v(\gep,\theta,\gamma)^2 $$  
and this concludes the proof of \eqref{lazt}.
\end{proof}

\section{Proof of Lemma \ref{teknikos}}\label{taikos}

We use Kahane convexity inequality in order to compare $B_n$ to the the partition function of a Gaussian branching random walk (or polymer on a $2^d$-adic tree).
We assume without loss of generality that $\Supp(f)\subset [0,1]^d$. 
For $x,y\in [0,1]^d$ we let $2^{-k(x,y)}$ be the sidelength of the smallest dyadic cube that contains $x$ and $y$. 
$$k(x,y):=  \inf\left\{n\ge 0\ : \ \exists {\bf m}\in \lint 0,2^{n}-1\rint^d, \  \{x,y\}\subset \left( 2^{-n} {\bf m}+[0,2^{-n})^d\right)  \right\}.$$
and set $k_n(x,y):= k(x,y)\wedge n.$
Note that $k_n$ defines a positive definite function and that $k(x,y)\le \log_2\left( \frac 1 {|x-y|}\right)+C$. 
Hence from Lemma \ref{petipanda} there exists a constant $A>0$ such that
\begin{equation}\label{compapa}
(\log 2) k_n(x,y)\le K_{\lceil n \log 2 \rceil}(x,y)+A.
\end{equation}
Using Kahane's convexity inequality (proved in \cite{zbMATH03960673} see also \cite[Theorem 2.1]{RVreview}) which we introduce in a simplified setup 
\begin{lemma}\label{kahaha}
 If $C_1$ and $C_2$ are two bounded positive definite kernel on an arbitrary space $\cX$ satisfying 
 $$ \forall x,y \in \cX, \quad C_1(x,y)\le C_2(x,y) $$
 $\mu$ is a finite measure on $[0,1]^d$ and $F: \bbR_+\to \bbR$ is a concave function with at most polynomial growth at infinity
and $Y_1$ and $Y_2$ are Gaussian fields with respective covariance $C_1$ and $C_2$ then we have for any $\theta\in \bbR$
\begin{equation}
\bbE\left[ F \left(\int e^{\theta Y_1(x)-\frac{\theta^2}{2}C_1(x)} \mu(\dd x) \right)\right] \le  \bbE\left[ F\left(\int e^{Y_2(x)-\frac{\theta^2}{2}C_2(x)} \mu(\dd x) \right)\right].
\end{equation}

\end{lemma}
Hence if $Z_n$ denotes a field defined on $[0,1]^d$ with covariance $k_n$ we can apply Lemma \ref{kahaha} result for the fields 
$\sqrt{\log 2} Z_n$ and $X_{\lceil n\log 2 \rceil}+\sqrt{A} \mathcal N$ where $\cN$ is an independent standard Gaussian (the fields have their resepective covariances given by the two sides of Equation \eqref{compapa}), $\mu(\dd x)=|f(x)|^2 \dd x$ and $F(u)=u^{p/2} $. Recalling \eqref{chuppa} we have 
 \begin{equation}
  \bbE\left[ B^{p/2}_{\lceil n\log 2 \rceil}\right] \le C  \bbE\left[\left( 2^{dn} \int_{[0,1]^d}|f(x)|^2 e^{2\alpha\sqrt{\log 2}(Z_n(x)-\sqrt{2d\log 2}n)}\dd x\right)^{p/2}\right].
 \end{equation}
 The constant $C$ above takes care of the fact that the variance of $\sqrt{\log 2}Z_n(x)$ and $X_{\lceil n\log 2 \rceil}$ differ by 
a $O(1)$ term, and also of the moment of the variable $\mathcal N$.
We can ignore the constant $f$ at the cost of a prefactor $\|f\|^p_{\infty}$. To conclude we thus need a bound on the moment of order $p/2$  of the  partition function of the Gaussian branching random walk 
\begin{equation}\label{alsosum}
 W_{n,\zeta}:= 2^{dn}\int_{[0,1]^d}e^{\zeta(Z_n(x)-\sqrt{2d\log 2}n)}\dd x
 =\sum_{{\bf m}\in \lint 0,2^{n}-1\rint^d}e^{\zeta(Z_n({\bf m}2^{-n})-\sqrt{2d\log 2}n)},
\end{equation}
for $\zeta=2\alpha\sqrt{\log 2}$.
The following result is a particular case of \cite[Theorem 1.6]{Hushi}. We present a shorter proof which is valid in our context for the sake of completeness.

\begin{lemma}\label{lemmadtwo}
 Given $\zeta>\sqrt{2d\log 2}$ and $q\le  \frac{\sqrt{2d\log 2}}{\zeta}$ there exists positive constant $C$ and $b$ such that 
$$ \bbE\left[ (W_{n,\zeta})^{q}\right] \le  C n^{-\frac{3q\zeta}{2\sqrt{2d\log 2}}}(\log n)^6$$
\end{lemma}

\begin{proof}
 We split our integral in three parts.
 We set 
 \begin{equation}\begin{split}
  \mathcal B_n(x)&:=\{ \exists  m\in \lint 1,n\rint,  \ Z_m(x)\ge \sqrt{2d\log 2}+(\log n)^2   \},\\
    \mathcal C_n(x)&:=\mathcal B^{\cc}_n(x)\cap \{ Z_n(x)\le \sqrt{2d\log 2}n - (\log n)^2   \},\\
     \mathcal A_n(x)&:= \mathcal B^{\cc}_n(x)\cap   \mathcal C^{\cc}_n(x) 
    \end{split}
 \end{equation}
 We define   $W_{n,\zeta}(\cA)$, $W_{n,\zeta}(\cB)$ and $W_{n,\zeta}(\cC)$ by setting, for $\cI\in \{\cA,\cB,\cC\}$
\begin{equation}
 W_{n,\zeta}(\cI):= 2^{dn} \int_{[0,1]^d}e^{\zeta(Z_n(x)-\sqrt{2d\log 2}n)}\ind_{\cI_n(x)}
\dd x
\end{equation}
Using subadditivity \eqref{subadd} we have
\begin{equation}
 \bbE\left[ (W_{n,\zeta})^{q}\right] \le  \bbE\left[W_{n,\zeta}(\cA)^q\right]+ \bbE\left[W_{n,\zeta}(\cB)^q\right]+\bbE\left[W_{n,\zeta}(\cC)^q\right].
\end{equation}
We are going to show that the two last terms in the r.h.s.\ decay faster than any negative power of $n$ and then prove a bound of the right order of magnitude for $ \bbE\left[ (W_{n,\zeta})^{q}\right]$.
 Letting setting $\cB_n:= \bigcup_{x\in [0,1]} \cB_n(x),$  and $q'= \sqrt{2d \log 2} \zeta^{-1}$  ($q'\in[q,1)$) we have 
 \begin{equation}\label{produkt}
 \bbE\left[ W_{n,\zeta}(\cB)^q\right]\le  \bbE\left[ (W_{n,\zeta})^q\ind_{\cB_n} \right]\le \bbE\left[ (W_{n,\zeta})^{q'}\right]^{\frac{q}{q'}} \bbP\left[ \cB_n \right]^{1-\frac{q}{q'}}.
 \end{equation}
Using subadditivity \eqref{subadd} for the sum \eqref{alsosum} with $\theta=q'$, 
 \begin{equation}
 \bbE\left[(W_{n,\zeta})^{q'}\right] \le \bbE\left[ W_{n,\sqrt{2d\log 2}}\right]=1.
\end{equation}
The inequality on the right comes from the fact that  $(W_{m,\sqrt{2d\log 2}})_{m\ge 1}$ is a martingale for the natural filtration associated with $Z_n$.
Using the optional stopping Theorem for this same martingale, we can obtain a bound for the probability of $\mathcal B_n$,
\begin{equation}
 \bbP[\cB_n]\le \bbP\left[ \exists m,\  W_{m,\sqrt{2d\log 2}} \ge e^{\sqrt{2\log 2}(\log n)^2} \right]\le e^{-\sqrt{2\log 2}(\log n)^2}. 
\end{equation}
This yields a subpolynomial decay for $\bbE[W_{n,\zeta}(\cB)^q]$.
For $W_{n,\zeta}(\cC)$
using the fact that $Z_n$ is a Gaussian of variance $n$, we obtain using Jensen's inequality, the Cameron-Martin formula and Gaussian tail bounds
  \begin{multline}
 \bbE\left[(W_{n,\zeta}(\cC))^q\right]^{1/q}\le  \bbE\left[W_{n,\zeta}(\cC)\right]=  2^{dn} \bbE\left[ e^{q(Z_n-\sqrt{2d\log 2}n)} \ind_{\{ Z_n\le \sqrt{2d\log 2}n - (\log n)^2\}} \right]\\
 =   e^{ \left( d\log 2+\frac{ \zeta^2}{2}- \zeta\sqrt{2d\log 2}\right)n} \bbP\left[ Z_n(0)\le (\sqrt{2d\log 2}- \zeta)n - (\log n)^2\right]
\le e^{(\sqrt{2d\log 2}-\zeta )(\log n)^2},
\end{multline}
also proving a subpolynomial decay.
 It remains to estimate the main part  $\bbE\left[(W_{n,\zeta}(\cA))^q\right]$. Using first subaddivity \eqref{subadd} and then Jensen's inequality
 \begin{equation}
   \bbE\left[(W_{n,\zeta}(\cA))^q\right]\le \bbE\left[W_{n,\sqrt{2d\log 2}}(\cA)^{\frac{q\zeta}{\sqrt{2d\log 2}}}\right]\le    \bbE\left[W_{n,\sqrt{2d\log 2}}(\cA)\right]^{\frac{q\zeta}{\sqrt{2d\log 2}}}.
 \end{equation}
The  Cameron-Martin formula directly expresses $\bbE\left[W_{n,\sqrt{2d\log 2}}(\cA)\right]$ as the probability concerning the Gaussian centered random walk $(Z_m(0))_{m\ge 0}$, 
\begin{multline}
  \bbE\left[W_{n,\sqrt{2d\log 2}}(\cA)\right]= 
 \bbP\left[ \forall m\in \lint 1,n\rint,\ Z_m(0) \le (\log n)^2 \ ; \ Z_n(0)\ge -(\log n)^2 \right] \\ 
 \le C n^{-3/2}(\log n)^6.
\end{multline}
The bound for the probability of the event above is valid for any random walk with IID centered increments with finite second moment (see for instance \cite[Lemma A.3]{Aishi})  which concludes our proof.

\end{proof}

\bibliographystyle{plain}
\bibliography{bibliography.bib}

\end{document}